\documentclass{amsart}
\usepackage[margin=3cm]{geometry}

\usepackage{amssymb,amsmath,amsfonts,mathtools,latexsym}
\usepackage[pdftex]{hyperref} 
\newtheorem{theorem}{Theorem}[section]
\newtheorem{corollary}[theorem]{Corollary}
\newtheorem{definition}[theorem]{Definition}
\newtheorem{example}[theorem]{Example}
\newtheorem{lemma}[theorem]{Lemma}

\newtheorem{proposition}[theorem]{Proposition}
\newtheorem{remark}[theorem]{Remark}
\usepackage{comment} 
\date{February 07, 2024}

\begin{document}
\title[$\mbox{Irreducible Graded Bimodules over Algebras and a Pierce Decomposition of the Jacobson Radical}$]{Irreducible Graded Bimodules over Algebras and a Pierce Decomposition of the Jacobson Radical}

\author[De Fran\c{c}a]{Antonio de Fran\c{c}a$^\dag$}
\address{Unidade Acad\^emica de Matem\'{a}tica, Universidade Federal de Campina Grande, Av. Apr\'igio Veloso, 785, Bodocong\'o, CEP 58429-970, Campina Grande, Para\'iba, Brasil}
\email{\href{mailto: a.defranca@yandex.com}{a.defranca@yandex.com}}
\thanks{$^\dag$The author was partially supported by Para\'iba State Research Foundation (FAPESQ), Grant \#2023/2158.}

\author[Sviridova]{Irina Sviridova}
\address{Departamento de Matem\'{a}tica, Universidade de Bras\'{i}lia, Campus Universit\'ario Darcy Ribeiro, ICC, Ala Centro, Bloco A, CEP 70910-900, Bras\'{i}lia, Distrito Federal, Brasil}
\email{\href{mailto: I.Sviridova@mat.unb.br}{I.Sviridova@mat.unb.br}}




\begin{abstract}
It is well known that the ring radical theory can be approached via language of modules. In this work, we present some generalizations of classical results from module theory, in the two-sided and graded sense. Let $\mathsf{G}$ be a group, $\mathbb{F}$ an algebraically closed field with $\mathsf{char}(\mathbb{F})=0$, $\mathfrak{A}$ a finite dimensional $\mathsf{G}$-graded associative $\mathbb{F}$-algebra and $\mathsf{M}$ a $\mathsf{G}$-graded unitary $\mathfrak{A}$-bimodule. We proved that if $\mathfrak{A}=M_n(\mathbb{F}^\sigma[\mathsf{H}])$ with a canonical elementary $\mathsf{G}$-grading, where $\mathsf{H}$ is a finite abelian subgroup of $\mathsf{G}$ and $\sigma\in\mathsf{Z}^2(\mathsf{H},\mathbb{F}^*)$, then $\mathsf{M}$ being irreducible graded implies that there exists a nonzero homogeneous element $w\in\mathsf{M}$ satisfying $\mathsf{M}=\mathfrak{B}w$ and $\mathfrak{B} w= w\mathfrak{B}$. Another result we proved generalizes the last one: if $\mathsf{G}$ is abelian, $\mathfrak{A}$ is simple graded and $\mathsf{M}$ is finitely generated, then there exist nonzero homogeneous elements $w_1, w_2,\dots,w_n\in\mathsf{M}$ such that 
	\begin{equation}\nonumber
\mathsf{M}=\mathfrak{A} w_1\oplus\mathfrak{A} w_2\oplus \cdots \oplus \mathfrak{A} w_n \ , 
	\end{equation}
where $w_i \mathfrak{A}=\mathfrak{A} w_i\neq0$ for all $i=1, 2,\dots,n$, and each $\mathfrak{A} w_i$ is irreducible. The elements $w_i$'s are associated with the irreducible characters of $\mathsf{G}$. We also describe graded bimodules over graded semisimple algebras. And we finish by presenting a Pierce decomposition of the graded Jacobson radical of any finite dimensional $\mathbb{F}$-algebra with a $\mathsf{G}$-grading.
\end{abstract}


\maketitle

\section{Introduction}


Amitsur started in \cite{Amit52} the construction of the general theory of radicals. As explained by Kurosch in his work \cite{Kuro53}, the concept of \textit{radical} has great importance in the study of associative algebras and rings. The study of radicals constitutes an important tool for the development of Ring Theory. In \cite{Amit54.1}, Amitsur presents an axiomatic foundation for radicals of rings. From this, Andrunakievi\v{c} and Rjabuhin showed in \cite{AndrRjab64} that the general theory of radicals of associative rings developed by Kurosh and Amitsur (and others) may be presented in the language of modules, and consequently in the language of representations.

In \cite{Jaco51}, Jacobson initiated the study of bimodules\footnote{An $(\mathfrak{A},\widetilde{\mathfrak{A}})$-bimodule $\mathsf{M}$ is a left $\mathfrak{A}$-module and a right $\widetilde{\mathfrak{A}}$-module that satisfies $a(mb)=(am)b$ for any $a\in\mathfrak{A}$, $b\in\widetilde{\mathfrak{A}}$ and $m\in\mathsf{M}$.} of Jordan algebras, and in \cite{Jaco54}, the author presented a study about the structure of alternative and Jordan bimodules. In turn, Foster, in \cite{Fost73}, presented a study showing that the general theory of radicals of associative rings can be naturally extended to varieties of algebras where the modules are replaced by bimodules. Already in \cite{ShesTrus16}, Shestakov and Trushina described the irreducible birepresentations of alternative\footnote{An algebra $\mathfrak{A}$ is called an \textit{alternative algebra} if satisfies $(xx)y=x(xy)$ and $(xy)y=x(yy)$ for any $x,y\in\mathfrak{A}$.} algebras and superalgebras. For more details about \textit{radicals}, \textit{(bi)modules} and \textit{representations (of groups)}, see \cite{CurtRein62,GardWieg03,Hers05,MarkWieg93,Wisb96}.

Let $\mathsf{R}$ and $\mathsf{S}$ be two rings, and $\mathsf{M}$ an $(\mathsf{R,S})$-bimodule. Remember that any ring is a $\mathbb{Z}$-bimodule. In \cite{BrowSmit85}, Brown and Smith examined what properties of $\mathsf{M}$ are guaranteed when $\mathsf{R}$ and $\mathsf{S}$ are noetherian and $\mathsf{M}$ is finitely generated as a left $\mathsf{R}$-module and as a right $\mathsf{S}$-module. The study of these noetherian bimodules played an important role in advancing in noetherian ring theory.

Giambruno and Zaicev, on the other hand, in \cite{GiamZaic03}, needed to describe the Jacobson radical $\mathsf{J}$ of an algebra $\mathfrak{A}=\mathfrak{B}\oplus\mathsf{J}$, which is a $\mathfrak{B}$-bimodule, to prove their main results in the work, where $\mathfrak{B}=M_n(\mathbb{F})$ is the algebra of $n\times n$ matrices over a field $\mathbb{F}$. According to Wisbauer (\cite{Wisb96}), one of the most important techniques in the theory of commutative associative algebras, and which can be successfully transferred to non-commutative associative algebras via one-sided modules, is the homological characterization of an algebra $\mathfrak{A}$ in the category of (one-sided) $\mathfrak{A}$-modules. Once the study of module theory has been used to describe the structure of algebras, Wisbauer devoted almost 400 pages of his book \cite{Wisb96} to presenting a compilation of the theory.

%

The problems exposed above show the importance of studying modules and radicals, and that famous mathematicians and algebraists have studied -- especially the properties of these structures -- these topics in the last decades. In this sense, we study in this work algebraic objects known as \textit{bimodules}. More specifically, here we study \textit{bimodules over algebras graded by groups}. Our results imply in a characterization of graded Jacobson radical of a suitable family of algebras.

Algebraic structures with gradings by groups ensure a very rich field of research, in ring theory, because from a structure of grading we can deduce properties of the ordinary object (i.e. object without grading). In \cite{BahtSehgZaic08}, Theorem 3, Bahturin, Sehgal and Zaicev showed that, when $\mathsf{G}$ is a group and $\mathbb{F}$ is an algebraically closed field with $\mathsf{char}(\mathbb{F})=0$, ``\textit{any finite-dimensional $\mathsf{G}$-graded $\mathbb{F}$-algebra $\mathfrak{A}$ is simple graded iff $\mathfrak{A}$ is isomorphic to $M_{k}(\mathbb{F}^{\sigma}[\mathsf{H}])$, where $\mathsf{H}$ is a finite subgroup of $\mathsf{G}$ and $\sigma$ is a $2$-cocycle on $\mathsf{H}$}''. Based on this lastest work, in \cite{Svir11}, Lemma 2, Sviridova presented a generalization of Wedderburn-Malcev Theorem\footnote{The Wedderburn-Malcev Theorem (see \cite{CurtRein62}, Theorem 72.19, p.491), which is a generalization made by Malcev, in \cite{Malc42}, of one of Wedderburn's Theorem, states that for any finite dimensional algebra over a field $\mathbb{F}$ such that $\mathfrak{A}/\mathsf{J}(\mathfrak{A})$ is separable, there exists a unique (up to the isomorphism) maximal semisimple subalgebra $\mathfrak{B}$ of $\mathfrak{A}$ satisfying $\mathfrak{A}=\mathfrak{B}\oplus\mathsf{J}(\mathfrak{A})$.} for algebras graded by a (finite abelian) group. She established that, in suitable conditions, any finite dimensional graded algebra $\mathfrak{A}$ is a direct sum of a maximal semisimple graded subalgebra $M_{k_1}(\mathbb{F}^{\sigma_1}[\mathsf{H}_1]) \times \cdots \times M_{k_p}(\mathbb{F}^{\sigma_p}[\mathsf{H}_p])$, where $\mathsf{H}_1, \dots,\mathsf{H}_p$ are finite subgroups of $\mathsf{G}$, and its Jacobson radical $ \mathsf{J}(\mathfrak{A})$, which is graded. Note that $\mathsf{J}(\mathfrak{A})$ is naturally a graded $\mathfrak{A}$-bimodule, and obviously a graded $\mathfrak{A}/\mathsf{J}(\mathfrak{A})$-bimodule. Therefore, it is important to study graded bimodules over simple graded algebras, and in special the graded bimodules over the matrix algebras $M_{r}(\mathbb{F}^{\sigma}[\mathsf{H}])$'s, since any finite-dimensional simple graded algebras are given by matrix algebras.

In this work, our interest is to study the graded bimodules and its irreducible graded subbimodules, as well as to apply this study to describe the graded Jacobson radical of a graded algebra. We studied the irreducible graded subbimodules of a graded bimodule $\mathsf{M}$, as well as the possibles decompositions for $\mathsf{M}$ in sum of irreducible graded objects. Basically, we have answered the following question:

	\vspace{0.15cm}
	\noindent{\bf Problem:} How to describe the irreducible graded bimodules?
	\vspace{0.15cm}

In this way, we have some natural questions: how to characterize the irreducible graded bimodules? When is possible to write a graded bimodule as sum of irreducible graded subbimodules? How do  the considered algebras interfere in the bimodules? Is it possible to present a decomposition of Jacobson radical (of a finite dimensional graded algebra) in irreducible graded subbimodules? If so, how and what are the consequences?

Let $\mathbb{F}$ be a field, $\mathsf{G}$ a group and $\mathfrak{A}$ and $\widetilde{\mathfrak{A}}$ two associative algebras. Consider $\mathsf{M}$ an $(\mathfrak{A},\widetilde{\mathfrak{A}})$-bimodule. Now, consider $\mathsf{G}$-gradings on $\mathfrak{A}$ and $\widetilde{\mathfrak{A}}$, namely, $\mathfrak{A}=\oplus_{g\in\mathsf{G}}\mathfrak{A}_g$ and $\widetilde{\mathfrak{A}}=\oplus_{h\in\mathsf{G}}\widetilde{\mathfrak{A}}_h$, and $\mathsf{M}=\oplus_{t\in\mathsf{G}}\mathsf{M}_t$ a $\mathsf{G}$-grading on $\mathsf{M}$, i.e. the $\mathsf{M}_t$'s are subspaces of $\mathsf{M}$ which satisfy $\mathfrak{A}_g\mathsf{M}_t \subset \mathsf{M}_{gt}$ and $\mathsf{M}_t \widetilde{\mathfrak{A}}_h \subset \mathsf{M}_{th}$ for any $t,g,h\in\mathsf{G}$. To more details, as well as an overview, about graded algebras and rings, as well as general theory of modules, see the works \cite{NastOyst04,NastOyst11,Wisb96} and their references.

Firstly, we record in $\S$\ref{preliminary} definitions and concepts necessary for the better development of this work. The reader familiarized with objects as \textit{$\mathsf{G}$-graded bimodules}, \textit{$\mathsf{G}$-noetherian bimodules}, \textit{$\mathsf{G}$-artinian bimodules}, \textit{2-cocycles on $\mathsf{G}$}, \textit{twisted group algebra} and \textit{group characters}, can continue reading from page \pageref{firstresults},  returning if necessary. Later, we dedicated in $\S$\ref{firstresults} of this work to generalize (in language of the graded bimodules) some classic results of (one-sided) modules theory, as Jordan-H\"{o}lder Theorem, Isomorphisms theorems and Correspondence Theorem. The main theorem of this first results section is given by:

	\vspace{0.15cm}
	\noindent{\bf Theorem \ref{1.17}}: If $\mathfrak{A}$ and $\widetilde{\mathfrak{A}}$ are weak $\mathsf{G}$-artinian and weak $\mathsf{G}$-noetherian, then the following conditions are equivalent: i) $\mathsf{M}$ is finitely generated; ii) $\mathsf{M}$ is weak $\mathsf{G}$-noetherian; iii) $\mathsf{M}$ is weak $\mathsf{G}$-artinian and weak $\mathsf{G}$-noetherian; iv) $\mathsf{M}$ has a graded composition series.
	\vspace{0.15cm}

In the theorem above, the word ``weak'' refers to the chains being of (graded) subbimodules. Other important result of this section, which is presented in the subsection \ref{generlemagiam}, is inspired by Lemma 2 in \cite{GiamZaic03} and it is given by

	\vspace{0.15cm}
	\noindent{\bf Theorem \ref{1.27}}
Suppose $\mathfrak{A}$ and $\widetilde{\mathfrak{A}}$ two unitary algebras. If $\mathsf{M}$ has a finite dimension (as vector space), then $\mathsf{M}$ can be decomposed as 
$
\mathsf{M}=\mathsf{M}_{00}\oplus \mathsf{M}_{10}\oplus \mathsf{M}_{01}\oplus \mathsf{M}_{11},
$ 
where $\mathsf{M}_{ij}$'s are $\mathsf{G}$-graded $(\mathfrak{A},\widetilde{\mathfrak{A}})$-bimodules such that:
	\begin{itemize}
		\item[i)] for $r=0,1$, $\mathsf{M}_{0r}$ is a $0$-left $\mathfrak{A}$-module and $\mathsf{M}_{1r}$ is a unitary left $\mathsf{G}$-graded $\mathfrak{A}$-module;
		\item[ii)] for $s=0,1$, $\mathsf{M}_{s0}$ is a $0$-right $\widetilde{\mathfrak{A}}$-module and $\mathsf{M}_{s1}$ is a unitary right $\mathsf{G}$-graded $\widetilde{\mathfrak{A}}$-module;
		\item[iii)] $\mathsf{M}_{11}$ is a unitary $\mathsf{G}$-graded $(\mathfrak{A},\widetilde{\mathfrak{A}})$-bimodule.
	\end{itemize}
	In addition, if $\mathfrak{A}$ (resp. $\widetilde{\mathfrak{A}}$) is simple graded, then $\mathsf{M}_{1r}$ (resp. $\mathsf{M}_{r1}$) is faithful on the left (resp. on the right), for $i=0,1$.
	\vspace{0.15cm}

In consequence from the previous result, we show in Corollary \ref{1.31} that $\mathsf{M}$ is a unitary $(\mathfrak{A}, \widetilde{\mathfrak{A}})$-bimodule if and only if $\mathsf{M}=\mathsf{M}_{11}$. In $\S$\ref{simplealgebras}, we describe the unitary bimodule $\mathsf{M}$.

%

The main results of this work are in $\S$\ref{simplealgebras}. The aim of this section is to exhibit a description of irreducible graded bimodules and a (special) decomposition of finitely generated graded bimodules. Given $\mathsf{H}$ a finite abelian subgroup of a group $\mathsf{G}$, $\sigma\in\mathsf{Z}^2(\mathsf{H},\mathbb{F}^*)$ a $2$-cocycle on $\mathsf{H}$, $\mathfrak{B}=M_n(\mathbb{F}^\sigma[\mathsf{H}])$ the $n\times n$ matrix algebra over $\mathbb{F}^\sigma[\mathsf{H}]$, and $\mathsf{M}$ a $\mathsf{G}$-graded $\mathfrak{B}$-bimodule, the key object for the most important results in this section is given by elements of the form
	\begin{equation}\label{1.61}
\hat{w}_{\chi}=\sum_{i=1}^n\left(\sum_{h\in \mathsf{H}} \chi(h)\sigma(h,h^{-1})^{-1} \eta_h E_{i1} m_0 E_{1i}\eta_{h^{-1}}\right) \ , 
	\end{equation}
where $\chi: \mathsf{H}\rightarrow \mathbb{F}^*$ is an irreducible character of $\mathsf{H}$, $E_{ij}$'s are the elementary matrices of $M_n(\mathbb{F})$ and $\eta_{h}E_{i1},E_{1i}\eta_{h^{-1}}\in M_n(\mathbb{F}^\sigma[\mathsf{H}])$, and $m_0\in\mathsf{M}$ is a nonzero homogeneous element. Here, $\mathfrak{B}$ is $\mathsf{G}$-graded by a canonical elementary defined by an $n$-tuple $(g_1, \dots, g_n)\in\mathsf{G}^n$ (i.e. $\mathfrak{B}=\bigoplus_{g\in\mathsf{G}}\mathfrak{B}_g$, where $\mathfrak{B}_g=\mathsf{span}_{\mathbb{F}}\{E_{ij}\eta_h\in\mathfrak{B}: g_i^{-1}hg_j=g\}$), and the elements $\hat{w}_{\chi}$'s belong to $\mathsf{M}$ and their product by the basic (homogeneous) elements $E_{ij}\eta_g$'s of $M_n(\mathbb{F}^\sigma[\mathsf{H}])$ is commutative, up to a scalar. The key and more important result of $\S$\ref{simplealgebras} is the 

	\vspace{0.15cm}
	\noindent{\bf Theorem \ref{1.03}}:
Let $\mathsf{G}$ be a group, $\mathsf{H}$ a finite abelian subgroup of $\mathsf{G}$, and $\mathbb{F}$ an algebraically closed field such that $\mathsf{char}(\mathbb{F})=0$. Consider $\mathsf{M}$ a $\mathsf{G}$-graded unitary $\mathfrak{B}$-bimodule, where $\mathfrak{B}=M_n(\mathbb{F}^\sigma[\mathsf{H}])$ with a canonical elementary $\mathsf{G}$-grading and $\sigma\in\mathsf{Z}^2(\mathsf{H},\mathbb{F}^*)$. If $\mathsf{M}$ is irreducible graded, then there exists a nonzero homogeneous element $w\in\mathsf{M}$ satisfying $\mathsf{M}=\mathfrak{B}w$, such that $\mathfrak{B} w= w\mathfrak{B}$.
	\vspace{0.15cm}

The element $w$ defined in the previous theorem is associated with an irreducible character of $\mathsf{H}$, and it is defined similarly to (\ref{1.61}).
%

In the result below, we present a decomposition for the graded bimodules over a simple graded algebra when such bimodules are finitely generated. In its proof we use the elements $\hat{w}_{\chi}$'s. This result is very important, and gives a decomposition into irreducible (subbimodules) similar to the one-sided case of modules. Such decomposition will be used to give a characterization of Jacobson radical of a family of the finite dimensional algebras. 

	\vspace{0.15cm}
	\noindent{\bf Theorem \ref{1.83}}: Let $\mathbb{F}$ be an algebraically closed field, $\mathsf{G}$ an abelian group and $\mathfrak{A}$ a finite dimensional algebra over $\mathbb{F}$ with a $\mathsf{G}$-grading. Suppose that $\mathsf{char}(\mathbb{F})=0$, and $\mathfrak{A}$ is simple graded. If $\mathsf{M}$ is a finitely generated $\mathsf{G}$-graded $\mathfrak{A}$-bimodule, then there exist nonzero homogeneous elements $w_1, w_2,\dots,w_n\in\mathsf{M}$ such that 
	\begin{equation}\nonumber
\mathsf{M}=\mathfrak{A} w_1\oplus \mathfrak{A} w_2\oplus \cdots \oplus \mathfrak{A} w_n
	\end{equation}
where $w_i \mathfrak{A}=\mathfrak{A} w_i\neq0$ for all $i=1,2,\dots,n$, and $\mathfrak{A} w_i$ is irreducible.
	\vspace{0.15cm}

We finalized $\S$\ref{simplealgebras} exhibiting a decomposition of bimodules over weak semisimple algebras in the form of a Pierce decomposition (see  Theorem \ref{1.29}). We prove that ``\textit{if $\mathfrak{A}$ and $\widetilde{\mathfrak{A}}$ are weak semisimple graded, and $\mathsf{M}$ is a unitary $\mathsf{G}$-graded $(\mathfrak{A},\widetilde{\mathfrak{A}})$-bimodule, then there exist central orthogonal idempotent elements $\mathfrak{i}_1,\dots,\mathfrak{i}_p\in\mathfrak{A}$ and $\hat{\mathfrak{i}}_1,\dots,\hat{\mathfrak{i}}_q\in\widetilde{\mathfrak{A}}$ such that $\mathfrak{A}=\bigoplus_{r=1}^p\mathfrak{A}_r$, $\widetilde{\mathfrak{A}}=\bigoplus_{s=1}^q\widetilde{\mathfrak{A}}_s$, and $\mathsf{M}=\bigoplus_{r,s=1}^{p,q} \mathfrak{i}_r \mathsf{M} \hat{\mathfrak{i}}_s$, where $\mathfrak{A}_r=\mathfrak{A}\mathfrak{i}_r$, $\widetilde{\mathfrak{A}}_s=\widetilde{\mathfrak{A}}\hat{\mathfrak{i}}_s$}''. Moreover, we showed that each $\mathfrak{i}_r \mathsf{M} \hat{\mathfrak{i}}_s$ is either equal to $\{0\}$ or  a faithful $\mathsf{G}$-graded $(\mathfrak{A}_r,\widetilde{\mathfrak{A}}_s)$-bimodule.

We finalize this work with one (our two) application of our main results. A big problem in the study of Ring Theory arises when we try to describe the properties of a ring $\mathfrak{R}$, more precisely, its Jacobson radical. In what follows, we exhibit a decomposition of the Jacobson radical of an algebra which can be used as a tool in future studies of the theory. For more details about Jacobson radical, we recommend the books \cite{Hers05,Jaco64,Rotm10}.

Let $\mathsf{G}$ be a group and $\mathfrak{A}=\mathfrak{B}\oplus\mathsf{J}$ a finite dimensional algebra with a $\mathsf{G}$-grading, where $\mathfrak{B}$ is a maximal semisimple $\mathsf{G}$-graded subalgebra of $\mathfrak{A}$, and $\mathsf{J}=\mathsf{J}(\mathfrak{A})$ is the Jacobson radical and a graded ideal of $\mathfrak{A}$. 
Besides of the decomposition guaranteed by Theorem \ref{1.27}, the Theorem \ref{1.30} states that if $\mathfrak{B}=M_{k_1}(\mathbb{F}^{\sigma_1}[\mathsf{H}_1]) \times \cdots \times M_{k_p}(\mathbb{F}^{\sigma_p}[\mathsf{H}_p])$, $\mathsf{G}$ is finite abelian, and $\mathbb{F}$ is an algebraically closed field with $\mathsf{char}(\mathbb{F})=0$, then
	\begin{itemize} 
		\item[a)] $\mathsf{J}_{11}=\bigoplus_{s,r=1}^p \mathfrak{i}_r \mathsf{J}_{11} \mathfrak{i}_s$, where each $\mathfrak{i}_r \mathsf{J}_{11} \mathfrak{i}_s$ is a $\mathsf{G}$-graded $(\mathfrak{B}_r,\mathfrak{B}_s)$-bimodule, where $\mathfrak{B}=\bigoplus_{i=1}^p\mathfrak{B}_i$ with $\mathfrak{B}_i=M_{k_i}(\mathbb{F}^{\sigma_i}[\mathsf{H}_i])$. In addition, $\mathfrak{i}_r \mathsf{J}_{11} \mathfrak{i}_s\neq0$ implies that $\mathfrak{i}_r \mathsf{J}_{11} \mathfrak{i}_s$ is a faithful left $\mathfrak{B}_i$-module and a faithful right $\mathfrak{B}_j$-module;
		\item[b)] For each $s=1,\dots, p$, there exists a $\mathsf{G}$-graded vector space $\mathsf{N}_s=\mathsf{span}_\mathbb{F}\{d_{1s},\dots,d_{r_s s}\}\subset\mathfrak{i}_s \mathsf{J}_{11} \mathfrak{i}_s$ such that $\mathfrak{i}_s \mathsf{J}_{11} \mathfrak{i}_s=\mathfrak{B}_s \mathsf{N}_s$ and $bd_{is}=\gamma_{is}(b)d_{is}b\neq0$ for any nonzero $b\in\beta_s$, and $i=1,\dots,r_s$, where $\gamma_{is}\in\mathbb{F}$, and $\beta_s=\{E_{l_sj_s}\eta_{h_s}\in\mathfrak{B}_s: l_s,j_s=1,\dots, k_s, h_s\in \mathsf{H}_s\}$ is the canonical homogeneous basis of $\mathfrak{B}_s=M_{k_s}(\mathbb{F}^{\sigma_s}[\mathsf{H}_s])$. Moreover, for each $i=1,\dots, r_s$, we have that $\mathfrak{B}_s d_{is}$ is a $\mathsf{G}$-simple $\mathfrak{B}_s$-bimodule.
	\end{itemize}

The decomposition of Jacobson radical of a finite dimensional graded algebra presented in the previous theorem can be used to describe graded varieties\footnote{A graded variety generated by $S\subset\mathbb{F}\langle X^\mathsf{G} \rangle$, denoted by $\mathsf{var}^\mathsf{G}(S)$, is the class of graded associative algebras that satisfy any $g\in S$, i.e. a $\mathsf{G}$-graded algebra $\mathfrak{A}$ belongs to $\mathsf{var}^\mathsf{G}(S)$ iff $g\equiv_\mathsf{G}0$ in $\mathfrak{A}$ for any $g\in S$. For more details about (graded) varieties of (graded) algebras, see \cite{Dren00}, Chapter 2, or \cite{GiamZaic05}, Chapter 1.} of graded algebras, which are studied in PI-Theory. A famous problem in PI-Theory is the well known {\bf Specht's Problem}\footnote{Specht's Problem was purposed in \cite{Spec50} by W. Specht (1950), and it can be formulated by the following question: given any algebra $\mathfrak{A}$, is any set of polynomial identities of $\mathfrak{A}$ a consequence of a finite number of identities of $\mathfrak{A}$? For more details about Specht's Problem, see \cite{BeloRoweVish12}.}. Kemer showed in \cite{Keme91} that the Specht's Problem has a positive solution in the variety of associative algebras of characteristic zero. In \cite{Svir11}, Sviridova proved the graded case (see Theorem 2, \cite{Svir11}). Specifically, Sviridova showed that ``\textit{if $\mathbb{F}$ is an algebraically closed field of characteristic zero, and $\mathsf{G}$ is any finite abelian group, then any $\mathsf{G} T$-ideal of graded identities of a $\mathsf{G}$-graded associative $PI$-algebra over $\mathbb{F}$ coincides with the ideal of $\mathsf{G}$-graded identities of the $\mathsf{G}$-graded Grassmann envelope of some finite dimensional $\mathsf{G} \times \mathbb{Z}_2$-graded associative $\mathbb{F}$-algebra}''. It is important to comment that Aljadeff and Kanel-Belov proved in \cite{AljaBelo10} (see Theorem 1.3), independently of Sviridova, a similar result for any finite group. From this, to study a graded variety $\mathfrak{W}$ of $\mathsf{G}$-graded associative algebras, it is enough to study the {\bf Grassmann envelope}\footnote{The Grassmann envelope of a $(\mathsf{G}\times\mathbb{Z}_2)$-graded algebra $\mathfrak{A}$ is given by $\mathsf{E}^\mathsf{G}(\mathfrak{A})=\left(\mathfrak{A}_0\otimes \mathsf{E}_0 \right) \oplus \left(\mathfrak{A}_1\otimes \mathsf{E}_1 \right)$, which is naturally $\mathsf{G}\times\mathbb{Z}_2$-graded and $\mathsf{G}$-graded. Here, $\mathsf{E}=\mathsf{E}_0\oplus \mathsf{E}_1$ is an infinitely generated non-unitary Grassmann algebra, i.e. $\mathsf{E}=\langle{e_1,e_2,e_3,\dots\mid e_ie_j=-e_je_i, \forall i,j}\rangle$ is $\mathbb{Z}_2$-graded with $\mathsf{E}_0=\mathsf{span}_\mathbb{F}\{e_{i_1}e_{i_2}\cdots e_{i_n}: n \mbox{ is even}\}$, and $\mathsf{E}_1=\mathsf{span}_\mathbb{F}\{e_{j_1}e_{j_2}\cdots e_{j_m}: m \mbox{ is odd}\}$ (see $\S$3.7 in \cite{GiamZaic05}, p.80-83).} of a finite dimensional $\mathsf{G}\times\mathbb{Z}_2$-graded algebra of type $\mathfrak{A}=\left(M_{k_1}(\mathbb{F}^{\sigma_1}[\mathsf{H}_1]) \times \cdots \times M_{k_p}(\mathbb{F}^{\sigma_p}[\mathsf{H}_p])\right) \oplus \mathsf{J}$, where $\mathsf{J}=\mathsf{J}(\mathfrak{A})$ has a decomposition as the described in Theorem \ref{1.30}.

To conclude, in the conditions of Corollary \ref{1.32}, i.e. the units $\mathfrak{i}_r$'s of the $M_{k_r}(\mathbb{F}^{\sigma_r}[\mathsf{H}_r])$'s are central in $\mathfrak{A}$, we have that $\mathsf{J}=\mathsf{J}_{00}\oplus\mathsf{J}_{11}$, where $\mathsf{J}_{11}=\mathfrak{i}_1 \mathsf{J} \mathfrak{i}_1\oplus\cdots\oplus\mathfrak{i}_p \mathsf{J} \mathfrak{i}_p$. And thus, it is easy to see that
\begin{equation}\nonumber
	\mathfrak{A}\cong_{\mathsf{G}}  \left(M_{k_1}(\mathbb{F}^{\sigma_1}[\mathsf{H}_1])\oplus\mathfrak{i}_1 \mathsf{J} \mathfrak{i}_1 \right) \times \cdots \times \left(M_{k_p}(\mathbb{F}^{\sigma_p}[\mathsf{H}_p])\oplus\mathfrak{i}_p \mathsf{J} \mathfrak{i}_p \right) \times\mathsf{J}_{00} \ .
	\end{equation}
Therefore, to study the graded variety $\mathfrak{W}$, and consequently its algebras, it is enough to study the algebras $\mathsf{E}^\mathsf{G}\left(M_{k_r}(\mathbb{F}^{\sigma_r}[\mathsf{H}_r])\oplus \mathfrak{i}_r \mathsf{J} \mathfrak{i}_r\right)$'s, and more specifically, the algebras $\mathsf{E}^\mathsf{G}\left(M_{k_r}(\mathbb{F}^{\sigma_r}[\mathsf{H}_r])\right)$'s and $\mathsf{E}^\mathsf{G}\left( \mathfrak{i}_r \mathsf{J} \mathfrak{i}_r\right)$'s.


\section{Preliminary Results}\label{preliminary}

In this section, we recall some concepts, although basic, which are important for our study. We talk a little bit about the definitions of ``\textit{algebra grading by a group}'', ``\textit{bimodules over algebras}'' and ``\textit{graded bimodules over graded algebras}'', as well as other important definitions. Later, let us introduce ``\textit{$\mathsf{G}$-graded bimodules}'', ``\textit{weak $\mathsf{G}$-noetherian bimodules}'' and ``\textit{weak $\mathsf{G}$-artinian bimodules}'', and other definitions connected to graded bimodules. We dedicate a part of the work to recall some (important) definitions and properties of ``\textit{2-cocycles}'', ``\textit{twisted group algebra}'' and ``\textit{group characters}''. The reader familiarized with these definitions and their main properties can continue from $\S$\ref{firstresults}. Here, unless stated otherwise, all the algebras $\mathfrak{A}$ are associative algebras over a field $\mathbb{F}$, and $\mathsf{G}$ is a group.

\subsection{Bimodules over Algebras}

Initially, let us define bimodules over algebras, as well as other related definitions. Let us assume known the main definitions and properties of the (one-sided) modules over algebras. For a detailed view of (one-sided) modules over algebras, we suggest reading Chapter 2 in \cite{CurtRein62}, or Chapter 6 in \cite{Rotm10}.


\begin{definition}
	Let $\mathbb{F}$ be a field, $\mathfrak{A}$ and $\widetilde{\mathfrak{A}}$ two $\mathbb{F}$-algebras (not necessarily unitary), and $\mathsf{M}$ an $\mathbb{F}$-vector space. We say that $\mathsf{M}$ is an {\bf $(\mathfrak{A}, \widetilde{\mathfrak{A}})$-bimodule} if it is a left $\mathfrak{A}$-module and a right $\widetilde{\mathfrak{A}}$-module, and its multiplication by scalar satisfies the associative law
	\begin{equation}\nonumber
r(ms)=(rm)s
	\end{equation}
for any $r\in\mathfrak{A}$, $s\in\widetilde{\mathfrak{A}}$ and $m\in\mathsf{M}$. When $\mathfrak{A}=\widetilde{\mathfrak{A}}$, we say that $\mathsf{M}$ is an {\bf $\mathfrak{A}$-bimodule}. When $\mathfrak{A}$ and $\widetilde{\mathfrak{A}}$ are unitary, $\mathsf{M}$ is a {\bf unitary $(\mathfrak{A}, \widetilde{\mathfrak{A}})$-bimodule} iff $\mathsf{M}$ is a unitary left $\mathfrak{A}$-module and a unitary right $\widetilde{\mathfrak{A}}$-module.
\end{definition}

Let $\mathfrak{A}$ and $\widetilde{\mathfrak{A}}$ be two algebras, and $\mathsf{M}$ an $(\mathfrak{A},\widetilde{\mathfrak{A}})$-bimodule. Recall that a \textbf{left} (resp. \textbf{right}) \textbf{submodule} $\mathsf{N}$ of $\mathsf{M}$ is a subspace of $\mathsf{M}$ which is also a left $\mathfrak{A}$-module (resp. right $\tilde{\mathfrak{A}}$-module). Take a subset $S$ of $\mathsf{M}$. 
Let us denote by $\mathfrak{A} S$ (resp. $S\tilde{\mathfrak{A}}$) the left (resp. right) submodule of $\mathsf{M}$ given by $\mathfrak{A} S=\mathsf{span}_{\mathbb{F}}\left\{r m \in\mathsf{M} : r\in\mathfrak{A}, m\in S \right\}$ (resp. $S\tilde{\mathfrak{A}}=\mathsf{span}_{\mathbb{F}}\{ \tilde{m} s \in\mathsf{M} : s\in\tilde{\mathfrak{A}}, \tilde{m}\in S \}$). 
 Observe that not always $S$ is a subset of $\mathfrak{A} S$ (resp. $S\tilde{\mathfrak{A}}$). From this, we define the \textbf{left} (resp. \textbf{right}) \textbf{submodule of $\mathsf{M}$ generated by $S$}, denoted by $_\mathfrak{A} S$ (resp. $S_{\tilde{\mathfrak{A}}}$), as being $_\mathfrak{A} S=\mathsf{span}_\mathbb{F}\{m\in S\}+\mathfrak{A} S$ (resp. $S_{\tilde{\mathfrak{A}}}=\mathsf{span}_\mathbb{F}\{m\in S\}+S\tilde{\mathfrak{A}}$). 
Observe that, necessarily, $S$ is a subset of $_\mathfrak{A} S$ (resp. $S_{\tilde{\mathfrak{A}}}$). If $\mathfrak{A}$ (resp. $\tilde{\mathfrak{A}}$) is unitary, and $\mathsf{M}$ is a unitary left $\mathfrak{A}$-module (resp. right $\tilde{\mathfrak{A}}$-module), then $_\mathfrak{A} S=\mathfrak{A} S$ (resp. $S_{\tilde{\mathfrak{A}}}=S\tilde{\mathfrak{A}}$), and $_\mathfrak{A} S$ (resp. $S_{\tilde{\mathfrak{A}}}$) is a unitary left $\mathfrak{A}$-module (resp. right $\tilde{\mathfrak{A}}$-module). Already a {\bf subbimodule} $\mathsf{N}$ of $\mathsf{M}$ is a subspace of $\mathsf{M}$ which is also an $(\mathfrak{A}, \widetilde{\mathfrak{A}})$-bimodule. Let us denote by $ \mathfrak{A} S\widetilde{\mathfrak{A}}$ the subbimodule of $\mathsf{M}$ given by $\mathfrak{A} S\widetilde{\mathfrak{A}}=\mathsf{span}_{\mathbb{F}}\{q \breve{m} p\in\mathsf{M} : q\in\mathfrak{A},p\in\widetilde{\mathfrak{A}}, \breve{m}\in S\}$. We define the {\bf subbimodule of $\mathsf{M}$ generated by $S$}, denoted by $_\mathfrak{A} S_{\widetilde{\mathfrak{A}}}$, as being $_\mathfrak{A} S_{\widetilde{\mathfrak{A}}}=\mathsf{span}_\mathbb{F}\{m\in S\}+\mathfrak{A} S + S\widetilde{\mathfrak{A}} + \mathfrak{A} S\widetilde{\mathfrak{A}}$. Observe that $\mathfrak{A} S\widetilde{\mathfrak{A}}\subseteq {_\mathfrak{A} S}_{\widetilde{\mathfrak{A}}}$ and $S\subseteq {_\mathfrak{A} S}_{\widetilde{\mathfrak{A}}}$, but not necessarily $S$ belongs to $\mathfrak{A} S\widetilde{\mathfrak{A}}$. 
If $S=\{m\}$, we denote $\mathfrak{A} S$, $S\tilde{\mathfrak{A}}$ and $\mathfrak{A} S\widetilde{\mathfrak{A}}$ by $\mathfrak{A} m$, $m\tilde{\mathfrak{A}}$ and $\mathfrak{A} m\widetilde{\mathfrak{A}}$, respectively, when no confusion can arise. When $\mathfrak{A}$ and $\widetilde{\mathfrak{A}}$ are unitary, and $\mathsf{M}$ is a unitary $(\mathfrak{A},\widetilde{\mathfrak{A}})$-bimodule, we have that $_\mathfrak{A} S_{\widetilde{\mathfrak{A}}}=\mathfrak{A} S\widetilde{\mathfrak{A}}$, and $_\mathfrak{A} S_{\widetilde{\mathfrak{A}}}$ is a unitary $(\mathfrak{A},\widetilde{\mathfrak{A}})$-bimodule.


\begin{definition}
An $(\mathfrak{A},\widetilde{\mathfrak{A}})$-bimodule $\mathsf{M}$ is called {\bf irreducible} (or {\bf simple}) if $\mathfrak{A} \mathsf{M}\widetilde{\mathfrak{A}}\neq\{0\}$, and $\{0\}$ and $\mathsf{M}$ are the only subbimodules of $\mathsf{M}$. Particularly, the condition $\mathfrak{A}\mathsf{M}\widetilde{\mathfrak{A}}\neq\{0\}$ means that $am\tilde{a}\neq0$ for some $a\in\mathfrak{A}$, $\tilde{a}\in\widetilde{\mathfrak{A}}$ and $m\in\mathsf{M}$.
\end{definition}

For any irreducible $(\mathfrak{A},\widetilde{\mathfrak{A}})$-bimodule $\mathsf{M}$, we have $\mathsf{M}=\mathfrak{A} m\widetilde{\mathfrak{A}}$ for any nonzero $m\in\mathsf{M}$. Indeed, it is sufficient to see that $\mathsf{N}=\{m\in\mathsf{M}: \mathfrak{A} m\widetilde{\mathfrak{A}}=\{0\}\}$ is a subbimodule of $\mathsf{M}$, and hence, we can conclude that $\mathsf{N}=\{0\}$. When $\mathsf{M}$ is a unitary $(\mathfrak{A},\widetilde{\mathfrak{A}})$-bimodule (and hence, $\mathfrak{A}$ and $\widetilde{\mathfrak{A}}$ are unitary also), the condition ``$\mathfrak{A}\mathsf{M}\widetilde{\mathfrak{A}}\neq\{0\}$'' is equivalent to condition ``$\mathsf{M}\neq\{0\}$''.

Notice that, given a subalgebra $I$ of $\mathfrak{A}$ such that $\mathfrak{A} I \mathfrak{A}\neq\{0\}$, $I$ is a (minimal) two-sided ideal of $\mathfrak{A}$ iff $I$ is an (irreducible) $\mathfrak{A}$-bimodule. 

Now, let us define a {\it quotient bimodule} and \textit{homomorphism of bimodules}. Let $\mathsf{N}$ be a subbimodule of an $(\mathfrak{A},\widetilde{\mathfrak{A}})$-bimodule $\mathsf{M}$. The {\bf quotient $(\mathfrak{A},\widetilde{\mathfrak{A}})$-bimodule} $\mathsf{M}/\mathsf{N}$ is defined as follow:
\begin{itemize}
	\item[i)] $\mathsf{M}/\mathsf{N}=\{\overline{m}=m+\mathsf{N}: m\in\mathsf{M}\}$ is a quotient vector space;
	\item[ii)] $a\overline{m}=\overline{am}$ for any $a\in\mathfrak{A}$ and $m\in\mathsf{M}$;
	\item[iii)] $\overline{m}b=\overline{mb}$ for any $b\in\widetilde{\mathfrak{A}}$ and $m\in\mathsf{M}$.
\end{itemize}
It is clear that $a\overline{m}b= \overline{am}b= a\overline{mb}= \overline{amb}$ for any $a\in\mathfrak{A}$, $b\in\widetilde{\mathfrak{A}}$ and $m\in\mathsf{M}$. By the three above items, we have that $\mathsf{M}/\mathsf{N}$ is an $(\mathfrak{A},\widetilde{\mathfrak{A}})$-bimodule naturally.


\begin{definition}
	Let $\mathfrak{A}$ and $\widetilde{\mathfrak{A}}$ be two algebras, $\mathsf{M}$ and $\widetilde{\mathsf{M}}$ two $(\mathfrak{A},\widetilde{\mathfrak{A}})$-bimodules and $\varphi: \mathsf{M}\longrightarrow\widetilde{\mathsf{M}}$ a linear transformation. We say that $\varphi$ is a {\bf homomorphism of $(\mathfrak{A},\widetilde{\mathfrak{A}})$-bimodules} if $\varphi$ satisfies $\varphi(am)=a\varphi(m)$ and $\varphi(mb)=\varphi(m)b$ for any $a\in\mathfrak{A}$, $b\in\widetilde{\mathfrak{A}}$ and $m\in\mathsf{M}$.
\end{definition}



\subsection{Graded Bimodules over Graded Algebras}
%

Let $\mathfrak{A}$ be an $\mathbb{F}$-algebra and $\mathsf{G}$ a group. A $\mathsf{G}$-grading on $\mathfrak{A}$ is a decomposition $\Gamma: \mathfrak{A}=\bigoplus_{g\in\mathsf{G}} \mathfrak{A}_g$ in $\mathbb{F}$-vector subspaces $\mathfrak{A}_g$'s of $\mathfrak{A}$ that satisfy $\mathfrak{A}_g \mathfrak{A}_h\subseteq \mathfrak{A}_{gh}$ for all $g,h\in\mathsf{G}$.

The group algebra $\mathbb{F}\mathsf{G}$ has a natural $\mathsf{G}$-grading. In general, for any group $\mathsf{G}$, any algebra $\mathfrak{A}$ has a $\mathsf{G}$-grading, called \textit{trivial grading}, defined by $\mathfrak{A}_e=\mathfrak{A}$ and $\mathfrak{A}_g=\{0\}$ for any $g\in\mathsf{G}\setminus\{e\}$, where $e$ is the neutral element of $\mathsf{G}$.

\begin{definition}\label{1.58}
	A $\mathsf{G}$-graded algebra $\mathfrak{A}$ is said to be {\bf $\mathsf{G}$-simple} (or {\bf simple graded}, or {\bf minimal graded}) if $\mathfrak{A}^2\neq\{0\}$ and $\mathfrak{A}$ does not have proper $\mathsf{G}$-graded ideals, i.e. if $I$ is a graded ideal of $\mathfrak{A}$, then either $I=\{0\}$ or $I=\mathfrak{A}$. Moreover, assuming $\mathfrak{A}$ be unitary, $\mathfrak{A}$ is a {\bf $\mathsf{G}$-division} (or {\bf division $\mathsf{G}$-graded}) algebra if all its nonzero homogeneous elements are invertible in $\mathfrak{A}$, i.e. for any $a\in \bigcup_{g\in\mathsf{G}}\mathfrak{A}_g$, $a\neq0$, there exists $a^{-1}\in \mathfrak{A}$ such that $aa^{-1}=a^{-1}a=1$.
\end{definition}

 Note that for any $a\in \bigcup_{g\in\mathsf{G}}\mathfrak{A}_g$, $a\neq0$, its inverse $a^{-1}$ is also homogeneous of degree $\mathsf{deg}(a^{-1})=(\mathsf{deg}(a))^{-1}$. Observe also that if $\mathfrak{A}$ is a simple algebra, thus $\mathfrak{A}$ is $\mathsf{G}$-simple. For more details about graded algebras (rings), we indicated the references \cite{NastOyst04,NastOyst11}.

%

Let us now define a $\mathsf{G}$-grading on an $(\mathfrak{A},\widetilde{\mathfrak{A}})$-bimodule $\mathsf{M}$, where $\mathfrak{A}$ and $\widetilde{\mathfrak{A}}$ are two $\mathsf{G}$-graded algebras and $\mathsf{G}$ is a group.

\begin{definition} 
Let $\mathsf{G}$ be a group, $\mathfrak{A}$ and $\widetilde{\mathfrak{A}}$ two $\mathsf{G}$-graded $\mathbb{F}$-algebras and $\mathsf{M}$ an $(\mathfrak{A},\widetilde{\mathfrak{A}})$-bimodule. A $\mathsf{G}$-grading on $\mathsf{M}$ is a decomposition of $\mathsf{M}$ in a direct sum of $\mathbb{F}$-subspaces $\mathsf{M}_g$ of $\mathsf{M}$, $g\in\mathsf{G}$, satisfying $\mathfrak{A}_g \mathsf{M}_h\subseteq \mathsf{M}_{gh}$ and $\mathsf{M}_h\widetilde{\mathfrak{A}}_{t}\subseteq \mathsf{M}_{ht}$, for any $g,h,t\in\mathsf{G}$. In this case, we say that $\mathsf{M}$ is a {\bf $\mathsf{G}$-graded $(\mathfrak{A},\widetilde{\mathfrak{A}})$-bimodule}.
\end{definition}

Analogously to the case of graded algebras, for any group $\mathsf{G}$, any bimodule $\mathsf{M}$ has a $\mathsf{G}$-grading, called \textit{trivial grading}, defined by $\mathsf{M}_e=\mathsf{M}$ and $\mathsf{M}_g=\{0\}$ for any $g\in\mathsf{G}\setminus\{e\}$, where $e$ is the neutral element of $\mathsf{G}$. A special case in the study of the $\mathsf{G}$-graded $(\mathfrak{A},\widetilde{\mathfrak{A}})$-bimodules occurs when $\mathfrak{A}=\widetilde{\mathfrak{A}}$. Such bimodules are called \textbf{$\mathsf{G}$-graded $\mathfrak{A}$-bimodules}. 

A subbimodule $\mathsf{N}$ of $\mathsf{M}$ is said to be a \textbf{graded subbimodule} if $\mathsf{N}=\bigoplus_{g\in\mathsf{G}}(\mathsf{N}\cap\mathsf{M}_g)$. Observe that a subbimodule of a graded bimodule is graded iff it can be generated as a bimodule by its homogeneous elements.

%

\begin{definition}\label{1.87}
Let $\mathsf{M}$ be a $\mathsf{G}$-graded $(\mathfrak{A},\widetilde{\mathfrak{A}})$-bimodule. We say that $\mathsf{M}$ is {\bf irreducible graded} (or \textbf{$\mathsf{G}$-irreducible}) if $\mathfrak{A} \mathsf{M}\widetilde{\mathfrak{A}}\neq\{0\}$, and $\mathsf{M}$ does not have proper graded subbimodules. This means that $\mathsf{M}$ is an irreducible $\mathsf{G}$-graded $(\mathfrak{A},\widetilde{\mathfrak{A}})$-bimodule iff $\mathfrak{A} \mathsf{M}\widetilde{\mathfrak{A}}\neq\{0\}$, and $\{0\}$ and $\mathsf{M}$ are the only graded subbimodules of $\mathsf{M}$.
\end{definition}

By above definition, the condition $\mathfrak{A}\mathsf{M}\widetilde{\mathfrak{A}}\neq\{0\}$ means that $am\tilde{a}\neq0$ for some homogeneous elements $a\in\mathfrak{A}$, $m\in\mathsf{M}$ and $\tilde{a}\in\widetilde{\mathfrak{A}}$. Obviously, for any irreducible $\mathsf{G}$-graded $(\mathfrak{A},\widetilde{\mathfrak{A}})$-bimodule $\mathsf{M}$, we have $\mathsf{M}=\mathfrak{A} m\widetilde{\mathfrak{A}}$ for any nonzero homogeneous element $m\in\mathsf{M}$. In fact, it is sufficient to see that $\mathsf{N}=\{m\in\mathsf{M}: \mathfrak{A} m\widetilde{\mathfrak{A}}=\{0\}\}$ is a graded subbimodule of $\mathsf{M}$, and to conclude that $\mathsf{N}=\{0\}$, because $\mathfrak{A}\mathsf{M}\widetilde{\mathfrak{A}}\neq\{0\}$ and $\mathfrak{A} \mathsf{N}\widetilde{\mathfrak{A}}=\{0\}$, and so $\mathsf{N}\neq\mathsf{M}$.

Notice that, given a graded subalgebra $I$ of $\mathfrak{A}$ such that $\mathfrak{A} I \mathfrak{A}\neq\{0\}$, $I$ is a (simple) graded two-sided ideal of $\mathfrak{A}$ iff $I$ is a (irreducible) graded $\mathfrak{A}$-bimodule. 

\begin{definition}\label{1.08}
Let $\mathsf{M}$ be a $\mathsf{G}$-graded $(\mathfrak{A},\widetilde{\mathfrak{A}})$-bimodule and $\mathsf{N}$ a proper graded subbimodule of $\mathsf{M}$. We say that $\mathsf{N}$ is a {\bf maximal graded} (or \textbf{$\mathsf{G}$-maximal}) subbimodule of $\mathsf{M}$ if $\mathsf{M}/\mathsf{N}$ is a $\mathsf{G}$-irreducible $(\mathfrak{A},\widetilde{\mathfrak{A}})$-bimodule. In other words, $\mathsf{N}$ is $\mathsf{G}$-maximal in $\mathsf{M}$ iff $\mathsf{N}$ is proper and $\mathsf{N}+ {_{\mathfrak{A}}}\{m\}_{\widetilde{\mathfrak{A}}}=\mathsf{M}$ for any $m\in\mathsf{M}\setminus\mathsf{N}$.
\end{definition}

It is important to note that, since any graded subbimodule $\mathsf{N}$ of $\mathsf{M}$ is still a left (and also right) graded submodule of $\mathsf{M}$, any $\mathsf{G}$-maximal (resp. $\mathsf{G}$-irreducible) subbimodule $\mathsf{N}$ of $\mathsf{M}$ is either a left $\mathsf{G}$-maximal (resp. left $\mathsf{G}$-irreducible) of $\mathsf{M}$ or is contained in (resp. contains) some.

Now, given a graded subbimodule $\mathsf{N}$ of a $\mathsf{G}$-graded  $(\mathfrak{A},\widetilde{\mathfrak{A}})$-bimodule $\mathsf{M}$, we have that the quotient $(\mathfrak{A},\widetilde{\mathfrak{A}})$-bimodule $\mathsf{M}/\mathsf{N}$ is naturally a $\mathsf{G}$-graded $(\mathfrak{A},\widetilde{\mathfrak{A}})$-bimodule. In fact, since $\mathsf{M}=\bigoplus_{g\in\mathsf{G}}\mathsf{M}_g$ and $\mathsf{N}=\bigoplus_{g\in\mathsf{G}}(\mathsf{N}\cap\mathsf{M}_g)$, we have the quotient space $\mathsf{M}_g/(\mathsf{N}\cap\mathsf{M}_g)$ is well defined, for any $g\in\mathsf{G}$. It is easy to see that
\begin{equation}\nonumber
\displaystyle\frac{\mathsf{M}}{\mathsf{N}}=\bigoplus_{g\in\mathsf{G}} \frac{\mathsf{M}_g}{\mathsf{N}\cap\mathsf{M}_g} \ ,
\end{equation}
and thus, $\mathsf{M}/\mathsf{N}$ is a $\mathsf{G}$-graded $(\mathfrak{A},\widetilde{\mathfrak{A}})$-bimodule, called {\bf graded quotient $(\mathfrak{A},\widetilde{\mathfrak{A}})$-bimodule}.

 
Let us now define \textit{homogeneous homomorphisms} and \textit{graded homomorphisms} of $(\mathfrak{A},\widetilde{\mathfrak{A}})$-bimodules.

	\begin{definition}
		Let $\mathsf{G}$ be a group, $\mathfrak{A}$ and $\widetilde{\mathfrak{A}}$ two $\mathsf{G}$-graded algebras, $\mathsf{M}$ and $\widetilde{\mathsf{M}}$ two $\mathsf{G}$-graded $(\mathfrak{A},\widetilde{\mathfrak{A}})$-bimodules, and $\varphi: \mathsf{M}\rightarrow\widetilde{\mathsf{M}}$ a homomorphism of $(\mathfrak{A},\widetilde{\mathfrak{A}})$-bimodules. We say that $\varphi$ is a {\bf homogeneous homomorphism} of degree $h_0\in\mathsf{G}$ if $\varphi(\mathsf{M}_g)\subseteq \left( \widetilde{\mathsf{M}}_{gh_0}\bigcap \widetilde{\mathsf{M}}_{h_0g} \right)$ for any $g\in\mathsf{G}$. A finite sum of homogeneous homomorphisms of $(\mathfrak{A},\widetilde{\mathfrak{A}})$-bimodules is called a {\bf graded homomorphism of $(\mathfrak{A},\widetilde{\mathfrak{A}})$-bimodules}.
	\end{definition}

Observe that if $\varphi$ is a homogeneous homomorphism of degree $h_0$, thus either $h_0$ belongs to center of $\{g\in\mathsf{G} : \varphi{(\mathsf{M}_g)}\neq\{0\}\}$ in $\mathsf{G}$, and hence, $\widetilde{\mathsf{M}}_{gh_0} = \widetilde{\mathsf{M}}_{h_0g}$ for any $g\in\mathsf{G}$, or $\varphi(\mathsf{M}_g)=\{0\}$ when $gh_0\neq h_0g$. 

Notice that any homogeneous homomorphism of $\mathsf{G}$-graded is also a graded homomorphism. Not always the kernel or image of a graded homomorphism is a graded subbimodule, but kernel and image of any  homogeneous homomorphism of graded bimodules are graded subbimodules. We write $\mathsf{M}\cong_\mathsf{G}\widetilde{\mathsf{M}}$ when two $\mathsf{G}$-graded $(\mathfrak{A},\widetilde{\mathfrak{A}})$-bimodules $\mathsf{M}$ and $\widetilde{\mathsf{M}}$ are homogeneously isomorphic, i.e. there exists a homogeneous isomorphism $\psi:\mathsf{M}\rightarrow\widetilde{\mathsf{M}}$.

Next, let us talk a little bit about chain conditions for graded bimodules.

\begin{definition}
Let $\mathsf{G}$ be a group, $\mathfrak{A}$ and $\widetilde{\mathfrak{A}}$ two $\mathsf{G}$-graded algebras and $\mathsf{M}$ a $\mathsf{G}$-graded $(\mathfrak{A},\widetilde{\mathfrak{A}})$-bimodule. We say that $\mathsf{M}$ is {\bf weak $\mathsf{G}$-noetherian} (resp. {\bf weak $\mathsf{G}$-artinian}) if it satisfies the ascending (resp. descending) chain condition (ACC) (resp. (DCC)) for graded subbimodules.
\end{definition}

When $\mathsf{M}$ is a $\mathsf{G}$-graded $(\mathfrak{A},\widetilde{\mathfrak{A}})$-bimodule which is weak $\mathsf{G}$-noetherian and weak $\mathsf{G}$-artinian at the same time, then we say that ``\textit{$\mathsf{M}$ satisfies both weak chain conditions}" instead of ``\textit{$\mathsf{M}$ satisfies both chain conditions for graded subbimodules}".

Any finite dimensional $\mathsf{G}$-graded $(\mathfrak{A},\widetilde{\mathfrak{A}})$-bimodules and any irreducible $\mathsf{G}$-graded $(\mathfrak{A},\widetilde{\mathfrak{A}})$-bimodules are examples of graded bimodules which satisfy both weak chain conditions.

\begin{remark}\label{1.69}
Let $\mathsf{M}$ be a $\mathsf{G}$-graded $(\mathfrak{A},\widetilde{\mathfrak{A}})$-bimodule, where $\mathsf{G}$ is a group, and $\mathfrak{A}$ and $\widetilde{\mathfrak{A}}$ are two $\mathsf{G}$-graded algebras. Since any graded subbimodule of $\mathsf{M}$ is also a left graded submodule, we have naturally that if $\mathsf{M}$ is left $\mathsf{G}$-noetherian (resp. $\mathsf{G}$-artinian), thus $\mathsf{M}$ is weak $\mathsf{G}$-noetherian (resp. weak $\mathsf{G}$-artinian). The ``right" case is analogous.
\end{remark} 

\begin{definition}
A \textbf{weak graded composition series} for an $(\mathfrak{A},\widetilde{\mathfrak{A}})$-bimodule $\mathsf{M}$ is a finite sequence $\mathsf{N}_{0}, \mathsf{N}_{1}, \dots, \mathsf{N}_{r}$ of graded subbimodules of $\mathsf{M}$ satisfying
\begin{equation}\nonumber
\{0\}=\mathsf{N}_{0}\subsetneq \mathsf{N}_{1}\subsetneq \cdots \subsetneq \mathsf{N}_{r}=M ,
\end{equation}
where each $\mathsf{N}_k$ is a maximal graded subbimodule of $\mathsf{N}_{k+1}$, i.e. $\displaystyle\frac{\mathsf{N}_{k+1}}{\mathsf{N}_{k}}$ is an irreducible graded $(\mathfrak{A},\widetilde{\mathfrak{A}})$-bimodule for each $k$. In this case, we say that $\mathsf{M}$ has a weak graded composition series with length $r$, and we denote by $l(\mathsf{M})$ the smallest length of a weak graded composition series of $\mathsf{M}$.
\end{definition}

Obviously, any $\mathsf{G}$-graded $(\mathfrak{A},\widetilde{\mathfrak{A}})$-bimodule which has a graded composition series is finitely generated (as a graded bimodule).


\subsection{Matrix over a twisted group algebra}

Here, let us define an important object in our study: the twisted group algebras. As previously, $\mathsf{G}$ denotes a group and $e\in\mathsf{G}$ its neutral element, $\mathbb{F}$ a field, and $\mathbb{F}^*=\mathbb{F}\setminus\{0\}$.

\begin{definition}\label{1.05}
	The mapping $\sigma: \mathsf{G}\times \mathsf{G} \longrightarrow \mathbb{F}^*$ which satisfies 
	\begin{equation}\nonumber
		\sigma(x,y)\sigma(xy,z)=\sigma(x,yz)\sigma(y,z) \mbox{ for all } x,y,z\in \mathsf{G}
	\end{equation}
	is called a {\bf $2$-cocycle} on $\mathsf{G}$ with values in $\mathbb{F}^*$. The set of all $2$-cocycles from $\mathsf{G}$ into $\mathbb{F}^*$ is denoted by $\mathsf{Z}^2(\mathsf{G}, \mathbb{F}^*)$.
\end{definition}


\begin{remark}\label{1.07}
	Let $\sigma: \mathsf{G}\times \mathsf{G}\longrightarrow \mathbb{F}^*$ be a $2$-cocycle on $\mathsf{G}$. To prove that $\sigma(x,e)=\sigma(e,y)$ for any $x,y\in \mathsf{G}$, it is enough to put $y=e$ in Definition \ref{1.05}. From this, it follows that $\sigma(x,e)=\sigma(e,x)=\sigma(e,e)$ for any $x\in \mathsf{G}$. Again in Definition \ref{1.05}, if we replace $y=x^{-1}$ and $z=x$, we can prove that $\sigma(x,x^{-1})=\sigma(x^{-1},x)$, for any $x\in \mathsf{G}$.
\end{remark}

Now, consider the group algebra $\mathbb{F}\mathsf{G}$. Observe that $\mathbb{F}\mathsf{G}$ is an associative algebra with unity; if $\mathsf{G}$ is commutative (resp. finite), then $\mathbb{F}\mathsf{G}$ is a commutative (resp. finite dimensional) algebra. Also, $\mathbb{F}\mathsf{G}$ has a natural $\mathsf{G}$-grading. Let us consider a more general situation.

\begin{definition}
Let $\mathsf{G}$ be a group, $\mathbb{F}$ a field, and $\sigma: \mathsf{G}\times \mathsf{G}\longrightarrow \mathbb{F}^*$ a $2$-cocycle on $\mathsf{G}$. Consider the $\mathbb{F}$-vector space 
	\begin{equation}\nonumber
	\mathbb{F}^\sigma[\mathsf{G}]=\left\{ \sum_{g\in \mathsf{G}} \alpha_g \eta_g: \alpha_g\in\mathbb{F}, g \in \mathsf{G} \right\} \ ,
	\end{equation}
where the set $\{\eta_g\}_{g\in\mathsf{G}}$ is linearly independent over $\mathbb{F}$. Consider on $\mathbb{F}^\sigma[\mathsf{G}]$ the multiplication which extends by linearity the product $\eta_g\eta_h=\sigma(g,h)\eta_{gh}$, $g, h\in \mathsf{G}$. With this multiplication, $\mathbb{F}^\sigma[\mathsf{G}]$ is an algebra, called a {\bf twisted group algebra}.
\end{definition}

Observe that the equality in Definition \ref{1.05} ensures that $\mathbb{F}^\sigma[\mathsf{G}]$ is an associative algebra. Obviously, if $\sigma$ is the trivial $2$-cocycle on $\mathsf{G}$, i.e. $\sigma(g,h)=1$ for any $g,h\in\mathsf{G}$, thus $\mathbb{F}^\sigma[\mathsf{G}]=\mathbb{F} \mathsf{G}$. Furthermore, we have that $\mathfrak{A}=\mathbb{F}^\sigma[\mathsf{G}]$ is $\mathsf{G}$-graded with the natural grading defined by $\mathfrak{A}_g = \mathsf{span}_\mathbb{F}\{\eta_g\}$, for each $g\in\mathsf{G}$.

\begin{remark}
	Note that $\mathbb{F}^\sigma[\mathsf{G}]$ is a unitary algebra, where its unity is given by $\sigma(e,e)^{-1} \eta_e$, since $\sigma(e,e)=\sigma(e,g)=\sigma(g,e)$ for any $g\in\mathsf{G}$ (by Remark \ref{1.07}). Moreover, $\mathbb{F}^{\sigma}[\mathsf{G}]$ is a division graded algebra, i.e. every nonzero homogeneous element of $\mathbb{F}^{\sigma}[\mathsf{G}]$ has multiplicative inverse in $\mathbb{F}^{\sigma}[\mathsf{H}]$. Note that $\mathsf{dim}(\mathbb{F}^{\sigma}[\mathsf{G}])<\infty$ iff $\mathsf{G}$ is a finite group.
\end{remark}

Now, let us present some results which will be our tools in the last part of this section.


	Consider the algebra $\mathfrak{B}=M_n(\mathbb{F}^\sigma[\mathsf{H}])$ of all $n\times n$ matrices over the twisted group algebra $\mathbb{F}^\sigma[\mathsf{H}]$, where $\mathsf{H}$ is a subgroup of a group $\mathsf{G}$ and $\sigma\in\mathsf{Z}^2(\mathsf{H},\mathbb{F}^*)$. We have that $|\mathsf{H}|<\infty$ implies $\mathsf{dim}(\mathfrak{B})=n^2|\mathsf{H}|<\infty$. Fix an arbitrary $k$-tuple $\xi=(g_1,\dots , g_k)\in G^k$. The vector subspaces defined by $\mathfrak{B}_g=\mathsf{span}_\mathbb{F}\{E_{ij}\eta_h \in\mathfrak{B}: g=g_i^{-1} hg_j \}$, $g\in\mathsf{G}$, define a $\mathsf{G}$-grading on $\mathfrak{B}$. This $\mathsf{G}$-grading is called {\bf canonical elementary grading} defined by $\xi$.

\begin{theorem}[Theorem 3, \cite{BahtSehgZaic08}] \label{teoBahtSehgZaic}
	Let $\mathfrak{A}$ be a finite dimensional $\mathsf{G}$-graded algebra over an algebraically closed field $\mathbb{F}$. Suppose that either $\mathsf{char}(\mathbb{F}) = 0$ or $\mathsf{char} (\mathbb{F})$ is coprime with the order of each finite subgroup of $\mathsf{G}$. Then $\mathfrak{A}$ is a graded simple algebra iff $\mathfrak{A}$ is isomorphic to the tensor product  $M_k (\mathbb{F})\otimes\mathbb{F}^\sigma [\mathsf{H}] \cong M_k (\mathbb{F}^\sigma [\mathsf{H}])$, where $\mathsf{H}$ is a finite subgroup of $\mathsf{G}$ and $\sigma\in\mathsf{Z}^2(\mathsf{H},\mathbb{F}^*)$. The $\mathsf{G}$-grading  on $M_k (\mathbb{F}^\sigma [\mathsf{H}])$ is a canonical elementary grading defined by a $k$-tuple $(g_1,\dots , g_k )\in \mathsf{G}^k$.
\end{theorem}


\subsection{Group Characters}

In this subsection, let us recall some definitions and properties of \textit{Group Characters}, which will be important for the development of this work. Although the results in this section are well-known, we add them here so that this work is self-contained. Here, let us denote by $\mathsf{G}$ an arbitrary finite multiplicative group with identity element $1$, $\mathbb{F}$ a field, and $GL_n(\mathbb{F})$ the group of invertible $n\times n$ matrices over $\mathbb{F}$.
 
 \begin{definition}
 A \textbf{representation of $\mathsf{G}$ of degree $n$} is a homomorphism of groups $T$ from $\mathsf{G}$ to $GL_n(\mathbb{F})$. Two representations $T$ and $T'$ are {\bf equivalent} if they have the same degree, namely $n$, and if there exists a fixed $S$ in $GL_n(\mathbb{F})$  such that $T'(g)=ST(g)S^{-1}$, for any $g\in \mathsf{G}$.
 \end{definition}

Let $T:\mathsf{G}\rightarrow GL_n(\mathbb{F})$ be a representation of a group $\mathsf{G}$. We say that $T$ is {\bf reducible} if there exist representations $T_1:\mathsf{G}\rightarrow GL_{n_1}(\mathbb{F})$ and $T_2:\mathsf{G}\rightarrow GL_{n_2}(\mathbb{F})$ of $\mathsf{G}$, with $n=n_1+n_2$, such that 
	\begin{equation}\nonumber
	T(g) \mbox{ and }\begin{bmatrix} T_1(g) & V(g)\\
		0 & T_2(g)
	\end{bmatrix}
\end{equation}
are equivalent, for any $g\in\mathsf{G}$, where $V(g)$ is a matrix over $\mathbb{F}$ of order $n_1\times n_2$ for each $g\in\mathsf{G}$; otherwise, then $T$ is an {\bf irreducible representation}. For more details, see \cite{CurtRein62,Isaa76,JameLieb01,Serr77}.

\begin{theorem}[Theorem 10.8 (Maschke), p.41, \cite{CurtRein62}]\label{1.59}
Let $\mathsf{G}$ be a finite group and $T: \mathsf{G} \rightarrow GL_n(\mathbb{F})$ a representation of $\mathsf{G}$. If $\mathsf{char}(\mathbb{F}) \nmid |\mathsf{G}|$, then there exist irreducible representations $T_1,\dots, T_n$ (not necessarily non-equivalent) of $\mathsf{G}$ such that $T=T_1+\cdots +T_n$.
\end{theorem}

  

  \begin{definition}
 Let $\mathsf{G}$ be a group, and $T: \mathsf{G}\rightarrow GL_n(\mathbb{F})$ a representation of $\mathsf{G}$. The {\bf character} of $T$ is a map $\zeta:\mathsf{G}\rightarrow \mathbb{F}$ defined by $\zeta(g)=\mathsf{tr}(T(g))$ for any $g\in\mathsf{G}$, where $\mathsf{tr}(T(g))$ is the trace of matrix $T(g)$. We say that $\zeta$ is an \textbf{irreducible character} when $T$ is an irreducible representation. The set of all the irreducible characters of $\mathsf{G}$ is called {\bf dual} of $\mathsf{G}$, and denoted by $\widehat{\mathsf{G}}$.
 \end{definition}

\begin{proposition}[Corollary 30.14, p.214, and Theorem 27.22, p.186, \cite{CurtRein62}]\label{1.60}
Let $T$ and $T_1$ be two representations of $\mathsf{G}$ over $\mathbb{F}$ with characters $\zeta$ and $\zeta_1$, respectively. If 
 $\mathsf{char}(\mathbb{F}) \nmid |\mathsf{G}|$, then $T$ and $T_1$ are equivalent iff $\zeta=\zeta_1$. In addition, if $\mathbb{F}$ is an algebraically closed field, then the number of distinct irreducible characters of $\mathsf{G}$ is finite.
\end{proposition}


\begin{proposition}[Orthogonality Relations, p.219-221,  \cite{CurtRein62}]\label{1.62}
	Let $\mathbb{F}$ be an algebraically closed field such that $\mathsf{char}(\mathbb{F})\nmid|\mathsf{G}|$, $\mathsf{G}$ a finite group, and $\zeta_1,\dots, \zeta_s$ all the distinct characters of $\mathsf{G}$. Then
\begin{enumerate}
\item[i)] $\sum_{i=1}^s (\zeta_i(1))^2=|\mathsf{G}|$;
\item[ii)] If $\zeta_i$'s are irreducible, then $\sum_{g\in\mathsf{G}} \zeta_i(g)\zeta_j(g^{-1})=\delta_{ij}|\mathsf{G}|$, where $\delta_{ij}$ is the Kronecker delta;
\item[iii)] For each $i=1,\dots,s$, the character $\zeta_i$ is irreducible iff $\sum_{g\in\mathsf{G}} \zeta_i(g)\zeta_i(g^{-1})=|\mathsf{G}|$.
\end{enumerate}
\end{proposition}
 
 The next result is a consequence of Proposition \ref{1.62}. 
 
\begin{theorem}[Theorem 9, p.25, \cite{Serr77}]\label{1.65}
	Let $\mathsf{G}$ be a finite group, and $\mathbb{F}$ an algebraically closed field such that $\mathsf{char}(\mathbb{F})\nmid |\mathsf{G}|$. Then $\mathsf{G}$ is abelian iff all the irreducible characters of $\mathsf{G}$ have degree $1$.
\end{theorem}

A consequence of previous theorem is that if $\mathsf{G}$ is finite abelian, then any irreducible character of $\mathsf{G}$ is a homomorphism of groups. In other words, if $\zeta: \mathsf{G}\rightarrow\mathbb{F}^*$ is an irreducible character of $\mathsf{G}$, then $\zeta(gh)=\zeta(g)\zeta(h)$ for any $g, h \in\mathsf{G}$.

The proposition below is a well-known result of Character Theory of Finite Group.

\begin{proposition}[Exercise 3.3, p.26, \cite{Serr77}]\label{1.68}
If $\mathsf{G}$ is a finite abelian group and $\mathbb{F}$ is an algebraically closed field such that $\mathsf{char}(\mathbb{F})=0$, then $\mathsf{G}$ and its dual $\widehat{\mathsf{G}}$ are isomorphic.
\end{proposition}
\begin{proof}
Let $k_1,\dots,k_n$ be prime powers such that $\mathsf{G}\cong \mathsf{C}_{k_1}\times\cdots\times\mathsf{C}_{k_n}$, where $\mathsf{C}_{k_i}$ is the (multiplicative) cyclic group of order $k_i$. It if easy to see that $\widehat{\mathsf{G}}\cong  \widehat{\mathsf{C}}_{k_1}\times\cdots\times\widehat{\mathsf{C}}_{k_n}$ and $|\mathsf{C}_{k_i}|=|\widehat{\mathsf{C}}_{k_i}|$ for all $i=1,\dots,n$. Since $\chi(g)$ is an $\mathsf{o}(g)$-th root of unity, for any character $\chi$ of $\mathsf{C}_{k_i}$ and any $g\in\mathsf{C}_{k_i}$, where $\mathsf{o}(g)$ is the order of $g$, the result follows.
\end{proof}


\section{First results on Graded Bimodules}\label{firstresults}

In this section, we present some generalizations of the well-known results of the module's theory over algebras.

Let $\mathsf{M}$ be a $\mathsf{G}$-graded left $\mathfrak{A}$-module. Considering $\widetilde{\mathfrak{A}}$ an another $\mathsf{G}$-graded algebra, we have that $\mathsf{M}$ is a right $\widetilde{\mathfrak{A}}$-module with the trivial product, i.e. $ma=0$ for any $m\in\mathsf{M}$ and $a\in\widetilde{\mathfrak{A}}$, and hence, $\mathsf{M}$ is a $\mathsf{G}$-graded $(\mathfrak{A},\widetilde{\mathfrak{A}})$-bimodule naturally. Analogously, we can assume that a right $\widetilde{\mathfrak{A}}$-module is also a $(\mathfrak{A},\widetilde{\mathfrak{A}})$-bimodule. Hence, we can prove for bimodules some results for left and right modules. We show below some results for bimodules that are originally for left (or right) modules, whose proofs are similar to the proofs for one-sided cases. Therefore, all issues of this sections are inspired by known results of the theory of one-sided modules.

\begin{lemma}[Isomorphism Theorems for Graded Bimodules]\label{1.02}
	Let $\mathsf{G}$ be a group, $\mathfrak{A}$ and $\widetilde{\mathfrak{A}}$ two $\mathsf{G}$-graded algebras, and $\mathsf{M}$ a $\mathsf{G}$-graded $(\mathfrak{A},\widetilde{\mathfrak{A}})$-bimodule.
	\begin{itemize}
		\item[i)] If $\mathsf{M}'$ is a $\mathsf{G}$-graded $(\mathfrak{A},\widetilde{\mathfrak{A}})$-bimodule and $\psi: \mathsf{M}\longrightarrow \mathsf{M}'$ is a homogeneous homomorphism of degree $h$, then 
		$\displaystyle\frac{\mathsf{M}}{\mathsf{ker}(\psi)} \cong_{\mathsf{G}} \mathsf{im}(\psi)$ 
		as $\mathsf{G}$-graded $(\mathfrak{A},\widetilde{\mathfrak{A}})$-bimodules;
		\item[ii)] If $\mathsf{N}$ and $\mathsf{W}$ are $\mathsf{G}$-graded subbimodules of $\mathsf{M}$, then 
		$\displaystyle\frac{\mathsf{N}+\mathsf{W}}{\mathsf{W}}\cong_{\mathsf{G}} \frac{\mathsf{N}}{\mathsf{N}\cap \mathsf{W}}$ 
		as $\mathsf{G}$-graded $(\mathfrak{A},\widetilde{\mathfrak{A}})$-bimodules. In addition, if $\mathsf{N}\subseteq \mathsf{W}$, then 
	$	\displaystyle\frac{\mathsf{M}}{\mathsf{W}}\cong_{\mathsf{G}} \frac{\mathsf{M}/{\mathsf{N}}}{{\mathsf{W}}/{\mathsf{N}}}$ 
		as $\mathsf{G}$-graded $(\mathfrak{A},\widetilde{\mathfrak{A}})$-bimodules.
	\end{itemize}  
\end{lemma}
\begin{proof}
The proof is adapted from Theorems 6.19, 6.20 and 6.21, \cite{Rotm10}, p.406.
\end{proof}

\begin{lemma}[Correspondence Theorem for Graded Bimodules]\label{1.04}
	Let $\mathsf{G}$ be a group, $\mathfrak{A}$ and $\widetilde{\mathfrak{A}}$ two $\mathsf{G}$-graded algebras and $\mathsf{M}$ a $\mathsf{G}$-graded $(\mathfrak{A},\widetilde{\mathfrak{A}})$-bimodule. Suppose $\mathsf{N}$ is a graded subbimodule of $\mathsf{M}$. Then any graded subbimodule of the quotient bimodule $\mathsf{M}/\mathsf{N}$ is of the form $\mathsf{P}/\mathsf{N}=\{x+\mathsf{N}: x\in \mathsf{P}\}$, where $\mathsf{P}$ is a graded subbimodule of $\mathsf{M}$ such that $\mathsf{N}\subset \mathsf{P}\subset \mathsf{M}$. The correspondence between graded subbimodules of $\mathsf{M}/\mathsf{N}$ and graded subbimodules of $\mathsf{M}$ which contain $\mathsf{N}$ is a bijection.
\end{lemma}
\begin{proof}
The proof is adapted from Theorem 6.22, \cite{Rotm10}, p.407.
\end{proof}

\begin{lemma}\label{1.11}
	Let $\mathsf{G}$ be a group, $\mathfrak{A}$ and $\widetilde{\mathfrak{A}}$ two $\mathsf{G}$-graded algebras and $\mathsf{M}$ a $\mathsf{G}$-graded $(\mathfrak{A},\widetilde{\mathfrak{A}})$-bimodule. Let $\mathsf{N}$ be a graded subbimodule of $\mathsf{M}$. Then $\mathsf{M}$ is weak $\mathsf{G}$-artinian (resp. weak $\mathsf{G}$-noetherian) iff $\mathsf{N}$ and $\mathsf{M}/\mathsf{N}$ are weak $\mathsf{G}$-artinian (resp. weak $\mathsf{G}$-noetherian).
\end{lemma}
\begin{proof}
The proof is adapted from Theorem 2, \cite{Ribe69}, p.41.
\end{proof}

Using the ideas of the previous lemma and Lemma \ref{1.04}, we can build a descending (or an ascending) chain of graded subbimodules with a special property. Consider a $\mathsf{G}$-graded $(\mathfrak{A},\widetilde{\mathfrak{A}})$-bimodule $\mathsf{M}$. Suppose $\mathsf{M}$ is weak $\mathsf{G}$-noetherian (resp. weak $\mathsf{G}$-artinian). By the previous lemma, given a graded subbimodule $\mathsf{N}$ of $\mathsf{M}$, we have that $\mathsf{N}$ and $\mathsf{M}/\mathsf{N}$ are weak $\mathsf{G}$-noetherian (resp. weak $\mathsf{G}$-artinian). If $\mathsf{M}$ is not $\mathsf{G}$-irreducible, so there exists a $\mathsf{G}$-maximal (resp. $\mathsf{G}$-irreducible) subbimodule $\mathsf{N}_1$ in $\mathsf{M}$. If $\mathsf{N}_1$ (resp. $\mathsf{M}/\mathsf{N}_1$) is $\mathsf{G}$-irreducible, then $\mathsf{M}=\mathsf{N}_0\supsetneq \mathsf{N}_1\supsetneq \mathsf{N}_2=\{0\}$ (resp. $\{0\}=\mathsf{N}_0\subsetneq \mathsf{N}_1\subsetneq \mathsf{N}_2=\mathsf{M}$) with $\mathsf{N}_{i+1}$ is $\mathsf{G}$-maximal in $\mathsf{N}_i$ (resp. $\mathsf{N}_{i+1}/\mathsf{N}_i$ is $\mathsf{G}$-irreducible), for $i=0,1$. Otherwise, suppose that $\mathsf{N}_1$ (resp. $\mathsf{M}/\mathsf{N}_1$) is not $\mathsf{G}$-irreducible, and hence, there exists a nonzero graded subbimodule $\mathsf{N}_2$ of $\mathsf{M}$ such that $\mathsf{N}_2\subsetneq \mathsf{N}_1$ and $\mathsf{N}_2$ is $\mathsf{G}$-maximal $\mathsf{N}_1$ (resp. $\mathsf{N}_2\supsetneq \mathsf{N}_1$ and $\mathsf{N}_2/\mathsf{N}_1$ is $\mathsf{G}$-irreducible). So, we obtain the chain $\mathsf{M}=\mathsf{N}_0\supsetneq \mathsf{N}_1\supsetneq \mathsf{N}_2\supsetneq \mathsf{N}_3=\{0\}$ (resp. $\{0\}=\mathsf{N}_0\subsetneq \mathsf{N}_1\subsetneq \mathsf{N}_2\subsetneq \mathsf{N}_3=\mathsf{M}$). If $\mathsf{N}_2$ (resp. $\mathsf{M}/\mathsf{N}_2$) is $\mathsf{G}$-irreducible, it follows that $\mathsf{N}_{i+1}$ is $\mathsf{G}$-maximal in $\mathsf{N}_i$ (resp. $\mathsf{N}_{i+1}/\mathsf{N}_i$ is $\mathsf{G}$-irreducible), for $i=0,1,2$. Otherwise, using this process inductively we must obtain a descending chain (resp. an ascending chain) of graded subbimodules
\begin{equation}\label{1.16}
\mathsf{M}=\mathsf{N}_0\supsetneq \mathsf{N}_1\supsetneq \mathsf{N}_2\supsetneq \cdots\supsetneq \{0\}
\quad (\mbox{resp. } \{0\}=\mathsf{N}_0\subsetneq \mathsf{N}_1\subsetneq \mathsf{N}_2\subsetneq \cdots\subsetneq \mathsf{M})
\ ,
\end{equation}
such that $\mathsf{N}_{i+1}$ is maximal in $\mathsf{N}_i$ (resp. $\mathsf{N}_{i+1}/\mathsf{N}_i$ is irreducible), for $i=0,1,2, \dots$. Notice that we use Lemma \ref{1.04} to ensure that if $\mathsf{M}/\mathsf{N}_i$ is not irreducible, then there exists a graded subbimodule $\mathsf{N}_{i+1}\supsetneq \mathsf{N}_i$ such that $\mathsf{N}_{i+1}/\mathsf{N}_i$ is irreducible.

Now, again by Lemma \ref{1.11}, it is easy to show that the finite direct sum (internal or external) of weak $\mathsf{G}$-noetherian (resp. weak $\mathsf{G}$-artinian) $(\mathfrak{A},\widetilde{\mathfrak{A}})$-bimodules is a weak $\mathsf{G}$-noetherian (resp. weak $\mathsf{G}$-artinian) $(\mathfrak{A},\widetilde{\mathfrak{A}})$-bimodule. The proof of this fact can be made by induction over the number of summands and recursively applying the Lemma \ref{1.11}.

Notice that given a homogeneous homomorphism of graded $(\mathfrak{A},\widetilde{\mathfrak{A}})$-bimodules $\varphi: \mathsf{M}_1\rightarrow\mathsf{M}_2$, if $\mathsf{M}_1$ is weak $\mathsf{G}$-noetherian (resp. weak $\mathsf{G}$-artinian), then $\mathsf{im}(\varphi)$ and $\mathsf{ker}(\varphi)$ are weak $\mathsf{G}$-noetherian (resp. weak $\mathsf{G}$-artinian) subbimodules of $\mathsf{M}_2$ and $\mathsf{M}_1$, respectively, since $\mathsf{ker}(\varphi)$ is a graded subbimodule of $\mathsf{M}_1$ (see Lemma \ref{1.11}) and $\mathsf{im}(\varphi) \cong_{\mathsf{G}}{\mathsf{M}_1}/{\mathsf{ker}(\varphi)}$ (see Lemma \ref{1.02}) is a graded subbimodule of $\mathsf{M}_2$.

\begin{lemma}\label{1.57}
	Let $\mathfrak{A}$ and $\widetilde{\mathfrak{A}}$ be two $\mathsf{G}$-graded algebras, and $\mathsf{M}=\mathsf{N}+\mathsf{W}$ a $\mathsf{G}$-graded $(\mathfrak{A},\widetilde{\mathfrak{A}})$-bimodule. If $\mathsf{N}$ and $\mathsf{W}$ are weak $\mathsf{G}$-artinian (resp. weak $\mathsf{G}$-noetherian) subbimodules of $\mathsf{M}$, then $\mathsf{M}$ is weak $\mathsf{G}$-artinian (resp. weak $\mathsf{G}$-noetherian).
\end{lemma}
\begin{proof}
By Lemma \ref{1.02}, item (ii), we have that 
	\begin{equation}\nonumber
		\displaystyle\frac{\mathsf{M}}{\mathsf{W}}\cong_{\mathsf{G}}\frac{\mathsf{N}+\mathsf{W}}{\mathsf{W}}\cong_{\mathsf{G}} \frac{\mathsf{N}}{\mathsf{N}\cap \mathsf{W}} \ .
	\end{equation}
Since $\mathsf{N}$ is weak $\mathsf{G}$-artinian (resp. weak $\mathsf{G}$-noetherian), it follows from Lemma \ref{1.11} that $\displaystyle\frac{\mathsf{N}}{\mathsf{N}\cap \mathsf{W}}$ is weak $\mathsf{G}$-artinian (resp. weak $\mathsf{G}$-noetherian). Again by Lemma \ref{1.11}, it follows that $\mathsf{M}$ is weak $\mathsf{G}$-artinian (resp. weak $\mathsf{G}$-noetherian), because $\mathsf{W}$ and $\displaystyle\frac{\mathsf{M}}{\mathsf{W}}\cong_{\mathsf{G}} \frac{\mathsf{N}}{\mathsf{N}\cap \mathsf{W}}$ are weak $\mathsf{G}$-artinian (resp. weak $\mathsf{G}$-noetherian).
\end{proof}

\begin{lemma}\label{1.46}
Let $\mathfrak{A}$ and $\widetilde{\mathfrak{A}}$ be two $\mathsf{G}$-graded algebras, and $\mathsf{M}$ a $\mathsf{G}$-graded $(\mathfrak{A},\widetilde{\mathfrak{A}})$-bimodule. Then $\mathsf{M}$ is weak $\mathsf{G}$-noetherian (resp. weak $\mathsf{G}$-artinian) iff every non-empty family $S$ of graded subbimodules of $\mathsf{M}$ contains a maximal (resp. a minimal) graded subbimodule in the family, that is, a graded subbimodule $\mathsf{P}\in S$ such that if $\mathsf{N}\in S$ and $\mathsf{N}\supset \mathsf{P}$ (resp. $\mathsf{N}\subset \mathsf{P}$), then $\mathsf{N}=\mathsf{P}$.
\end{lemma}
\begin{proof}
The proof is adapted from Proposition 5.33, \cite{Rotm10}, p.312.
\end{proof}

By the previous lemma, it is clear that any weak $\mathsf{G}$-noetherian (resp. weak $\mathsf{G}$-artinian) $(\mathfrak{A},\widetilde{\mathfrak{A}})$-bimodule which is not irreducible has some maximal (resp. minimal) proper graded subbimodule.

\begin{lemma}\label{1.45}
Let $\mathfrak{A}$ and $\widetilde{\mathfrak{A}}$ be two $\mathsf{G}$-graded algebras, and $\mathsf{M}$ a $\mathsf{G}$-graded $(\mathfrak{A},\widetilde{\mathfrak{A}})$-bimodule. Then $\mathsf{M}$ is weak $\mathsf{G}$-noetherian iff every $\mathsf{G}$-graded subbimodule of $\mathsf{M}$ is finitely generated.
\end{lemma}
\begin{proof}
	First, suppose that $\mathsf{M}$ is weak $\mathsf{G}$-noetherian. Let $\mathsf{N}$ be any nonzero graded $\mathsf{G}$-subbimodule of $\mathsf{M}$. Consider the family $S_\mathsf{N}$ of all the finitely generated graded subbimodules of $\mathsf{N}$, which is obviously non empty. Since $\mathsf{N}$ is weak $\mathsf{G}$-noetherian (by Lemma \ref{1.46}), there is $\mathsf{W}\in S_\mathsf{N}$ maximal in $S_\mathsf{N}$, i.e. there is no $\widetilde{\mathsf{W}}\in S_\mathsf{N}$ such that $\mathsf{W}\subsetneq \widetilde{\mathsf{W}}$. Take any homogeneous element $x\in\mathsf{N}$, and consider $\widetilde{\mathsf{W}}=\mathsf{W}+_{\mathfrak{A}}\{x\}_{\widetilde{\mathfrak{A}}}$. Note that $\widetilde{\mathsf{W}}$ is finitely generated, and hence, $\widetilde{\mathsf{W}}\in S_\mathsf{N}$. Because $\mathsf{W}$ is maximal in $S_\mathsf{N}$, we conclude that $\mathsf{W}=\widetilde{\mathsf{W}}$, and thus $x\in \mathsf{W}$ for any $x\in\mathsf{N}$. Therefore, $\mathsf{N}$ is finitely generated.
	
	Now, assume that every $\mathsf{G}$-graded subbimodule of $\mathsf{M}$ is finitely generated. Consider an ascending chain of $\mathsf{G}$-subbimodules of $\mathsf{M}$ given by
	\begin{equation}\nonumber
	\mathsf{N}_1\subseteq\mathsf{N}_2\subseteq \cdots\subseteq\mathsf{N}_k\subseteq \cdots .
	\end{equation}
	Put $\mathsf{N}=\bigcup_{k\in\mathbb{N}}\mathsf{N}_k$, which is a $\mathsf{G}$-subbimodule of $\mathsf{M}$. By hypothesis, $\mathsf{N}$ is finitely generated, and hence, there exist $x_1,x_2, \dots, x_s\in\mathsf{N}$, homogeneous elements, such that $\mathsf{N}=_{\mathfrak{A}}\{x_1,x_2, \dots, x_s\}_{\widetilde{\mathfrak{A}}}$. Let $k_1,k_2, \dots, k_s\in\mathbb{N}$ such that $x_i\in\mathsf{N}_{k_i}$ for all $i=1,2,\dots,k$. Put $k_0=\mathsf{max}\{k_1,k_2, \dots, k_s\}$. We have that $x_1,x_2, \dots, x_s\in\mathsf{N}_{k_0}$, and thus $\mathsf{N}\subseteq\mathsf{N}_{k_0}$. It follows that $\mathsf{N}_l=\mathsf{N}_{k_0}$ for all $l\geq k_0$. Therefore, $\mathsf{M}$ is weak $\mathsf{G}$-noetherian.
\end{proof}

\begin{example}
Let $\mathbb{R}$ be the real field, and $\mathfrak{A}=\mathbb{R}[x_1,x_2,\dots]$ a polynomial algebra (with unit) in infinitely many (commutative) variables. We have that $\mathfrak{A}$ is a $\mathbb{Z}$-graded $\mathfrak{A}$-bimodule which is not $\mathbb{Z}$-noetherian, since the subbimodule $_{\mathfrak{A}} \{x_1,x_2,\dots\}_{\mathfrak{A}}$ of $\mathfrak{A}$ is graded and cannot be finitely generated.
\end{example}

We present below a result inspired by the Jordan-H\"{o}lder Theorem (see Theorem 7.3 in \cite{Rotm10}, p.526).

\begin{lemma}\label{1.51}
Let $\mathsf{M}$ be a $\mathsf{G}$-graded $(\mathfrak{A},\widetilde{\mathfrak{A}})$-bimodule, where $\mathfrak{A}$ and $\widetilde{\mathfrak{A}}$ are $\mathsf{G}$-graded algebras. Suppose that $\mathsf{M}$ has a weak graded composition series with length $r$. Then any graded subbimodule $\mathsf{N}$ of $\mathsf{M}$ has a weak graded composition series, and any weak graded composition series of $\mathsf{N}$ has length at most $r$.
%
\end{lemma}
\begin{proof}
Let $\{0\}=\mathsf{M}_0\subsetneq \mathsf{M}_1\subsetneq \cdots\subsetneq \mathsf{M}_r=\mathsf{M}$ be a weak graded composition series of $\mathsf{M}$ such that $r=l(\mathsf{M})$. Take $\mathsf{N}$ a graded subbimodule of $\mathsf{M}$. For each $i=0,1,\dots, r$, consider the graded subbimodule $\mathsf{N}_i=\mathsf{N}\cap\mathsf{M}_i$ of $\mathsf{N}$. Since $\mathsf{M}_{i+1}/\mathsf{M}_{i}$ is irreducible graded, it is easy to see that $\mathsf{N}_{i+1}/\mathsf{N}_{i}\cong_{\mathsf{G}} \mathsf{M}_{i+1}/\mathsf{M}_{i}$ or $\mathsf{N}_{i}=\mathsf{N}_{i+1}$ for all $i=0,1,\dots, r-1$. Hence, $\{\mathsf{N}_i\}_i$ ensures a weak graded composition series of $\mathsf{N}$, whose length $l(\mathsf{N})\leq r$. 


The proof of the second part of this lemma is an adaptation of the proof of Theorem 1 (Jordan-H\"{o}lder), \cite{Ribe69}, p.33.
\end{proof}




The previous lemma is a key tool to prove some results in this work. Next, we present an equivalence between a graded bimodule having a weak graded composition series and satisfying both weak chain conditions.

\begin{lemma}\label{1.41}
Let $\mathfrak{A}$ and $\widetilde{\mathfrak{A}}$ be two $\mathsf{G}$-graded algebras, and $\mathsf{M}$ a $\mathsf{G}$-graded $(\mathfrak{A},\widetilde{\mathfrak{A}})$-bimodule. Then $\mathsf{M}$ has a weak graded composition series iff $\mathsf{M}$ is weak $\mathsf{G}$-noetherian and weak $\mathsf{G}$-artinian.
\end{lemma}
\begin{proof}
First, suppose that $\mathsf{M}$ has a weak graded composition series with length $r$. By Lemma \ref{1.51}, any graded subbimodule $\mathsf{W}$ of $\mathsf{M}$ also has a weak graded composition series with length at most $r$, and hence, $\mathsf{W}$ is finitely generated, and so, by Lemma \ref{1.45}, we have that $\mathsf{M}$ is weak $\mathsf{G}$-noetherian. Now, let us show that $\mathsf{M}$ is weak $\mathsf{G}$-artinian. Suppose that there exists a descending chain $\mathsf{N}_1\supsetneq \mathsf{N}_2\supsetneq \mathsf{N}_3\supsetneq \cdots$ of graded subbimodules of $\mathsf{M}$ which is infinite. Since $\mathsf{M}$ has a weak composition series, without loss of generality, we can assume that $\mathsf{N}_{i+1}$ is maximal in $\mathsf{N}_{i}$, for all $i\in\{1,2,3,\dots\}$ (see Lemma 1 (Schreier), \cite{Ribe69}, p.34). 
As any graded subbimodule of $\mathsf{M}$ has a weak graded composition series with length at most $r$, 
let $\{0\}=\widetilde{\mathsf{N}}_{r+s}\subsetneq\widetilde{\mathsf{N}}_{r+s-1}\subsetneq\cdots\subsetneq \widetilde{\mathsf{N}}_{r}=\mathsf{N}_{r}$ be a graded composition series of $\mathsf{N}_{r}$, where $s-1=l(\mathsf{N_r})\leq r$. Hence, we have that 
\begin{equation}\nonumber
\{0\}=\widetilde{\mathsf{N}}_{r+s}\subsetneq\widetilde{\mathsf{N}}_{r+s-1}\subsetneq\cdots\subsetneq \widetilde{\mathsf{N}}_{r}=\mathsf{N}_{r}\subsetneq\mathsf{N}_{r-1}\subsetneq\cdots\subsetneq \mathsf{N}_{2}\subsetneq\mathsf{N}_1
\end{equation}
is a weak graded composition series of $\mathsf{N}_1$, whose length is $r+s-1$. Thus, by Lemma \ref{1.51}, it follows that $r+s-1\leq r$, and hence, $s\leq1$. Consequently, either $\widetilde{\mathsf{N}}_{r+1}=\{0\}$, and so $\mathsf{N}_r$ is irreducible graded, or $\mathsf{N}_r=\{0\}$. Anyway, we conclude that $\mathsf{N}_s=\{0\}$ for all $s>r$. Therefore, $\mathsf{M}$ is weak $\mathsf{G}$-artinian.

%

Reciprocally, assume that $\mathsf{M}$ is weak $\mathsf{G}$-noetherian and weak $\mathsf{G}$-artinian. If $\mathsf{M}$ is not irreducible graded, consider $\mathsf{F}_1$ the family of all the graded subbimodules $\mathsf{N}$ of $\mathsf{M}$, with $\mathsf{N}\neq\mathsf{M}$. Since $\mathsf{M}$ is weak $\mathsf{G}$-noetherian, by Lemma \ref{1.46}, there is $\mathsf{N}_1\in \mathsf{F}_1$ maximal in $\mathsf{F}_1$. Note that $\mathsf{N}_1$ is maximal in $\mathsf{M}$. If $\mathsf{N}_1$ is irreducible graded, $\mathsf{M}=\mathsf{N}_0\supsetneq\mathsf{N}_1\supsetneq\{0\}$ is a weak graded composition series of $\mathsf{M}$. Otherwise, if $\mathsf{N}_1$ is not irreducible graded, consider $\mathsf{F}_2$ the family of all the graded subbimodules $\mathsf{P}$ of $\mathsf{N}_1$, with $\mathsf{P}\neq\mathsf{N}_1$. Again by Lemma \ref{1.46}, there is $\mathsf{N}_2\in \mathsf{F}_2$ maximal in $\mathsf{F}_2$. Note that $\mathsf{N}_2$ is maximal in $\mathsf{N}_1$. Again, if $\mathsf{N}_2$ is irreducible graded, $\mathsf{M}=\mathsf{N}_0\supsetneq\mathsf{N}_1\supsetneq\mathsf{N}_2\supsetneq\{0\}$ is a weak graded composition series of $\mathsf{M}$. Otherwise, the above process has to continue, and so, we can construct a descending chain
\begin{equation}\nonumber
\mathsf{M}=\mathsf{N}_0\supsetneq \mathsf{N}_1 \supsetneq \mathsf{N}_2 \supsetneq \mathsf{N}_3 \supsetneq\cdots 
\end{equation}
of graded subbimodules of $\mathsf{M}$, where $\mathsf{N}_{i+1}$ is maximal in $\mathsf{N}_i$, for all $i=0,1,\dots$. Since $\mathsf{M}$ is weak $\mathsf{G}$-artinian, there is $s_0\in\mathbb{N}$ such that $\mathsf{N}_s=\mathsf{N}_{s_0}$ for all $s\geq s_0$. Hence, we conclude that $\mathsf{N}_s=\{0\}$. Therefore, $\mathsf{M}=\mathsf{N}_0\supsetneq \mathsf{N}_1 \supsetneq \mathsf{N}_2 \supsetneq \mathsf{N}_3 \supsetneq\cdots \supsetneq \mathsf{N}_s=\{0\}$ is a weak graded composition series for $\mathsf{M}$.
\end{proof}
%
An immediate consequence of the previous lemma is that if $\mathsf{M}$ is a $\mathsf{G}$-graded $(\mathfrak{A},\widetilde{\mathfrak{A}})$-bimodule such that $\mathsf{M}$ has a weak graded composition series, then $\mathsf{M}$ is necessarily finitely generated. For the next result, recall that any $\mathsf{G}$-graded $(\mathfrak{A},\widetilde{\mathfrak{A}})$-bimodule can be generated by its homogeneous elements.

\begin{lemma}\label{1.71}
Let $\mathfrak{A}$ and $\widetilde{\mathfrak{A}}$ be two finite dimensional $\mathsf{G}$-graded $\mathbb{F}$-algebras, and $\mathsf{M}$ a $\mathsf{G}$-graded $(\mathfrak{A},\widetilde{\mathfrak{A}})$-bimodule. If $\mathsf{M}$ is finitely generated, then $\mathsf{M}$ is weak $\mathsf{G}$-artinian and weak $\mathsf{G}$-noetherian.
\end{lemma}
\begin{proof}
Suppose that $S=\{x_1,\dots, x_n\}$ is a subset of homogeneous elements of $\mathsf{M}$ such that $\mathsf{M}=_{\mathfrak{A}}S_{\widetilde{\mathfrak{A}}}$. Since $_{\mathfrak{A}}S_{\widetilde{\mathfrak{A}}}=\sum_{i=1}^n {_{\mathfrak{A}}}\{x_i\}_{\widetilde{\mathfrak{A}}}$, to prove that $\mathsf{M}$ is weak $\mathsf{G}$-artinian and weak $\mathsf{G}$-noetherian is enough to prove that each $_{\mathfrak{A}}\{x_i\}_{\widetilde{\mathfrak{A}}}$ is weak $\mathsf{G}$-artinian and weak $\mathsf{G}$-noetherian (see Lemma \ref{1.57}). Fixed any $x\in S$, let us show that $_{\mathfrak{A}}\{x\}_{\widetilde{\mathfrak{A}}}$ is weak $\mathsf{G}$-artinian and weak $\mathsf{G}$-noetherian.
 
First, recall that $_{\mathfrak{A}}\{x\}_{\widetilde{\mathfrak{A}}}=\mathsf{span}_{\mathbb{F}}\{x\}+\mathfrak{A}x+x\widetilde{\mathfrak{A}}+\mathfrak{A}x\widetilde{\mathfrak{A}}$. Now, since $\mathfrak{A}$ and $\widetilde{\mathfrak{A}}$ are finite-dimensional $\mathbb{F}$-algebras, consider nonzero homogeneous elements $a_1,\dots,a_r\in\mathfrak{A}$ and $b_1,\dots,b_s\in\widetilde{\mathfrak{A}}$ such that $\mathfrak{A}=\mathsf{span}_\mathbb{F}\{a_1,\dots,a_r\}$ and $\widetilde{\mathfrak{A}}=\mathsf{span}_\mathbb{F}\{b_1,\dots,b_s\}$. It is easy to see that $\mathfrak{A}x=\mathsf{span}_\mathbb{F}\{a_1x,\dots,a_rx\}$ and $x\widetilde{\mathfrak{A}}=\mathsf{span}_\mathbb{F}\{xb_1,\dots,xb_s\}$, and hence, $\mathfrak{A}x$ and $x\widetilde{\mathfrak{A}}$ have finite dimensions. Finally, given any $m\in\mathfrak{A}x\widetilde{\mathfrak{A}}$, there exist $\hat{a}_1,\dots,\hat{a}_t\in\mathfrak{A}$ and $\hat{b}_1,\dots,\hat{b}_t\in\widetilde{\mathfrak{A}}$ such that $m=\hat{a}_1 x\hat{b}_1 + \cdots+\hat{a}_t x \hat{b}_t$. Let $\lambda_{11},\dots,\lambda_{tr},\gamma_{11},\dots,\gamma_{ts}\in\mathbb{F}$ such that $\hat{a}_i=\sum_{j=1}^r\lambda_{ij}a_j$ and $\hat{b}_l=\sum_{j=1}^s\gamma_{lj}b_j$, for $i,l=1,\dots,t$. Hence, 
\begin{equation}\nonumber
m=\sum_{i=1}^t \hat{a}_i x\hat{b}_i=\sum_{i=1}^t \left(\left(\sum_{j=1}^r\lambda_{ij}a_j\right) x\left(\sum_{l=1}^s\gamma_{ij}b_l\right)\right)=\sum_{i,j,l=1}^{t,r,s} \lambda_{ij}\gamma_{il}a_j xb_l \ ,
\end{equation}
and so, $m\in\mathsf{span}_\mathbb{F}\{a_j xb_l: j=1,\dots,r, \mbox{ and } l=1,\dots,s\}$. Consequently, $\mathfrak{A}x\widetilde{\mathfrak{A}}=\mathsf{span}_\mathbb{F}\{a_j xb_l: j=1,\dots,r, \mbox{ and } l=1,\dots,s\}$, and thus, $\mathfrak{A}x\widetilde{\mathfrak{A}}$ has a finite dimension. From this, it follows that 
\begin{equation}\nonumber
\mathsf{dim}_{\mathbb{F}}\left(_{\mathfrak{A}}\{x\}_{\widetilde{\mathfrak{A}}}\right)=\mathsf{dim}_{\mathbb{F}}\left(\mathsf{span}_{\mathbb{F}}\{x\}+\mathfrak{A}x+x\widetilde{\mathfrak{A}}+\mathfrak{A} x \widetilde{\mathfrak{A}}\right)\leq 1+r+s+rs<\infty \,
\end{equation}
i.e. $_{\mathfrak{A}}\{x\}_{\widetilde{\mathfrak{A}}}$ has a finite dimension. Therefore, we conclude that $_{\mathfrak{A}}\{x\}_{\widetilde{\mathfrak{A}}}$ is weak $\mathsf{G}$-artinian and weak $\mathsf{G}$-noetherian, and the result follows.
\end{proof}

By the previous result, given $\mathfrak{A}$ and $\widetilde{\mathfrak{A}}$  finite dimensional $\mathsf{G}$-graded $\mathbb{F}$-algebras and $\mathsf{M}$ a $\mathsf{G}$-graded $(\mathfrak{A},\widetilde{\mathfrak{A}})$-bimodule, for any finite set $S$ of homogeneous elements of $\mathsf{M}$, the graded subbimodules $\mathfrak{A}S\widetilde{\mathfrak{A}}$ and $\mathfrak{A}S+S\widetilde{\mathfrak{A}}+\mathfrak{A}S\widetilde{\mathfrak{A}}$ of $\mathsf{M}$ are weak $\mathsf{G}$-artinian and weak $\mathsf{G}$-noetherian.

In \cite{Vamo78}, Theorem 11, P. V\'amos proved that, if $\mathbb{K}$ is an extension field of $\mathbb{F}$ such that $\mathbb{K}$ is not finitely generated over $\mathbb{F}$, then the tensorial algebra $\mathbb{K}\otimes_\mathbb{F} \mathbb{K}$ is not noetherian. Therefore, Lemma \ref{1.71} is not true when the condition ``$\mathfrak{A}$ and $\widetilde{\mathfrak{A}}$ be two finite dimensional $\mathsf{G}$-graded $\mathbb{F}$-algebras'' is not assumed.

\begin{theorem}\label{1.17}
	Let $\mathsf{G}$ be a group, $\mathfrak{A}$ and $\widetilde{\mathfrak{A}}$ two finite dimensional $\mathsf{G}$-graded $\mathbb{F}$-algebras, and $\mathsf{M}$ a $\mathsf{G}$-graded $(\mathfrak{A},\widetilde{\mathfrak{A}})$-bimodule. Then the following conditions are equivalent:
	\begin{enumerate}
\item[i)] $\mathsf{M}$ is finitely generated (as $(\mathfrak{A},\widetilde{\mathfrak{A}})$-bimodule);

\item[ii)] $\mathsf{M}$ is weak $\mathsf{G}$-noetherian;

\item[iii)] $\mathsf{M}$ is weak $\mathsf{G}$-artinian and weak $\mathsf{G}$-noetherian;

\item[iv)] $\mathsf{M}$ has a weak graded composition series.
	\end{enumerate}
\end{theorem}
\begin{proof}
First, by Lemmas \ref{1.45} and \ref{1.71}, the items (i), (ii) and (iii) are equivalent. Already the equivalence of (iii) and (iv) follows from Lemma \ref{1.41}.
\end{proof}

\begin{corollary}\label{1.42}
Let $\mathsf{G}$ be a group, $\mathsf{H}_1$ and $\mathsf{H}_2$ two finite subgroups of $\mathsf{G}$, and $\mathbb{F}^{\sigma_1}[\mathsf{H}_1]$ and $\mathbb{F}^{\sigma_2}[\mathsf{H}_2]$ two twisted group algebras, with $\sigma_i\in\mathsf{Z}^2(\mathsf{H}_i,\mathbb{F}^*)$. For each $i=1,2$, consider $\mathfrak{B}_i=M_{n_i}(\mathbb{F}^{\sigma_i}[\mathsf{H}_i])$ the algebra of $n_i\times n_i$ matrices over $\mathbb{F}^{\sigma_i}[\mathsf{H}_i]$ with a canonical elementary $\mathsf{G}$-grading. Any finitely generated $\mathsf{G}$-graded $(\mathfrak{B}_1,\mathfrak{B}_2)$-bimodule satisfies both weak chain conditions.
\end{corollary}
\begin{proof}
Indeed, for each $i=1,2$, we have that $\mathsf{dim}(\mathfrak{B}_i)=n_i^2|\mathsf{H}_i|<\infty$, since $\mathsf{H}_i$ is finite, and so, $\mathfrak{B}_1$ and $\mathfrak{B}_2$ have finite dimensions. By Theorem \ref{1.17}, the affirmation is immediate.
\end{proof}

\subsection{A decomposition for graded bimodules of finite dimension}\label{generlemagiam}
 We say that a left (resp. right) $\mathfrak{A}$-module $\mathsf{M}$ is a {\bf $0$-left $\mathfrak{A}$-module} (resp. {\bf $0$-right $\mathfrak{A}$-module}) if $\mathfrak{A}\mathsf{M}=\{0\}$ (resp. $\mathsf{M}\mathfrak{A}=\{0\}$); $\mathsf{M}$ is a left (resp. right) {\bf faithful} if $a \mathsf{M}=\{0\}$ (resp. $\mathsf{M}a=\{0\}$) to imply $a=0$, where $a\in\mathfrak{A}$. 
%
It is important to note that, for any $a=\sum_{g\in\mathsf{G}}a_g$ in $\mathfrak{A}$, ``$a\mathsf{M}=\{0\}$ implies $a=0$'' iff ``$a_g\mathsf{M}=\{0\}$ implies $a_g=0$, for any $g\in\mathsf{G}$''. The  ``right'' case is analogous. 

\begin{theorem}\label{1.27}
Let $\mathsf{G}$ be a group, $\mathfrak{A}$ and $\widetilde{\mathfrak{A}}$ two unitary algebras with $\mathsf{G}$-gradings, and $\mathsf{M}$ a $\mathsf{G}$-graded $(\mathfrak{A},\widetilde{\mathfrak{A}})$-bimodule, not necessarily unitary. If $\mathsf{M}$ has a finite dimension (as vector space), then $\mathsf{M}$ can be decomposed as 
	\begin{equation}\nonumber
\mathsf{M}=\mathsf{M}_{00}\oplus \mathsf{M}_{10}\oplus \mathsf{M}_{01}\oplus \mathsf{M}_{11} \ ,
	\end{equation}
where $\mathsf{M}_{ij}$'s are $\mathsf{G}$-graded $(\mathfrak{A},\widetilde{\mathfrak{A}})$-bimodules such that:
	\begin{itemize}
		\item[i)] for $r=0,1$, $\mathsf{M}_{0r}$ is a $0$-left $\mathfrak{A}$-module and $\mathsf{M}_{1r}$ is a unitary left $\mathsf{G}$-graded $\mathfrak{A}$-module;
		\item[ii)] for $s=0,1$, $\mathsf{M}_{s0}$ is a $0$-right $\widetilde{\mathfrak{A}}$-module and $\mathsf{M}_{s1}$ is a unitary right $\mathsf{G}$-graded $\widetilde{\mathfrak{A}}$-module;
		\item[iii)] $\mathsf{M}_{11}$ is a unitary $\mathsf{G}$-graded $(\mathfrak{A},\widetilde{\mathfrak{A}})$-bimodule.
	\end{itemize}
In addition, if $\mathfrak{A}$ (resp. $\widetilde{\mathfrak{A}}$) is simple graded, then $\mathsf{M}_{1r}$ (resp. $\mathsf{M}_{r1}$) is faithful on the left (resp. on the right), for $r=0,1$.
\end{theorem}
\begin{proof}
Let $\mathfrak{i}$ and $\hat{\mathfrak{i}}$ be the units of $\mathfrak{A}$ and $\widetilde{\mathfrak{A}}$, respectively, and consider the applications $L, R:\mathsf{M} \longrightarrow \mathsf{M}$ defined by $L(x)=\mathfrak{i} x$ and $R(x)=x \hat{\mathfrak{i}}$ for all $x, y \in \mathsf{M}$, respectively. Note that $L$ and $R$ are homogeneous homomorphisms of $\mathsf{G}$-graded $(\mathfrak{A},\widetilde{\mathfrak{A}})$-bimodules such that $L^2 = L$ and $R^2 = R$, since $\mathfrak{i}^2=\mathfrak{i}\in\mathfrak{A}_{e}$ and $\hat{\mathfrak{i}}^2=\hat{\mathfrak{i}}\in\widetilde{\mathfrak{A}}_{e}$. Hence, $\mathsf{ker} (L)$, $\mathsf{im} (L)$, $\mathsf{ker} (R)$ and $\mathsf{im} (R)$ are $\mathsf{G}$-graded $(\mathfrak{A},\widetilde{\mathfrak{A}})$-bimodules. On the other hand, we conclude that $0,1\in\mathbb{F}$ are the only eigenvalues of $L$ and $R$. So, we have that $L$ and $R$ are diagonalizable, since their minimal polynomials can be written as the product of linear factors (see \cite{HoffKunz71}, Theorem 6.4.6, p.204). Notice that $R\circ L= L\circ R$, and hence, it follows that $\mathsf{ker} (R)$ and $\mathsf{im} (R)$ are invariant by $L$, and $\mathsf{ker} (L)$ and $\mathsf{im} (L)$ are invariant by $R$. From this, it is easy to check that $V_0^R=\mathsf{ker} (R)$ and $V_0^L=\mathsf{ker} (L)$ are the eigenspaces of $R$ and $L$ associated with $0$, respectively, and $V_1^R=\mathsf{im} (R)$ and $V_1^L=\mathsf{im} (L)$ are the eigenspaces of $R$ and $L$ associated with $1$, respectively.

Put $\mathsf{M}_{00}=V_0^R\cap V_0^L$, $\mathsf{M}_{10}=V_0^R\cap V_1^L$, $\mathsf{M}_{01}=V_1^R\cap V_0^L$ and $\mathsf{M}_{11}=V_1^R\cap V_1^L$. Note that the $\mathsf{M}_{rs}$'s are $\mathsf{G}$-graded $(\mathfrak{A},\widetilde{\mathfrak{A}})$-bimodules. Let us show that  
\begin{equation}\label{1.28}
\mathsf{M}=\mathsf{M}_{00}\oplus \mathsf{M}_{01}\oplus \mathsf{M}_{10}\oplus \mathsf{M}_{11} \ .
\end{equation}

Let $\widetilde{\mathsf{M}}=\sum_{r,s=0,1}\mathsf{M}_{rs}$. We have that
\begin{align*}
\mathsf{M}_{10} &=\{x\in \mathsf{M} : x=\mathfrak{i} x, x\hat{\mathfrak{i}} =0\} \ , \
\mathsf{M}_{01} =\{x\in \mathsf{M} : x=x\hat{\mathfrak{i}}, \mathfrak{i} x=0\} \ , \\
\mathsf{M}_{11} &=\{ x\in \mathsf{M} : x=\mathfrak{i} x\hat{\mathfrak{i}}\} \ , \
\mathsf{M}_{00} =\{x\in \mathsf{M} : \mathfrak{i} x=x\hat{\mathfrak{i}} =0\} \ .
\end{align*}
From this, if $x_{10}+x_{01}+x_{11}+x_{00}=0$, where $x_{rs}\in\mathsf{M}_{rs}$, then
\begin{align*}
0 &=\mathfrak{i}(x_{10}+x_{01}+x_{11}+x_{00})\hat{\mathfrak{i}}= \mathfrak{i} x_{11}\hat{\mathfrak{i}}=x_{11} \ , \\
0 &=\mathfrak{i}(x_{10}+x_{01}+x_{11}+x_{00})=  \mathfrak{i} x_{10} + \mathfrak{i} x_{11}= x_{10} + x_{11} \ , \\
0 &=(x_{10}+x_{01}+x_{11}+x_{00})\hat{\mathfrak{i}}= x_{01}\hat{\mathfrak{i}} + x_{11}\hat{\mathfrak{i}}= x_{01} + x_{11}  \ ,
\end{align*}
and hence, $x_{11}=x_{10}=x_{01}=x_{00}=0$, and thus, $\widetilde{\mathsf{M}}=\bigoplus_{r,s=0,1}\mathsf{M}_{rs}$\ . On the other hand, by Rank-Nullity Theorem (see \cite{HoffKunz71}, Theorem 3.1.2, p.71), it follows that $\mathsf{M}=\mathsf{ker}(R) \oplus \mathsf{im} (R) = \mathsf{ker}(L) \oplus \mathsf{im} (L)$, since $\mathsf{dim}_\mathbb{F} \mathsf{M}<\infty$ and $\mathsf{ker}(R) \cap \mathsf{im} (R) = \mathsf{ker}(L) \cap \mathsf{im} (L)=\{0\}$. Take any $x\in\mathsf{M}$. There exist $x_{0}\in\mathsf{ker}(L)$ and $x_{1}\in \mathsf{im} (L)$ such that $x=x_{0}+x_{1}$, and hence, there exist $x_{00},x_{10}\in\mathsf{ker}(R)$ and $x_{01},x_{11}\in \mathsf{im} (R)$ such that $x_{0}=x_{00}+x_{01}$ and $x_{1}=x_{10}+x_{11}$. Notice that $x_{00}, x_{01}\in\mathsf{ker}(L)$, since $0=\mathfrak{i} x_{0}=\mathfrak{i} x_{00}+\mathfrak{i} x_{01}$ and $\mathsf{ker}(L) \cap \mathsf{im} (L)=\{0\}$. Similarly, $x_{10}, x_{11}\in\mathsf{im}(L)$, since $ x_{10}+ x_{11}=x_{1}=\mathfrak{i} x_{1}=\mathfrak{i} x_{10}+\mathfrak{i} x_{11}$, and so $x_{10}-\mathfrak{i} x_{10}=\mathfrak{i} x_{11}- x_{11}$, but $\mathsf{ker}(R) \cap \mathsf{im} (R) = \{0\}$. Therefore, 
	\begin{equation}\nonumber
x=x_{00}+x_{01}+x_{10}+x_{11}\in \mathsf{M}_{00}\oplus \mathsf{M}_{01}\oplus \mathsf{M}_{10}\oplus \mathsf{M}_{11}
\ .
	\end{equation}
This finishes the proof of (\ref{1.28}).

Now, to show that $\mathsf{M}_{rs}$'s satisfy items {\it i)}, {\it ii)} and \textit{iii)}, it is enough to see that, since $\mathsf{M}_{rs}$'s are $\mathsf{G}$-graded $(\mathfrak{A},\widetilde{\mathfrak{A}})$-bimodules and $\mathfrak{i}$ and $\hat{\mathfrak{i}}$ are the units of $\mathfrak{A}$ and $\widetilde{\mathfrak{A}}$, respectively, $a=a\mathfrak{i}=\mathfrak{i}a=\mathfrak{i}a\mathfrak{i}$ and $b=\hat{\mathfrak{i}} b=b\hat{\mathfrak{i}}=\hat{\mathfrak{i}}b\hat{\mathfrak{i}}$ for any $a\in\mathfrak{A}$ and $b\in\widetilde{\mathfrak{A}}$, and so $ax=(a\mathfrak{i})x=a(\mathfrak{i} x)$ and $xb=x(\hat{\mathfrak{i}} b)=(x\hat{\mathfrak{i}})b$ for any $a\in\mathfrak{A}$, $b\in\widetilde{\mathfrak{A}}$ and $x\in\mathsf{M}$. Thus, the items {\it i)}, {\it ii)} and {\it iii)} follow from definition of $\mathsf{M}_{rs}$'s. 

Finally, suppose $\mathfrak{A}$ (resp. $\widetilde{\mathfrak{A}}$) simple graded. To obtain a contradiction, suppose that there exists some nonzero homogeneous element $a_0\in\mathfrak{A}$ (resp. $a_0\in\widetilde{\mathfrak{A}}$) such that $a_0\mathsf{M}_{1r}=\{0\}$ (resp. $\mathsf{M}_{r1}a_0=\{0\}$). By the simplicity of $\mathfrak{A}$ (resp. $\widetilde{\mathfrak{A}}$), we have that $\mathfrak{A}=\mathfrak{A}a_0\mathfrak{A}$ (resp. $\widetilde{\mathfrak{A}}=\widetilde{\mathfrak{A}}a\widetilde{\mathfrak{A}}$), and hence, $\mathfrak{A}\mathsf{M}_{1r}=(\mathfrak{A}a_0\mathfrak{A})\mathsf{M}_{1r}\subseteq(\mathfrak{A}a_0)\mathsf{M}_{1r}=\mathfrak{A}(a_0\mathsf{M}_{1r})=\{0\}$ (resp. $\mathsf{M}_{r1}\widetilde{\mathfrak{A}}=\mathsf{M}_{r1}(\widetilde{\mathfrak{A}}a_0\widetilde{\mathfrak{A}})\subseteq\mathsf{M}_{r1}(a_0\widetilde{\mathfrak{A}})=(\mathsf{M}_{r1}a_0)\widetilde{\mathfrak{A}}=\{0\}$), which contradicts the items \textit{i)} and \textit{ii)} of this theorem. Therefore, we conclude that $\mathsf{M}_{1r}$ (resp. $\mathsf{M}_{r1}$) is faithful on the left (resp. on the right), for $i=0,1$.
 \end{proof}

By above theorem (and its proof), it is possible to prove that if $\mathfrak{A}$ and $\widetilde{\mathfrak{A}}$ are simple graded, then $\mathsf{M}_{11}$ is a ``\textbf{faithful bimodule}'', i.e. given $a\in\mathfrak{A}$ and $b\in\widetilde{\mathfrak{A}}$, $a\mathsf{M}b=\{0\}$ iff $a=0$ or $b=0$.

The next result shows the importance of studying and describing the subbimodule $\mathsf{M}_{11}$ of $\mathsf{M}$.

\begin{corollary}\label{1.31}
Assume the same assumptions of Theorem \ref{1.27}. If $\mathsf{M}$ is a unitary $(\mathfrak{A}, \widetilde{\mathfrak{A}})$-bimodule, then $\mathsf{M}=\mathsf{M}_{11}$.
\end{corollary}
\begin{proof}
Supposing that $\mathsf{M}$ is a unitary bimodule, we have that $x=\mathfrak{i} x= x\hat{\mathfrak{i}}=\mathfrak{i} x\hat{\mathfrak{i}}$ for any $x\in\mathsf{M}$, and thus,
	\begin{equation}\nonumber
\mathsf{J}_{00}\oplus \mathsf{J}_{01}\oplus \mathsf{J}_{10}=\mathfrak{i} \left( \mathsf{J}_{00}\oplus \mathsf{J}_{01}\oplus \mathsf{J}_{10} \right)\hat{\mathfrak{i}} =\{0\}
\ .
	\end{equation}
Therefore, the result follows.
\end{proof}

Motivated by Theorem \ref{1.27} and Corollary \ref{1.31}, in the next sections, let us study unitary bimodules (over (semi)simple algebras).

A particular case of the previous theorem is given when $\mathsf{M}$ is a graded algebra, in special, $\mathsf{M}$ is an ideal of the algebra $\mathfrak{A}$. To finally this section, let us present some results in this line.

\begin{corollary}\label{1.25}
Let $\mathfrak{A}$ be an algebra with a $\mathsf{G}$-grading, $\mathfrak{B}$ a graded subalgebra and $\mathsf{N}$ a graded (two-sided) ideal, both of $\mathfrak{A}$. Suppose that $\mathfrak{B}$ is unitary and $\mathsf{N}$ has a finite dimension. Then $\mathsf{N}=\bigoplus_{i,j=0,1} \mathsf{N}_{ij}$, where the $\mathsf{N}_{ij}$'s are as in Theorem \ref{1.27} and satisfy $\mathsf{N}_{rq}\mathsf{N}_{qs}\subseteq \mathsf{N}_{rs}$ for all $r,p,q,s\in\{0,1\}$, with $\mathsf{N}_{rp}\mathsf{N}_{qs}=\{0\}$ for $p\neq q$.
\end{corollary}
\begin{proof}
Obviously $\mathsf{N}$ is a $\mathsf{G}$-graded $\mathfrak{B}$-bimodule, and so we can write $\mathsf{N}=\bigoplus_{i,j=0,1} \mathsf{N}_{ij}$ as in Theorem \ref{1.27}. Now, being $\mathfrak{i}$ be the unit of $\mathfrak{B}$, since $\mathfrak{A}$ has an associative product, we have that $(x\mathfrak{i})y=x(\mathfrak{i}y)$ for any $x,y\in\mathfrak{A}$, and hence, from the definitions of the $\mathsf{N}_{ij}$'s it follows that
	\begin{equation}\nonumber
		\begin{split}
\mathsf{N}_{00}\mathsf{N}_{11} &=\mathsf{N}_{00}\mathsf{N}_{10}=\mathsf{N}_{01}\mathsf{N}_{00}=\mathsf{N}_{10}\mathsf{N}_{11}=\mathsf{N}_{11}\mathsf{N}_{01}=\mathsf{N}_{11}\mathsf{N}_{00}=\{0\} \ , \\ 
\mathsf{N}_{00}\mathsf{N}_{00} &=\mathsf{span}_\mathbb{F} \{x\cdot y\in \mathsf{N} : \mathfrak{i} x=x\mathfrak{i} =0, \mathfrak{i} y=y\mathfrak{i}=0,\ x,y \in \mathsf{N}\} \subseteq \mathsf{N}_{00} \ ,\\ 
\mathsf{N}_{00}\mathsf{N}_{01} &=\mathsf{span}_\mathbb{F} \{x\cdot y\mathfrak{i} \in \mathsf{N} : \mathfrak{i} x=x\mathfrak{i} =0, \mathfrak{i} y=0, y\mathfrak{i}=y,\ x,y \in \mathsf{N}\} \subseteq \mathsf{N}_{01} \ ,\\ 
\mathsf{N}_{10}\mathsf{N}_{00} &=\mathsf{span}_\mathbb{F} \{\mathfrak{i} x\cdot y\in \mathsf{N} : \mathfrak{i}x=x, x\mathfrak{i} =0, \mathfrak{i} y=y\mathfrak{i} =0,\ x,y \in \mathsf{N}\} \subseteq \mathsf{N}_{10} \ ,\\ 
\mathsf{N}_{10}\mathsf{N}_{01} &=\mathsf{span}_\mathbb{F} \{\mathfrak{i} x\cdot y\mathfrak{i} \in \mathsf{N} : \mathfrak{i}x=x, x\mathfrak{i} =0, \mathfrak{i} y=0, y\mathfrak{i}=y,\ x,y \in \mathsf{N}\} \subseteq \mathsf{N}_{11} \ ,\\ 
\mathsf{N}_{01}\mathsf{N}_{10} &=\mathsf{span}_\mathbb{F} \{x\mathfrak{i}\cdot  \mathfrak{i} y\in \mathsf{N} : \mathfrak{i} x=0, x\mathfrak{i}=x, \mathfrak{i}y=y, y\mathfrak{i} =0,\ x,y \in \mathsf{N}\} \subseteq \mathsf{N}_{00} \ ,\\ 
\mathsf{N}_{01}\mathsf{N}_{11} &=\mathsf{span}_\mathbb{F} \{x\mathfrak{i} \cdot \mathfrak{i}y\mathfrak{i} \in \mathsf{N} : \mathfrak{i} x=0, x\mathfrak{i}=x, \mathfrak{i}y=y\mathfrak{i}=y,\ x,y\in \mathsf{N}\} \subseteq \mathsf{N}_{01} \ , \\
\mathsf{N}_{11}\mathsf{N}_{10} &=\mathsf{span}_\mathbb{F} \{\mathfrak{i} x\mathfrak{i} \cdot \mathfrak{i}y \in \mathsf{N} : \mathfrak{i}x=x\mathfrak{i}=x, \mathfrak{i}y=y, y\mathfrak{i} =0,\ x,y \in \mathsf{N}\} \subseteq \mathsf{N}_{10} \ , \\
\mathsf{N}_{11}\mathsf{N}_{11} &=\mathsf{span}_\mathbb{F} \{\mathfrak{i} x\mathfrak{i} \cdot \mathfrak{i}y\mathfrak{i} \in \mathsf{N} : \mathfrak{i}x=x\mathfrak{i}=x, \mathfrak{i}y=y\mathfrak{i}=y,\ x,y \in \mathsf{N}\} \subseteq \mathsf{N}_{11} \ ,
		\end{split}
	\end{equation}
and this ensures the our affirmation.
\end{proof}

\begin{corollary}\label{1.56}
Under the same hypotheses of Corollary \ref{1.25}, let $\mathfrak{i}$ be the unit of $\mathfrak{B}$. If $\mathfrak{i} \in \mathcal{Z}_\mathfrak{A}(\mathsf{N}_h)$ (the center of $\mathsf{N}_h$ in $\mathfrak{A}$) for some $h\in \mathsf{G}$, then $\mathsf{N}_h=(\mathsf{N}_{00})_h\oplus (\mathsf{N}_{11})_h$. In addition, $\mathfrak{i} \in \mathcal{Z}(\mathsf{N})$ (center of $\mathsf{N}$ in $\mathfrak{A}$) iff $\mathsf{N}=\mathsf{N}_{00}\oplus \mathsf{N}_{11}$. 
\end{corollary}
\begin{proof}
	By Theorem \ref{1.27}, we can write 
	\begin{equation}\nonumber
\mathsf{N}=\bigoplus_{g\in \mathsf{G}} \left((\mathsf{N}_{00})_g\oplus (\mathsf{N}_{01})_g\oplus (\mathsf{N}_{10})_g\oplus (\mathsf{N}_{11})_g\right)=\bigoplus_{\substack{{g\in \mathsf{G}}\\{i,j\in \{0,1\}}}} (\mathsf{N}_{ij})_g \ .
	\end{equation}

Now, suppose $\mathfrak{i}\in \mathcal{Z}_\mathfrak{A}(\mathfrak{A}_h)$ for some $h\in \mathsf{G}$, i.e. $\mathfrak{i} a_h=a_h \mathfrak{i}$ for all $a_h\in \mathfrak{A}_h$. In particular, for any $x_h\in \mathsf{N}_h\subseteq \mathfrak{A}_h$, we have $\mathfrak{i} x_h=x_h \mathfrak{i}$. Since $\mathsf{N}_h=(\mathsf{N}_{00})_h\oplus (\mathsf{N}_{10})_h\oplus (\mathsf{N}_{01})_h\oplus (\mathsf{N}_{11})_h$, where $\mathsf{N}_{00}=\{z\in \mathsf{N} : \mathfrak{i} z=0=z\mathfrak{i}\}$, $\mathsf{N}_{10}=\{x\in \mathsf{N} : \mathfrak{i} x=x, x\mathfrak{i}=0\}$ and $\mathsf{N}_{01}=\{y\in \mathsf{N} : y\mathfrak{i}=y, \mathfrak{i} y=0\}$, it follows that $\mathsf{N}_{10}=\mathsf{N}_{01}=\{0\}$. Therefore, $\mathfrak{i}\in \mathcal{Z}_\mathfrak{A}(\mathfrak{A}_h)$ for $h\in \mathsf{G}$ implies $\mathsf{N}_h=(\mathsf{N}_{00})_h\oplus( \mathsf{N}_{11})_h$. 

Consequently, if $\mathfrak{i}\in\mathcal{Z}(\mathfrak{A})$, it follows that $\mathsf{N}_h=(\mathsf{N}_{00})_h\oplus( \mathsf{N}_{11})_h$ for any $h\in \mathsf{G}$, and so we conclude that 
	\begin{equation}\nonumber
\mathsf{N}=\bigoplus_{g\in \mathsf{G}} \mathsf{N}_g=\bigoplus_{g\in \mathsf{G}} \left((\mathsf{N}_{00})_g\oplus  (\mathsf{N}_{11})_g\right)=\mathsf{N}_{00}\oplus \mathsf{N}_{11}
\ .
	\end{equation}
	
On the other hand, suppose $\mathsf{N}=\mathsf{N}_{00}\oplus \mathsf{N}_{11}$. Take any $z=x+y\in\mathsf{N}$ with $x\in\mathsf{N}_{00}$ and $y\in\mathsf{N}_{11}$. Since $\mathfrak{i}x=0=x\mathfrak{i}$ and $\mathfrak{i}y=y=y\mathfrak{i}$, we have that $\mathfrak{i}z=\mathfrak{i}(x+y)=\mathfrak{i}x+\mathfrak{i}y=x\mathfrak{i}+y\mathfrak{i}=(x+y)\mathfrak{i}=z\mathfrak{i}$, and the result follows.
\end{proof}


\section{Graded Bimodules over Simple Graded Algebras}\label{simplealgebras}



In this section, let us exhibit the main results of this work: a description of irreducible graded bimodules and a (special) decomposition of finitely generated graded bimodules.

Let $\mathsf{G}$ be a group, $\mathsf{H}$ a finite subgroup of $\mathsf{G}$, $\sigma\in\mathsf{Z}^2(\mathsf{H},\mathbb{F}^*)$ a $2$-cocycle, and $\mathfrak{B}=M_n(\mathbb{F}^\sigma[\mathsf{H}])$ with the canonical elementary $\mathsf{G}$-grading defined by an $n$-tuple $(g_1,\dots,g_n)\in\mathsf{G}^n$, i.e. $\mathfrak{B}=\bigoplus_{g\in\mathsf{G}}\mathfrak{B}_g$, where $\mathfrak{B}_g=\mathsf{span}_{\mathbb{F}}\{E_{ij}\eta_h\in\mathfrak{B}: g_i^{-1}hg_j=g\}$. Recall that $\mathfrak{B}$ is unitary with unity $1_\mathfrak{B}=\sigma(e,e)^{-1}\sum_{i=1}^n E_{ii}\eta_e$, and the set $\{E_{ij}\eta_h\in\mathfrak{B}: i,j=1,\dots,n, \mbox{ and } h\in \mathsf{H}\}$ is a homogeneous basis of $\mathfrak{B}$. Let us consider $\mathsf{M}$ a unitary $\mathsf{G}$-graded $\mathfrak{B}$-bimodule. In what follows, let us define elements in $\mathsf{M}$ whose product with (homogeneous) elements of $\mathfrak{B}$ is similar to the product in $\mathfrak{B}$.

Fix a nonzero homogeneous element $w_0\in\mathsf{M}$. Since $\mathsf{M}$ is unitary, it follows that $\eta_e E_{i_0 i_0}w_0 E_{j_0 j_0}\eta_e\neq0$ for some $i_0, j_0\in\{1,\dots,n\}$, and hence, $\eta_g E_{ri_0}w_0E_{j_0s}\eta_h\neq0$ for any $r, s\in\{1,\dots,n\}$ and $g,h\in\ \mathsf{H}$ (because $\eta_e E_{ii}=\sigma(g,g^{-1})^{-1} E_{ij}\eta_g E_{ji}\eta_{g^{-1}}$ for any $g\in\mathsf{H}$ and $i,j=1,\dots,n$). Observe that all elements $\eta_g E_{ri_0}w_0E_{j_0s}\eta_h$'s are homogeneous. Without loss of generality, we can consider the element $\eta_e E_{11}w_0E_{11}\eta_e\neq0$ instead of $w_0$. Given any $g\in \mathsf{H}$ and $i,j\in\{1,\dots,n\}$, define the element 
	\begin{equation}\label{1.76}
m_{ij}^g \coloneqq \sum_{h\in \mathsf{H}} \sigma(h,h^{-1}g)^{-1} \eta_h E_{i1} w_0 E_{1j}\eta_{h^{-1}g} \ ,
	\end{equation}
where each $\eta_g E_{r1}w_0E_{1s}\eta_h\neq0$ for any $r, s\in\{1,\dots,n\}$ and $g,h\in\mathsf{H}$.

When $m_{ij}^g\neq0$ for some $g\in \mathsf{H}$ and $i,j\in\{1,\dots,n\}$, observe that $m_{ij}^g$ has behaviour similar to $E_{ij}\eta_g$ in relation to the product by elements of $\mathfrak{B}$, i.e. for any $r,s\in\{1,\dots,n\}$ and $t\in \mathsf{H}$, we have
	\begin{equation}\label{1.44}
	\begin{split}
E_{rs}\eta_t m_{ij}^g &=\delta_{si}\sigma(t,g) m_{rj}^{tg}
\ \ \mbox{ and } \ \
 m_{ij}^g E_{rs}\eta_t =\delta_{jr}\sigma(g,t) m_{is}^{gt} \ ,
	\end{split}
	\end{equation}
where $\delta_{ij}=\left\{ \begin{array}{ccc}
0 ,& \mbox{if} & i\neq j \\ 
1, & \mbox{if} & i= j
\end{array}\right.$
is the Kronecker delta. In fact, take any $t\in \mathsf{H}$ and $r,s\in\{1,\dots,n\}$. It is obvious that $E_{rs}\eta_h m_{ij}^g=0$ when $s\neq i$, and $m_{ij}^g E_{rs}\eta_h=0$ when $r\neq j$. Now, if $s=i$, we have
	\begin{equation}\nonumber
	\begin{split}
E_{ri}\eta_t m_{ij}^g &=E_{ri}\eta_t \left(\sum_{h\in \mathsf{H}} \sigma(h,h^{-1}g)^{-1} E_{i1} \eta_h w_0 E_{1j}\eta_{h^{-1}g}\right) \\
	&=\sum_{h\in \mathsf{H}} \sigma(t,h) \sigma(h,h^{-1}g)^{-1} E_{r1} \eta_{th} w_0 E_{1j}\eta_{h^{-1}g} \\
	&=\sum_{h\in \mathsf{H}} \sigma(t,h) \sigma(h,(th)^{-1}tg)^{-1} E_{r1} \eta_{th} w_0 E_{1j}\eta_{(th)^{-1}tg} \\
	&=\sigma(t,g) \left(\sum_{h\in \mathsf{H}} \sigma(th,(th)^{-1} tg)^{-1} E_{r1} \eta_{th} w_0 E_{1j}\eta_{(th)^{-1}tg}\right) \\
	&= \sigma(t,g)\ m_{rj}^{tg} \ ,
	\end{split}
	\end{equation}
since $\sigma(t,h)\sigma(th,(th)^{-1} tg)=\sigma(t,g)\sigma(h,(th)^{-1}tg)$ for any $h,g,t\in \mathsf{H}$ (see Definition \ref{1.05}). And if $r=j$, we have
	\begin{equation}\nonumber
	\begin{split}
m_{ij}^g E_{js}\eta_t &=\left(\sum_{h\in \mathsf{H}} \sigma(h,h^{-1}g)^{-1} E_{i1} \eta_h w_0 E_{1j}\eta_{h^{-1}g}\right) E_{js}\eta_t \\
	&=\sum_{h\in \mathsf{H}} \sigma(h^{-1}g,t) \sigma(h,h^{-1}g)^{-1} E_{i1} \eta_{h} w_0 E_{1s}\eta_{h^{-1}gt} \\
	&=\sigma(g,t) \left(\sum_{h\in \mathsf{H}} \sigma(h,h^{-1}gt)^{-1} E_{i1} \eta_{h} w_0 E_{1s}\eta_{h^{-1}gt} \right) \\
	&= \sigma(g,t)\ m_{is}^{gt} \ ,
	\end{split}
	\end{equation}
since $\sigma(h,h^{-1}g)\sigma(g,t)=\sigma(h,h^{-1}gt)\sigma(h^{-1}g,t)$ for any $h,g,t\in \mathsf{H}$ (see Definition \ref{1.05}).

Another peculiarity of $m_{ij}^g$'s is that if $m_{ij}^g\neq0$ for some $i,j\in\{1,\dots,n\}$ and $g\in \mathsf{H}$, then $m_{lj}^{hg}=\sigma(h,g)^{-1}E_{li}\eta_h m_{ij}^g\neq0$ and $m_{il}^{gh}=\sigma(g,h)^{-1} m_{ij}^g E_{jl}\eta_h \neq0$ for any $h\in \mathsf{H}$ and $l\in\{1,\dots,n\}$, since 
\begin{equation}\nonumber
m_{ij}^g=(\sigma(e,e)^{-1}\sigma(h^{-1},h)^{-1}E_{il}\eta_{h^{-1}}E_{li}\eta_{h})m_{ij}^g=(\sigma(e,e)^{-1}\sigma(h^{-1},h)^{-1}E_{il}\eta_{h^{-1}})(E_{li}\eta_h m_{ij}^g),
\end{equation}
 and 
\begin{equation}\nonumber
m_{ij}^g=m_{ij}^g(\sigma(e,e)^{-1}\sigma(h,h^{-1})E_{jl}\eta_{h}E_{lj}\eta_{h^{-1}})=(m_{ij}^g E_{jl}\eta_{h})(\sigma(e,e)^{-1}\sigma(h,h^{-1})E_{lj}\eta_{h^{-1}}).
\end{equation}
 From this, we can deduce that $m_{ij}^g\neq0$ for some $i,j\in\{1,\dots,n\}$ and $g\in \mathsf{H}$ iff $m_{rs}^h\neq0$ for any $r,s\in\{1,\dots,n\}$ and $h\in \mathsf{H}$.

Besides that, it is easy to prove that the element $\mathfrak{i}_{\mathsf{M}}\coloneqq \sigma(e,e)^{-1}\sum_{i=1}^n m_{ii}^e\in\mathsf{M}$ satisfies $b\mathfrak{i}_{\mathsf{M}}=\mathfrak{i}_{\mathsf{M}} b$ for any $b\in \mathfrak{B}$.

Notice that when either $\mathsf{deg}(w_0)\in \mathcal{Z}(\mathsf{G})$ or $\mathsf{H}\subset \mathcal{Z}(\mathsf{G})$, we have that $m_{ij}^g$ is a homogeneous element of $\mathsf{M}$ for any $g\in \mathsf{H}$, and $i,j\in\{1,\dots,n\}$. Particularly, when $\mathsf{deg}(w_0)\in \mathcal{Z}(\mathsf{G})$, it follows that $m_{ij}^g \in\mathsf{M}_{g_i^{-1}gg_j\mathsf{deg}(w_0)}=\mathsf{M}_{\mathsf{deg}(w_0)g_i^{-1}gg_j}$ for any $g\in \mathsf{H}$, $i,j=1,\dots,n$. By these observations, it follows that the linear transformation $\psi: \mathfrak{B}\rightarrow \mathsf{M}$ which extends the map $E_{ij}\eta_g\mapsto m_{ij}^g$ is a homogeneous homomorphism of $\mathfrak{B}$-bimodules of degree $\mathsf{deg}(w_0)$ when $\mathsf{deg}(w_0)\in \mathcal{Z}(\mathsf{G})$.

\begin{remark}
Let $\mathsf{G}$ be a group, $\mathsf{H}$ a finite subgroup of $\mathsf{G}$, $\mathbb{F}$ a field, $\sigma\in\mathsf{Z}^2(\mathsf{H},\mathbb{F}^*)$, and $\mathfrak{B}=M_n(\mathbb{F}^\sigma[\mathsf{H}])$ with a canonical elementary $\mathsf{G}$-grading. Let $\mathsf{M}$ be a unitary $\mathsf{G}$-graded $\mathfrak{B}$-bimodule. Fix a nonzero homogeneous element $m_0 \in\mathsf{M}$, and, as in (\ref{1.76}), define
	\begin{equation}\nonumber
m_{ij}^g=\sum_{h\in \mathsf{H}}\sigma(h,h^{-1}g)^{-1} E_{i1} \eta_h m_0 E_{1j}\eta_{h^{-1}g} \ ,
	\end{equation}
for any $g\in \mathsf{H}$, and $i,j=1,\dots, n$. Recall that $E_{rs}\eta_h m_{ij}^g=\delta_{si}\sigma(h,g)m_{rj}^{hg}$ and $m_{ij}^g E_{rs}\eta_h=\delta_{jr}\sigma(g,h)m_{is}^{gh}$, for any $h,g\in \mathsf{H}$ and $i,j,r,s\in\{1\,\dots,n\}$, where $\delta_{ij}$ is the Kronecker delta; and $m_{ij}^g\neq0$ for some $g\in \mathsf{H}$ and $i,j\in\{1,\dots,n\}$ implies that $m_{rs}^h\neq0$ for any $h\in \mathsf{H}$ and $r,s\in\{1,\dots,n\}$.

Suppose $m_{ij}^g\neq0$ for some $g\in \mathsf{H}$ and $i,j\in\{1,\dots,n\}$, and $\mathbb{F}$ is an algebraically closed field such that $\mathsf{char}(\mathbb{F}) = 0$. Take $\mathsf{N}=\mathsf{span}_\mathbb{F}\{m_{ij}^g: g\in \mathsf{H}, i,j=1,\dots,n \}$. \textbf{Affirmation:} if $\mathsf{deg}(m_0)\in\mathcal{Z}(\mathsf{G})$, then $\mathsf{N}$ is an irreducible $\mathsf{G}$-graded $\mathfrak{B}$-subbimodule of $\mathsf{M}$. Indeed, consider the linear transformation $\psi:\mathfrak{B}\longrightarrow\mathsf{N}$ which extends the map $E_{ij}\eta_g \mapsto m_{ij}^g$. By (\ref{1.44}), it follows that $\psi$ is a homomorphism of $\mathfrak{B}$-bimodules. Notice that $\psi$ is surjective and homogeneous of degree $\mathsf{deg}(m_0)$. Since $\mathfrak{B}$ is an irreducible graded $\mathfrak{B}$-bimodule (see Theorem \ref{teoBahtSehgZaic}), it follows that $\psi$ is injective ($\mathsf{ker}(\psi)$ is a graded subbimodule of $\mathfrak{B}$). We conclude that $\psi$ is a homogeneous isomorphism of $\mathsf{G}$-graded $\mathfrak{B}$-bimodules. Therefore, $\mathsf{N}$ is an irreducible graded $\mathfrak{B}$-bimodule. In particular, if $\mathsf{M}$ is irreducible graded, then $\mathsf{M}=\mathsf{N}$.
\end{remark}

In the above remark, it is important to comment that not always $m_{ij}^g\neq0$ for some $g\in \mathsf{H}$ and $i,j\in\{1,\dots,n\}$. In the next result, let us present other cases of irreducible $\mathsf{G}$-graded $\mathfrak{B}$-bimodules (possibly when $m_{ij}^g=0$ for some $g\in \mathsf{H}$ and $i,j\in\{1,\dots,n\}$), which are not isomorphic to $\mathfrak{B}$ as graded $\mathfrak{B}$-bimodules.

\begin{proposition}\label{1.64}
Let $\mathsf{G}$ be a finite abelian group, $\mathbb{F}$ an algebraically closed field with $\mathsf{char}(\mathbb{F})=0$, and $\mathfrak{B}=\mathbb{F}^\sigma[\mathsf{G}]$ a twisted group algebra. Let $\mathsf{M}$ be an irreducible unitary $\mathsf{G}$-graded $\mathfrak{B}$-bimodule. There exist a nonzero homogeneous element $w_0\in \mathsf{M}$ and a character $\chi:\mathsf{G}\rightarrow \mathbb{F}^*$ satisfying $\mathsf{M}=\mathfrak{B} w_0$ and $\eta_g w_0=\chi(g)^{-1}w_0\eta_g$ for any $g\in\mathsf{G}$.
\end{proposition}
\begin{proof}
Fix a nonzero homogeneous element $m_0 \in\mathsf{M}$, and hence, we have that $\mathsf{M}=\mathfrak{B} m_0 \mathfrak{B}$. Let $\chi_1,\dots,\chi_s$ be all distinct irreducible characters of $\mathsf{G}$ (see Corollary \ref{1.60}). Since $\mathsf{G}$ and $\widehat{\mathsf{G}}=\{\chi_1,\dots,\chi_s\}$ are isomorphic groups (see Proposition \ref{1.68}), we have that $s=|\mathsf{G}|$. For each $i\in\{1,\dots,s\}$, define
	\begin{equation}\nonumber
		w_{\chi_i}=\sum_{h\in \mathsf{G}} \sigma(h,h^{-1})^{-1}\chi_i(h) \eta_h m_0 \eta_{h^{-1}} \ .
	\end{equation}
Let us show that $\mathfrak{B} w_{\chi_i}= w_{\chi_i}\mathfrak{B}$ and $\mathfrak{B} m_0 \mathfrak{B}=\sum_{i=1}^s \mathfrak{B} w_{\chi_i}$, for all $i\in\{1,\dots,s\}$. Observe that $\mathfrak{B} w_{\chi_i}\mathfrak{B}$ is a graded subbimodule of $\mathsf{M}$, since $m_0$ is a homogeneous element, and $\mathsf{G}$ is abelian.

Take any $\eta_t\in \mathfrak{B}$, and $i=1,\dots,s$. We have that 
	\begin{equation}\nonumber
	\begin{split}
		\eta_t w_{\chi_i} &=\eta_t \left(\sum_{h\in \mathsf{G}} \sigma(h,h^{-1})^{-1}\chi_i(h) \eta_h m_0 \eta_{h^{-1}}\right)
	=\sum_{h\in \mathsf{G}} \sigma(t,h) \sigma(h,h^{-1})^{-1} \chi_i(h) \eta_{th} m_0\eta_{h^{-1}} \\
		&=\sum_{h\in \mathsf{G}} \sigma(t,h) \sigma(h,h^{-1})^{-1} \sigma((th)^{-1},t)^{-1}\chi_i(h) \eta_{th} m_0 \eta_{(th)^{-1}}\eta_t \\
		&=\left(\sum_{h\in \mathsf{G}} \sigma((th)^{-1},th)^{-1}\chi_i(h) \eta_{th} m_0 \eta_{(th)^{-1}}\right)\eta_t 
	= \chi_i(t)^{-1} w_{\chi_i} \eta_t \ ,
	\end{split}
\end{equation}
since $\sigma(t,h)\sigma((th)^{-1},th)=\sigma(h,h^{-1})\sigma((th)^{-1},t)$ for any $h,t\in \mathsf{G}$ (it is enough to make $x=(th)^{-1}$, $y=t$ and $z=h$ in Definition \ref{1.05}, and $\sigma(h,h^{-1})=\sigma(h^{-1},h)$ by Remark \ref{1.07}), and $\chi_i(ht)=\chi_i(h)\chi_i(t)$ for any $h,t\in\mathsf{G}$ (see Theorem \ref{1.65}). From this, it follows that $\mathfrak{B} w_{\chi_i}=w_{\chi_i}\mathfrak{B}$ for all $i=1,\dots,s$.

Now, write $\mathsf{G}=\{g_1,\dots,g_s\}$. Since the matrix 
\begin{equation}\nonumber
	[\chi]\coloneqq\begin{bmatrix} \chi_1(g_1) & \chi_1(g_2) & \cdots & \chi_1(g_s)\\
							\chi_2(g_1) & \chi_2(g_2) & \cdots & \chi_2(g_s)\\
							\vdots &\vdots& \ddots &\vdots \\
							\chi_s(g_1) & \chi_s(g_2) & \cdots & \chi_s(g_s)\\
	\end{bmatrix}
\end{equation}
is invertible, with inverse matrix given by 
\begin{equation}\nonumber
	[\chi]^{-1}= |\mathsf{G}|^{-1}\begin{bmatrix} 
		\chi_1(g_1^{-1}) & \chi_2(g_1^{-1}) & \cdots & \chi_s(g_1^{-1})\\
		\chi_1(g_2^{-1}) & \chi_2(g_2^{-1}) & \cdots & \chi_s(g_2^{-1})\\
		\vdots &\vdots& \ddots &\vdots \\
		\chi_1(g_s^{-1}) & \chi_2(g_s^{-1}) & \cdots & \chi_s(g_s^{-1})\\
	\end{bmatrix}
\end{equation}
(this is a consequence of Proposition \ref{1.62}), we have that $\eta_h m_0 \eta_{h^{-1}}\in\sum_{i=1}^s\mathbb{F} w_{\chi_i}$ for any $h\in\mathsf{G}$, since
\begin{equation}\nonumber
	\begin{bmatrix} w_{\chi_1}\\
		w_{\chi_2}\\
		\vdots\\
		w_{\chi_s}\\
	\end{bmatrix}
=
\begin{bmatrix} \chi_1(g_1) & \chi_1(g_2) & \cdots & \chi_1(g_s)\\
		\chi_2(g_1) & \chi_2(g_2) & \cdots & \chi_2(g_s)\\
		\vdots &\vdots& \ddots &\vdots \\
		\chi_s(g_1) & \chi_s(g_2) & \cdots & \chi_s(g_s)\\
	\end{bmatrix}
	\cdot
	\begin{bmatrix} \sigma(g_1,g_1^{-1})^{-1} \eta_{g_1} m_0 \eta_{g_1^{-1}}\\
	\sigma(g_2,g_2^{-1})^{-1} \eta_{g_2} m_0 \eta_{g_2^{-1}}\\
	\vdots\\
	\sigma(g_s,g_s^{-1})^{-1}\eta_{g_s} m_0 \eta_{g_s^{-1}}\\
\end{bmatrix} \ .
\end{equation}
Hence, we can conclude that $\eta_g m_0 \eta_{h}\in\sum_{i=1}^s\mathfrak{B} w_{\chi_i}$ for any $g,h\in\mathsf{G}$, and so $\mathfrak{B} m_0\mathfrak{B}=\sum_{i=1}^s\mathfrak{B} w_{\chi_i}$. Therefore, we conclude that $\mathsf{M}=\mathfrak{B} w_{\chi}$ for some irreducible character $\chi$ of $\mathsf{G}$, since each $\mathfrak{B} w_{\chi_i}$ is a graded subbimodule of $\mathsf{M}$ and $\mathsf{M}$ is irreducible graded.
\end{proof}

Below, let us show that, given a finite dimensional simple $\mathsf{G}$-graded algebra $\mathfrak{A}$, any unitary $\mathsf{G}$-graded $\mathfrak{A}$-bimodule $\mathsf{M}$ which satisfies both weak chain conditions can be written as a finite direct sum of irreducible $\mathsf{G}$-graded $\mathfrak{A}$-subbimodules in the form $\mathfrak{A} w$, where $w\in \mathsf{M}$ and $w\mathfrak{A}=\mathfrak{A} w$.

\begin{theorem}\label{1.03}
Let $\mathsf{G}$ be a group, $\mathsf{H}$ a finite abelian subgroup of $\mathsf{G}$, and $\mathbb{F}$ an algebraically closed field such that $\mathsf{char}(\mathbb{F})=0$. Consider $\mathsf{M}$ a $\mathsf{G}$-graded unitary $\mathfrak{B}$-bimodule, where $\mathfrak{B}=M_n(\mathbb{F}^\sigma[\mathsf{H}])$ with a canonical elementary $\mathsf{G}$-grading and $\sigma\in\mathsf{Z}^2(\mathsf{H},\mathbb{F}^*)$. If $\mathsf{M}$ is irreducible graded, then there exists a nonzero homogeneous element $w\in\mathsf{M}$ satisfying $\mathsf{M}=\mathfrak{B}w$, such that $\mathfrak{B} w= w\mathfrak{B}$.
\end{theorem}
\begin{proof}
	If $n=1$, then the result follows from Proposition \ref{1.64}.

	Assume that $n>1$. Let $\chi_1,\dots,\chi_s$ be all the distinct irreducible characters of $\mathsf{H}$, where $\mathsf{H}\cong\widehat{\mathsf{H}}=\{\chi_1,\dots,\chi_s\}$. Fix a nonzero homogeneous element $m_0 \in\mathsf{M}$ and $i=1,\dots, s$. From the discussion in the third paragraph of this chapter (text above of equality (\ref{1.76})), we can assume, without loss of generality, that $m_0=\eta_e E_{11}\hat{m}E_{11}\eta_e$ for some nonzero $\hat{m}\in\mathsf{M}$, and so $E_{p1}\eta_g w_0\eta_h E_{1q}\neq0$ for any $p,q\in\{1,\dots,n\}$ and $g,h\in\mathsf{H}$. Now, inspired by the definition of $w_{\chi_i}$ in the proof of Proposition \ref{1.64}, define
	\begin{equation}\nonumber
\hat{w}_{\chi_i}=\sum_{j=1}^s \sum_{h\in \mathsf{H}} \sigma(h,h^{-1})^{-1}\chi_i(h) E_{j1} \eta_h m_0 \eta_{h^{-1}}E_{1j} \ .
	\end{equation}
As in the proof of Proposition \ref{1.64}, for any $E_{pq}\eta_t\in\mathfrak{B}$, it is not difficult to see that
		\begin{equation}\nonumber
	\begin{split}
	E_{pq}\eta_t \hat{w}_{\chi_i} &=\sum_{h\in \mathsf{H}} \sigma(t,h) \sigma(h,h^{-1})^{-1}\chi_i(h) E_{p1} \eta_{th} m_0 \eta_{h^{-1}}E_{1q} \\
		&=\sum_{h\in \mathsf{H}} \sigma(t,h) \sigma(h,h^{-1})^{-1}\sigma((th)^{-1},t)^{-1}\chi_i(h) E_{p1} \eta_{th} m_0 \left(\eta_{(th)^{-1}}E_{1p}\right) E_{pq}\eta_{t} \\
		&=\chi_i(t)^{-1}\left(\sum_{h\in \mathsf{H}} \sigma((th)^{-1},th)^{-1}\chi_i(th) E_{p1} \eta_{th} m_0 \eta_{(th)^{-1}}E_{1p}\right) E_{pq}\eta_{t} = \chi_i(t)^{-1} \hat{w}_{\chi_i} E_{pq}\eta_t \ ,
	\end{split}
\end{equation}
Therefore, we have proved that $E_{pq}\eta_t \hat{w}_{\chi_i}=\chi_i(t)^{-1}\hat{w}_{\chi_i}E_{pq}\eta_t$ for any $t\in \mathsf{H}$, $p,q,i\in\{1,\dots,s\}$, and hence, $\mathfrak{B} \hat{w}_{\chi_i} = \hat{w}_{\chi_i}\mathfrak{B}$, for all $i=1,\dots, s$.

Finally, let us deduce that $\mathsf{M}=\sum_{i=1}^s\mathfrak{B} \hat{w}_{\chi_i}$. Being $\mathsf{H}=\{h_1,h_2,\dots,h_s\}$, we have that
\begin{equation}\nonumber
	\begin{bmatrix} 
		\hat{w}_{\chi_1}\\
		\hat{w}_{\chi_2}\\
			\vdots\\
		\hat{w}_{\chi_s}\\
	\end{bmatrix}
=
\begin{bmatrix} 
		\chi_1(h_1) & \chi_1(h_2) & \cdots & \chi_1(h_s)\\
		\chi_2(h_1) & \chi_2(h_2) & \cdots & \chi_2(h_s)\\
		\vdots &\vdots& \ddots &\vdots \\
		\chi_s(h_1) & \chi_s(h_2) & \cdots & \chi_s(h_s)\\
\end{bmatrix}
\cdot
\begin{bmatrix}
	\sigma(h_1,h_1^{-1})^{-1} \cdot\sum_{j=1}^s\left(E_{j1} \eta_{h_1} m_0 \eta_{h_1^{-1}} E_{1j}\right) \\
	\sigma(h_2,h_2^{-1})^{-1} \cdot\sum_{l=1}^s\left(E_{l1} \eta_{h_2} m_0 \eta_{h_2^{-1}} E_{1l}\right) \\
	\vdots\\
	\sigma(h_s,h_s^{-1})^{-1} \cdot\sum_{k=1}^s\left(E_{k1} \eta_{h_s} m_0 \eta_{h_s^{-1}} E_{1k}\right)\\
\end{bmatrix} \ .
\end{equation}
Since $E_{p1}\eta_g w_0\eta_h E_{1q}\neq0$ for any $p,q\in\{1,\dots,n\}$ and $g,h\in\mathsf{H}$, note that $\sum_{j=1}^s E_{j1} \eta_{h_i} m_0 \eta_{h_i^{-1}} E_{1j}\neq0$, for all $i=1,\dots,s$. As $\mathsf{M}$ is irreducible graded, it follows that $\mathsf{M}=\mathfrak{B}\left( \sum_{j=1}^s E_{j1} \eta_{h_i} m_0 \eta_{h_i^{-1}} E_{1j} \right)\mathfrak{B}$ for all $i=1,\dots,s$. Similarly to the final part of the proof of Proposition \ref{1.64}, and using the matrix equality above, it is easy to check that $\mathsf{M}=\sum_{i=1}^s\mathfrak{B} \hat{w}_{\chi_i}$, and so, as $\mathsf{M}$ is an irreducible $\mathsf{G}$-graded $\mathfrak{B}$-bimodule, we conclude that $\mathsf{M}=\mathfrak{B} \hat{w}_{\chi}$ for some irreducible character $\chi$ of $\mathsf{H}$.
\end{proof}

Observing the proof of Theorem \ref{1.03}, and considering the subset $\beta=\{E_{ij}\eta_h: i,j=1,\dots,n, h\in\mathsf{H}\}$ of $\mathfrak{B}$, we have that $wb=\gamma_{b}bw\neq0$ for any $b\in\beta$, where $\gamma_{b}\in\mathbb{F}^*$. This means that any element of $\mathfrak{B}$ commutes with any element of $\mathsf{M}$ up to a scalar.

Below, let us rewrite the previous theorem from its proof.

\begin{corollary}\label{1.43}
Assume the same assumptions of Theorem \ref{1.03}. 
If $\mathsf{M}$ is irreducible graded, then there exist a nonzero homogeneous element $m_0 \in\mathsf{M}$ and a character $\chi: \mathsf{H}\rightarrow \mathbb{F}^*$ such that the element defined by
	\begin{equation}\nonumber
\hat{w}_{\chi}=\sum_{i=1}^n\left(\sum_{h\in \mathsf{H}} \chi(h)\sigma(h,h^{-1})^{-1} \eta_h E_{i1} m_0 E_{1i}\eta_{h^{-1}}\right) 
	\end{equation}
satisfies $\mathsf{M}=\mathfrak{B} \hat{w}_{\chi}$, and $E_{pq}\eta_t\hat{w}_{\chi}=\chi(t)^{-1} \hat{w}_{\chi}E_{pq}\eta_t $ for all $p,q\in\{1,\dots,n\}$ and $t\in \mathsf{H}$.
\end{corollary}
\begin{proof}
It is immediate of the proof of Theorem \ref{1.03}, and of Proposition \ref{1.64}.
%
%
\end{proof}

Observe that the element $\hat{w}_{\chi}$ in Corollary \ref{1.43} is homogeneous, and satisfies $\mathfrak{B} \hat{w}_{\chi}=\hat{w}_{\chi}\mathfrak{B}$. Recall that $\chi$ is an irreducible character of $\mathsf{H}$.

\begin{corollary}\label{1.09}
Let $\mathsf{G}$ be an abelian group, $\mathbb{F}$ an algebraically closed field with $\mathsf{char}(\mathbb{F})=0$, and $\mathfrak{A}$ a finite dimensional algebra over $\mathbb{F}$ with a $\mathsf{G}$-grading. If $\mathfrak{A}$ is simple graded, then any irreducible $\mathsf{G}$-graded $\mathfrak{A}$-bimodule $\mathsf{M}$ is equal to a $\mathsf{G}$-graded $\mathfrak{A}$-bimodule $\mathfrak{A} w$, for some homogeneous element $w\in\mathsf{M}$ which satisfies $\mathfrak{A} w = w\mathfrak{A}$.
\end{corollary}
\begin{proof}
Let $\mathsf{M}$ be an irreducible $\mathsf{G}$-graded $\mathfrak{A}$-bimodule. By Theorem \ref{teoBahtSehgZaic}, there exist a finite subgroup $\mathsf{H}$ of $\mathsf{G}$, $\sigma\in\mathsf{Z}^2(\mathsf{H}, \mathbb{F}^*)$, and a $\mathsf{G}$-graded isomorphism (of algebras) $\psi$ from $\mathfrak{A}$ in $M_n(\mathbb{F}^\sigma[\mathsf{H}])$, where $M_n(\mathbb{F}^\sigma[\mathsf{H}])$ is graded with a canonical elementary $\mathsf{G}$-grading. Put $\mathfrak{B}=M_n(\mathbb{F}^\sigma[\mathsf{H}])$, and consider the (two-sided) action from $\mathfrak{B}$ in $\mathsf{M}$ defined by $b \cdot m\mapsto\psi^{-1}(b)m$ and $m \cdot b\mapsto m\psi^{-1}(b)$. From this, $\mathsf{M}$ is also a $\mathsf{G}$-graded $\mathfrak{B}$-bimodule. Since $\mathfrak{B}$ is unitary and $\mathfrak{A}=\psi^{-1}(\mathfrak{B})$, it follows that $\mathsf{M}$ is an irreducible $\mathsf{G}$-graded unitary $\mathfrak{B}$-bimodule. By Theorem \ref{1.03}, there exists a homogeneous element $w\in\mathsf{M}$ satisfying $\mathsf{M}=\mathfrak{B}  w$, such that $\mathfrak{B} w=w \mathfrak{B}$. Therefore, because $\mathfrak{B} w=\psi^{-1}(\mathfrak{B})w=\mathfrak{A}w$ and $w \mathfrak{B}=w\psi^{-1}(\mathfrak{B})=w\mathfrak{A}$, the result follows.
\end{proof}

\begin{theorem}\label{1.19}
	Let $\mathsf{G}$ be a group, $\mathsf{H}$ a finite abelian subgroup of $\mathsf{G}$, $\mathbb{F}$ an algebraically closed field with $\mathsf{char}(\mathbb{F})=0$, $\sigma\in\mathsf{Z}^2(\mathsf{H}, \mathbb{F}^*)$, and $\mathfrak{B}=M_n(\mathbb{F}^\sigma[\mathsf{H}])$ with a canonical elementary $\mathsf{G}$-grading. Let $\mathsf{M}$ be $\mathsf{G}$-graded unitary $\mathfrak{B}$-bimodule. If $\mathsf{M}$ is finitely generated (as bimodule), then $\mathsf{M}$ can be written as a finite direct sum of irreducible graded subbimodules of the form $\mathfrak{B} w$, where $w\in\mathsf{M}$ is a homogeneous element which satisfies $w\mathfrak{B}=\mathfrak{B} w$.
\end{theorem}
\begin{proof}
By Corollary \ref{1.42} and Lemma \ref{1.41}, we can obtain a finite sequence 
	\begin{equation}\label{1.21}
	\{0\}=\mathsf{M}_0\subsetneq \mathsf{M}_1 \subsetneq \cdots \subsetneq \mathsf{M}_{p-1} \subsetneq \mathsf{M}_{p}=\mathsf{M} \ ,
	\end{equation}
where $\mathsf{M}_i$'s are $\mathsf{G}$-graded $\mathfrak{B}$-subbimodules of $\mathsf{M}$ such that $\mathsf{M}_{i+1}/\mathsf{M}_i$ is an irreducible $\mathsf{G}$-graded $\mathfrak{B}$-bimodule for $i=0,1,\dots, p-1$.

By induction on $p$, let us show that there exist homogeneous elements $w_1,\dots, w_p\in\mathsf{M}$ such that $\mathsf{M}=\bigoplus_{i=1}^p \mathfrak{B} w_i$ with $\mathfrak{B}w_i=w_i\mathfrak{B}\neq0$ for all $i=1,\dots, p$.
	
	First, suppose $p=1$. Hence, $\mathsf{M}=\mathsf{M}_1$ with $\mathsf{M}/\mathsf{M}_0\cong_{\mathsf{G}}\mathsf{M}$ $\mathsf{G}$-irreducible. It follows from Theorem \ref{1.03} that there exists a nonzero homogeneous element $w_1\in \mathsf{M}_1$ such that $\mathsf{M}_1=\mathfrak{B}w_1$ with $w_1 \mathfrak{B}=\mathfrak{B}w_1\neq0$.
	
	Assume that the result is true for $p=d>1$, i.e. there exist nonzero homogeneous elements $w_1\in\mathsf{M}_1\setminus\mathsf{M}_0,  w_2\in\mathsf{M}_2\setminus\mathsf{M}_1, \dots, w_d\in\mathsf{M}_d\setminus\mathsf{M}_{d-1}$ satisfying 
	\begin{equation}\nonumber
	\mathsf{M}_d=\mathfrak{B} w_1\oplus \mathfrak{B} w_2\oplus\cdots\oplus \mathfrak{B} w_d \ ,
	\end{equation}
where each $w_i \mathfrak{B}=\mathfrak{B} w_i\neq0$ is an irreducible graded subbimodule, for all $i=1,\dots, d$.

Now, for $p=d+1$, notice that $\mathsf{M}_{d+1}=\mathsf{M}_d+\mathfrak{B} w\mathfrak{B}$ (by maximality of $\mathsf{M}_d$ in $\mathsf{M}_{d+1}$) for any $w\in \mathsf{M}_{d+1}\setminus\mathsf{M}_d$, $w\neq0$. Since $\mathsf{M}_{d+1}/\mathsf{M}_d$ is irreducible graded, by Corollary \ref{1.43}, there exist a nonzero homogeneous element $w_0\in\mathsf{M}_{d+1}\setminus\mathsf{M}_d$ and a character $\chi: \mathsf{H}\rightarrow \mathbb{F}^*$ such that $\hat{w}_{\chi}$ defined by
	\begin{equation}\nonumber
\hat{w}_{\chi}=\sum_{i=1}^n\left(\sum_{h\in \mathsf{H}} \chi(h)\sigma(h,h^{-1})^{-1} \eta_h E_{i1} w_0 E_{1i}\eta_{h^{-1}}\right) \neq0
	\end{equation}
satisfies $\mathsf{M}_{d+1}/\mathsf{M}_d=\mathfrak{B} (\hat{w}_{\chi}+\mathsf{M}_d)$, and $E_{ij}\eta_g(\hat{w}_{\chi}+\mathsf{M}_d)=\chi(g)^{-1} (\hat{w}_{\chi}+\mathsf{M}_d) E_{ij}\eta_g$ for all $i,j\in\{1,\dots,n\}$ and $g\in \mathsf{H}$. Let us consider $w_{d+1}\coloneqq\hat{w}_{\chi}$. As $\mathfrak{B} (w_{d+1}+\mathsf{M}_d)\neq 0+\mathsf{M}_d$, there exists $E_{kl}\eta_g \in\mathfrak{B}$ such that $E_{kl}\eta_g w_{d+1}=\chi(g)^{-1}w_{d+1}E_{kl}\eta_g\notin\mathsf{M}_d$. From this, using the equality $E_{kl}\eta_g=\sigma(gh^{-1},h)^{-1}\sigma(g,e)^{-1} E_{ki}\eta_{gh^{-1}} E_{ij}\eta_h E_{jl}\eta_{e}$, for any $i,j\in\{1,\dots,n\}$ and $h\in \mathsf{H}$, it is immediate that $E_{ij}\eta_h w_{d+1}=\chi(h)^{-1}w_{d+1}E_{ij}\eta_h\notin\mathsf{M}_d$ for any $i,j\in\{1,\dots,n\}$ and $h\in \mathsf{H}$. Consequently, we have that $\mathfrak{B} w_{d+1}\mathfrak{B}=\mathfrak{B} w_{d+1}=w_{d+1}\mathfrak{B}$ and $\mathfrak{B}w_{d+1}\cap\mathsf{M}_d=\{0\}$, and so $\mathsf{M}_{d+1}=\mathsf{M}_d\oplus\mathfrak{B} w_{d+1}$.
By Isomorphisms Theorem (see Theorem \ref{1.02}), we have that
	\begin{equation}\nonumber
\displaystyle\frac{\mathsf{M}_{d+1}}{\mathsf{M}_d}=\frac{\mathsf{M}_d \oplus\mathfrak{B} w_{d+1}}{\mathsf{M}_d}\cong_{\mathsf{G}} \frac{\mathfrak{B} w_{d+1}}{\mathsf{M}_d\cap\mathfrak{B} w_{d+1}}=\frac{\mathfrak{B} w_{d+1}}{\{0\}}\cong_{\mathsf{G}} \mathfrak{B} w_{d+1} \ ,
	\end{equation}
as  $\mathsf{G}$-graded $\mathfrak{B}$-bimodules. Therefore, we proved that $\mathfrak{B} w_{d+1}$ is an irreducible $\mathsf{G}$-graded $\mathfrak{B}$-subbimodule of $\mathsf{M}$ such that $w_{d+1}\mathfrak{B}=\mathfrak{B}w_{d+1}$, and hence,
	\begin{equation}\nonumber
\mathsf{M}_{d+1}=\mathsf{M}_d \oplus\mathfrak{B} w_{d+1} =\mathfrak{B} w_1\oplus\cdots\oplus \mathfrak{B} w_d \oplus\mathfrak{B} w_{d+1} ,
	\end{equation}
where each $\mathfrak{B} w_i$ is an irreducible graded subbimodule of $\mathsf{M}$ satisfying $\mathfrak{B} w_i=w_i \mathfrak{B}\neq\{0\}$ for all $i=1,\dots, d+1$. Furthermore, by induction, the result is proved.
\end{proof}

Under the same conditions of Theorem \ref{1.19}, take $w_1,\dots,w_l\in\mathsf{M}$ such that $\mathsf{M}=\bigoplus_{k=1}^l \mathfrak{B}w_k$, with $\mathfrak{B}w_k=w_k\mathfrak{B}$ irreducible graded for all $k=1,\dots,l$. By Corollary \ref{1.43} (or by the proof of Theorem \ref{1.03}, and Proposition \ref{1.64}), for all $k=1,\dots,l$, there exist a map $\chi_k: \mathsf{\mathsf{H}}\rightarrow \mathbb{F}^*$ ($\chi_k$ is the irreducible character associated with $w_k$) satisfying 
	\begin{equation}\nonumber
E_{ij}\eta_g w_k=\chi_k(g)^{-1}w_k E_{ij}\eta_g
	\end{equation}
for any $g\in\mathsf{H}$ and $i,j=1,\dots,n$. In addition, considering the map $\varphi: \mathbb{N}_l\times\mathsf{\mathsf{H}}\rightarrow \mathbb{F}^*$ defined by $\varphi(k,g)=\chi_k(g)^{-1}$, where $\mathbb{N}_l=\{1,\dots,l\}$, we have that $E_{ij}\eta_g w_k=\varphi(k,g)w_k E_{ij}\eta_g$, for any $k=1,\dots,l$,  $g\in\mathsf{H}$ and $i,j=1,\dots,n$.

To conclude this section, let us present a result that combines all the previous results.

\begin{theorem}\label{1.83}
Let $\mathbb{F}$ be an algebraically closed field, $\mathsf{G}$ an abelian group and $\mathfrak{A}$ a finite dimensional algebra over $\mathbb{F}$ with a $\mathsf{G}$-grading. Suppose that $\mathsf{char}(\mathbb{F})=0$, and $\mathfrak{A}$ is $\mathsf{G}$-simple. If $\mathsf{M}$ is a finitely generated $\mathsf{G}$-graded $\mathfrak{A}$-bimodule, then there exist nonzero homogeneous elements $w_1, w_2,\dots,w_n\in\mathsf{M}$ such that 
	\begin{equation}\nonumber
\mathsf{M}=\mathfrak{A} w_1\oplus \mathfrak{A} w_2\oplus \cdots \oplus \mathfrak{A} w_n \ ,
	\end{equation}
where $w_i \mathfrak{A}=\mathfrak{A} w_i\neq0$ for all $i=1,2,\dots,n$, and $\mathfrak{A} w_i$ is $\mathsf{G}$-irreducible.
\end{theorem}
\begin{proof}
By Theorem \ref{teoBahtSehgZaic}, there exists a $\mathsf{G}$-graded isomorphism $\psi$ from $\mathfrak{A}$ in $M_n(\mathbb{F}^\sigma[\mathsf{H}])$, for some finite subgroup $\mathsf{H}$ of $\mathsf{G}$ and $\sigma\in\mathsf{Z}^2(\mathsf{H}, \mathbb{F}^*)$, and $M_n(\mathbb{F}^\sigma[\mathsf{H}])$  with a canonical elementary $\mathsf{G}$-grading. Put $\mathfrak{B}=M_n(\mathbb{F}^\sigma[\mathsf{H}])$. We have that $\mathsf{M}$ is a $\mathsf{G}$-graded $\mathfrak{A}$-bimodule via left action $b\cdot m=\psi^{-1}(b)\cdot m$ and right action $m \cdot b=m \cdot \psi^{-1}(b)$ for any $b\in\mathfrak{B}$ and $m\in\mathsf{M}$. Since $\mathfrak{B}$ is unitary, it follows that $\mathsf{M}$ is unitary, and hence, we have the same conditions of Theorem \ref{1.19}. Therefore, the result follows.
\end{proof}

Note that Theorem \ref{1.83} simplifies the way in which the elements of a finitely generated $\mathsf{G}$-graded $\mathfrak{A}$-bimodule $\mathsf{M}$ can be presented. \textit{A priori}, being $x_1,\dots,x_m\in\mathsf{M}$ a generating set of $\mathsf{M}$, given any $y\in\mathsf{M}$, there are $a_{ij},b_{ij}\in\mathfrak{A}$, with $i=1,\dots,m$ and $j=1,\dots,k_i$, such that $y=\sum_{i=1}^{m}\left(\sum_{j=1}^{k_i} a_{ij}x_ib_{ij}\right)$. Using the previous theorem, we can write only $y=a_1w_1+\cdots+a_nw_n$ for some $a_1,\dots,a_n\in\mathfrak{A}$. Therefore, Theorem \ref{1.83} plays an important tool when it is necessary to know the structure/properties of a given (graded) bimodule.


\subsection{
Graded Bimodules over Semisimple Graded Algebras}

%



Suppose that $\mathfrak{A}$ is a unitary algebra. If $\epsilon\in\mathfrak{A}$ is a central idempotent element, i.e. $\epsilon\in\mathcal{Z}(\mathfrak{A})$ with $\epsilon^2=\epsilon$, then $1-\epsilon\in\mathfrak{A}$ is also a central idempotent element of $\mathfrak{A}$ such that $\epsilon$ and $(1-\epsilon)$ are orthogonal when $\epsilon\neq0$, i.e. $\epsilon(1-\epsilon)=(1-\epsilon)\epsilon=0$. Therefore, given $x\in\mathfrak{A}$, notice that $x=x\epsilon+x(1-\epsilon)$, and hence, it is not difficult to see that $\mathfrak{A}=\mathfrak{A} \epsilon \oplus \mathfrak{A}(1-\epsilon)$. This decomposition is called the {\bf Peirce decomposition} of $\mathfrak{A}$ relative to $\epsilon$. Naturally, we can extend this definition to $n$ idempotent elements of $\mathfrak{A}$, as follows. Let $\epsilon_1,\dots,\epsilon_n\in \mathfrak{A}$ be distinct central orthogonal idempotent elements. Without loss of generality, suppose that $1=\sum_{i=1}^n \epsilon_i$. Given $x\in\mathfrak{A}$, we have  $x=x1=\sum_{i=1}^n x\epsilon_i$, and hence, $\mathfrak{A}= \mathfrak{A} \epsilon_1 \oplus\cdots\oplus \mathfrak{A} \epsilon_n$ is the Peirce decomposition of $\mathfrak{A}$ relative to $\epsilon_1,\dots,\epsilon_n$.

We remember that a graded algebra $\mathfrak{A}$ is called \textit{left (resp. right) semisimple graded} when there exist simple graded left (resp. right) ideals $\mathfrak{A}_r$'s, $r\in I$ (not necessarily finite), such that $\mathfrak{A}=\bigoplus_{r\in I}\mathfrak{A}_r$. Analogously, we define a \textit{two-sided} semisimple graded algebra, in the sense of two-sided ideals. More precisely, we say that the graded algebra $\mathfrak{A}$ is \textbf{weak semisimple graded} if there exist simple graded two-sided ideals $\mathfrak{A}_r$'s, $r\in I$ (not necessarily finite), such that $\mathfrak{A}=\bigoplus_{r\in I}\mathfrak{A}_r$. Obviously, this last sum is finite when $\mathfrak{A}$ is unitary. In fact, there exist $\mathfrak{i}_{r_1}\in\mathfrak{A}_{r_1}$, $\dots$, $\mathfrak{i}_{r_n}\in\mathfrak{A}_{r_n}$ (not all zero) such that $1_{\mathfrak{A}}=\sum_{j=1}^n\mathfrak{i}_{r_j}\in \bigoplus_{j=1}^n\mathfrak{A}_{r_j}$. Put $I_0=\{r_1,\dots,r_n\}$, and take any $y\in\bigcup_{k\in I\setminus I_0}\mathfrak{A}_k$. Since $1_{\mathfrak{A}}x=x1_{\mathfrak{A}}=x$ for any $x\in\mathfrak{A}$, we have that $1_{\mathfrak{A}}y=y$, but this implies that $y=y\cdot 1_{\mathfrak{A}}=y \left(\sum_{j=1}^n\mathfrak{i}_{r_j}\right)=\sum_{j=1}^n(y\mathfrak{i}_{r_j})$ which belongs to $\bigoplus_{l\in I_0}\mathfrak{A}_{l}$, because the $\mathfrak{A}_{l}$'s are graded two-sided ideals of $\mathfrak{A}$. From this, it follows that $y=0$, since $\displaystyle\left(\bigoplus_{l\in I_0}\mathfrak{A}_{l}\right)\bigcap \left(\bigoplus_{k\in I\setminus I_0}\mathfrak{A}_k\right)=\{0\}$. Therefore, $\mathfrak{A}=\bigoplus_{r=1}^n\mathfrak{A}_r$ when $\mathfrak{A}$ is unitary. It is not difficult to prove that the algebras $\mathfrak{A}_r$'s are unitary.

\begin{theorem}\label{1.29}
Let $\mathfrak{A}$ and $\widetilde{\mathfrak{A}}$ be two $\mathsf{G}$-graded unitary algebras and $\mathsf{M}$ a unitary $\mathsf{G}$-graded $(\mathfrak{A},\widetilde{\mathfrak{A}})$-bimodule. If $\mathfrak{A}$ and $\widetilde{\mathfrak{A}}$ are weak semisimple graded, then there exist central orthogonal idempotent elements $\mathfrak{i}_1,\dots,\mathfrak{i}_p\in\mathfrak{A}$ and $\hat{\mathfrak{i}}_1,\dots,\hat{\mathfrak{i}}_q\in\widetilde{\mathfrak{A}}$ such that $\mathsf{M}=\displaystyle\bigoplus_{r,s=1}^{p,q} \mathfrak{i}_r \mathsf{M} \hat{\mathfrak{i}}_s$, $\mathfrak{A}=\displaystyle\bigoplus_{r=1}^p\mathfrak{A}_r$ and $\widetilde{\mathfrak{A}}=\displaystyle\bigoplus_{s=1}^q\widetilde{\mathfrak{A}}_s$, where $\mathfrak{A}_r=\mathfrak{A}\mathfrak{i}_r$ and $\widetilde{\mathfrak{A}}_s=\widetilde{\mathfrak{A}}\hat{\mathfrak{i}}_s$ are $\mathsf{G}$-simple two-sided ideals. In addition, each $\mathfrak{i}_r \mathsf{M} \hat{\mathfrak{i}}_s$ is either equal to zero or a faithful $\mathsf{G}$-graded $(\mathfrak{A}_r,\widetilde{\mathfrak{A}}_s)$-bimodule.
\end{theorem}
\begin{proof}
Since $\mathfrak{A}$ and $\widetilde{\mathfrak{A}}$ are unitary and weak semisimple graded, there exist simple graded ideals $\mathfrak{A}_1, \dots,\mathfrak{A}_p\subset\mathfrak{A}$ and $\widetilde{\mathfrak{A}}_1,\dots,\widetilde{\mathfrak{A}}_q\subset\widetilde{\mathfrak{A}}$ such that $\mathfrak{A}=\bigoplus_{r=1}^p\mathfrak{A}_r$, $\widetilde{\mathfrak{A}}=\bigoplus_{s=1}^q\widetilde{\mathfrak{A}}_s$, and we can take central orthogonal idempotent elements $\mathfrak{i}_1,\dots,\mathfrak{i}_p\in\mathfrak{A}$ and $\hat{\mathfrak{i}}_1,\dots,\hat{\mathfrak{i}}_q\in\widetilde{\mathfrak{A}}$ such that $\mathfrak{i}=\sum_{r=1}^p\mathfrak{i}_r$ and $\hat{\mathfrak{i}}=\sum_{s=1}^q\hat{\mathfrak{i}}_s$, where $\mathfrak{i}_r\in\mathfrak{A}_r$ and $\hat{\mathfrak{i}}_s\in\widetilde{\mathfrak{A}}_s$. Note that $\mathfrak{A}_r=\mathfrak{A}_r\mathfrak{i}_r$ and $\widetilde{\mathfrak{A}}_s=\widetilde{\mathfrak{A}}_s\hat{\mathfrak{i}}_s$.

Since $\mathsf{M}$ is unitary, we have $\mathsf{M}=\mathfrak{i}\mathsf{M}\hat{\mathfrak{i}}$. Put $\widetilde{\mathsf{M}}=\sum_{r,s=1}^{p,q} \mathfrak{i}_r \mathsf{M} \hat{\mathfrak{i}}_s$. Since the $\mathfrak{i}_r$'s and $\hat{\mathfrak{i}}_s$'s are orthogonal elements, it follows that $\widetilde{\mathsf{M}}=\bigoplus_{r,s=1}^{p,q} \mathfrak{i}_r \mathsf{M} \hat{\mathfrak{i}}_s$. Obviously, $\widetilde{\mathsf{M}}\subseteq \mathsf{M}$. On the other hand, for any $x\in \mathsf{M}$, we have 
	$
x=\mathfrak{i} x\hat{\mathfrak{i}}=\left(\sum_{r=1}^p \mathfrak{i}_r \right) x \left(\sum_{s=1}^q \hat{\mathfrak{i}}_s \right)=\sum_{r,s=1}^{p,q} \mathfrak{i}_r x \hat{\mathfrak{i}}_s \in \widetilde{\mathsf{M}}.$ 
Hence, $\mathsf{M}\subseteq \widetilde{\mathsf{M}}$, and so $\mathsf{M}=\widetilde{\mathsf{M}}$.

Now, it is immediate that $\mathfrak{i}_r \mathsf{M} \hat{\mathfrak{i}}_s$ is a $\mathsf{G}$-graded $(\mathfrak{A}_r,\widetilde{\mathfrak{A}}_s)$-bimodule for all $r=1,\dots, p$ and $s=1,\dots, q$, because the $\mathfrak{A}_r$'s and $\widetilde{\mathfrak{A}}_s$'s are graded ideals of $\mathfrak{A}$ and $\widetilde{\mathfrak{A}}$, respectively. And by the simplicities of $\mathfrak{A}_r$'s and $\widetilde{\mathfrak{A}}_s$'s, it follows that $\mathfrak{i}_r \mathsf{M} \hat{\mathfrak{i}}_s$ is a faithful bimodule when it is nonzero.
\end{proof}

By Theorem \ref{1.29}, a particular case is given when there exist $r_0\in\{1,\dots,p\}$ and $s_0\in\{1,\dots,q\}$ such that $\mathfrak{A}_{r_0}$ and $\mathfrak{A}_{s_0}$ are finite-dimensional and graded isomorphic (as graded algebras). Hence, if the basis field $\mathbb{F}$ is algebraically closed with $\mathsf{char}(\mathbb{F})=0$, $\mathsf{G}$ is abelian, and $\mathfrak{i}_{r_0} \mathsf{M} \hat{\mathfrak{i}}_{s_0}$ is finitely generated, then there exist homogeneous elements $d_{1},\dots,d_{k_{(r_0,s_0)}}\in \mathsf{M}$ such that $\mathfrak{i}_{r_0} \mathsf{M} \hat{\mathfrak{i}}_{s_0}=\bigoplus_{j=1}^{k_{(r_0,s_0)}}\mathfrak{A}_{r_0} d_j$, where $\mathfrak{A}_{r_0} d_{j}$ is a $\mathsf{G}$-simple $\mathfrak{A}_{r_0}$-bimodule with $\mathfrak{A}_{r_0}d_{j}=d_{j}\widetilde{\mathfrak{A}}_{s_0}$ for all $j=1,\dots,k_{(r_0,s_0)}$. In fact, assume that $\psi:\mathfrak{A}_{r_0}\rightarrow \mathfrak{A}_{s_0}$ is a graded isomorphism of graded algebras, and put $\mathsf{V}_{(r_0,s_0)}=\mathfrak{i}_{r_0} \mathsf{M} \hat{\mathfrak{i}}_{s_0}$. Consider the right action on $\mathsf{V}_{(r_0,s_0)}$ given by $x\cdot a=x\psi(a)$ for any $x\in\mathsf{V}_{(r_0,s_0)}$ and $a\in\mathfrak{A}_{r_0}$. Hence, we can see $\mathsf{V}_{(r_0,s_0)}$ as an $\mathfrak{A}_{r_0}$-bimodule. The result follows from Theorem \ref{1.83}.

%

\begin{corollary}\label{1.33}
Let $\mathsf{G}$ be an abelian group, $\mathbb{F}$ an algebraically closed field with $\mathsf{char}(\mathbb{F})=0$, $\mathfrak{A}$ a finite dimensional unitary $\mathbb{F}$-algebra with a $\mathsf{G}$-grading, and $\mathsf{M}$ a $\mathsf{G}$-graded unitary $\mathfrak{A}$-bimodule. If $\mathfrak{A}$ is weak semisimple graded and $\mathsf{M}$ is finitely generated, then there exists a decomposition of $\mathsf{M}$ of the form $\bigoplus_{i,j=1}^p \mathsf{M}_{ij}$, $p\in\mathbb{N}$, where each $\mathsf{M}_{ij}$ is either equal to zero or a faithful $\mathsf{G}$-graded $\mathfrak{A}$-bimodule, and each $\mathsf{M}_{ii}$ can be written as a finite sum of irreducible $\mathsf{G}$-graded $\mathfrak{A}$-subbimodules of the form $\mathfrak{A} w$, for some homogeneous element $w\in\mathsf{M}$ which satisfies $\mathfrak{A} w = w\mathfrak{A}$.
\end{corollary}
\begin{proof}
It is immediate from Theorems \ref{1.29} and \ref{1.83}.
\end{proof}

It is important to comment that $\mathsf{M}_{ij}$ presented by the previous corollary is of the form $\mathfrak{i}_r\mathsf{M}\mathfrak{i}_s$, as in Theorem \ref{1.83}. Observe that if $\mathfrak{i}_s\mathsf{M}=\mathsf{M}\mathfrak{i}_s$ for some $s=1,\dots, p$, thus $\mathsf{M}_{is}=\{0\}$ for any $i\neq s$. In addition, we have the following result.

\begin{corollary}\label{1.32}
Under the assumptions of Theorem \ref{1.29}, suppose that $\widetilde{\mathfrak{A}}=\mathfrak{A}$ and $\mathsf{M}$ is a (two-sided) ideal of $\mathfrak{A}$. Then $\mathsf{M}=\bigoplus_{s=1}^p  \mathfrak{i}_s \mathsf{M} \mathfrak{i}_s$ iff $\mathfrak{i}_1,\dots,\mathfrak{i}_p\in \mathcal{Z}_{\mathfrak{A}}(\mathsf{M})$.
\end{corollary}
\begin{proof}
By Theorem \ref{1.29} and Corollary \ref{1.33}, we have $\mathsf{M}=\bigoplus_{r,s=1}^p \mathsf{M}_{rs}$, where $\mathsf{M}_{rs}=\mathfrak{i}_s \mathsf{M} \mathfrak{i}_r$. Suppose $\mathfrak{i}_s\in \mathcal{Z}_{\mathfrak{A}}(\mathsf{M})$ for any $s=1,\dots,p$. Hence, we can conclude that $\mathfrak{i}_s \mathsf{M} \mathfrak{i}_r=\{0\}$ for $r\neq s$, since $\mathfrak{i}_s x \mathfrak{i}_r=(\mathfrak{i}_s x \mathfrak{i}_r)\mathfrak{i}_s=0$ for all $x\in\mathsf{M}$ and $r\neq s$. From this, $\mathsf{M}=\bigoplus_{s=1}^p \mathfrak{i}_s \mathsf{M} \mathfrak{i}_s$. Therefore, we have the result.
\end{proof}



\section{A Pierce Decomposition of the Jacobson Radical}\label{applications}

The goal of this section is to present one application of the previous results concerning the Jacobson radical of $\mathfrak{A}$, where $\mathfrak{A}$ is a $\mathsf{G}$-graded finite dimensional algebra over a field $\mathbb{F}$. Here, let us use some of the main results of this work, for example the Theorems \ref{1.19} and \ref{1.29}, to give a description of the Jacobson radical $\mathsf{J}$ of a finite dimensional algebra $\mathfrak{A}$ with some grading over a group $\mathsf{G}$ in terms of the concept of the Peirce decomposition and of your semisimple part $\mathfrak{B}$. For more details about Jacobson radical, see \cite{Hers05,Jaco64}.
 
%

In \cite{Svir11}, Sviridova presented a generalization of Theorem of Wedderburn-Malcev for algebras graded by a (finite abelian) group. She described, in suitable conditions, any finite dimensional graded algebras as a direct sum of a maximal semisimple graded subalgebra and its Jacobson radical, which is also graded. More precisely, Sviridova proved:

\begin{lemma}[Lemma 2, \cite{Svir11}]\label{teoIrina03}
	Let $\mathsf{G}$ be a finite abelian group, and $\mathbb{F}$ an algebraically closed field of characteristic zero. Any finite dimensional $\mathsf{G}$-graded $\mathbb{F}$-algebra $\mathfrak{A}$ is isomorphic as $\mathsf{G}$-graded
	algebra to a $\mathsf{G}$-graded $\mathbb{F}$-algebra of the form 
	\begin{equation}\label{1.06}
	\mathfrak{A}' = \left(M_{k_1}(\mathbb{F}^{\sigma_1}[\mathsf{H}_1]) \times \cdots \times M_{k_p}(\mathbb{F}^{\sigma_p}[\mathsf{H}_p])\right) \oplus \mathsf{J} \ .
	\end{equation}
	Here the Jacobson radical $\mathsf{J} = \mathsf{J}(\mathfrak{A})$ of $\mathfrak{A}$ is a graded ideal, and $\mathfrak{B} = M_{k_1}(\mathbb{F}^{\sigma_1}[\mathsf{H}_1]) \times \cdots \times M_{k_p}(\mathbb{F}^{\sigma_p}[\mathsf{H}_p])$ (direct product of algebras) is the maximal graded semisimple subalgebra of $\mathfrak{A}'$, $p\in\mathbb{N}\cup\{0\}$. The $\mathsf{G}$-grading on $\mathfrak{B}_l = M_{k_l} (\mathbb{F}^{\sigma_l} [\mathsf{H}_l])$ is the canonical elementary grading  defined by some $k_l$-tuple $(\theta_{l_1},\dots,\theta_{l_{k_l}})\in \mathsf{G}^{k_l}$, where $\mathsf{H}_l$ is a subgroup $\mathsf{G}$ and $\sigma\in\mathsf{Z}^2(\mathsf{H}_l,\mathbb{F}^*)$ is a $2$-cocycle.
\end{lemma}

From this, since the algebras $M_{k_i}(\mathbb{F}^{\sigma_i}[\mathsf{H}_i])$'s are well known (see \cite{BahtSehgZaic08}), it is interesting to present some properties (when possible) of the (graded) ideal $\mathsf{J}$, as well as its connection with the algebras $M_{k_i}(\mathbb{F}^{\sigma_i}[\mathsf{H}_i])$'s. In the next result, let us exhibit a description of the Jacobson radical of a finite dimensional graded algebra, which is a consequence of Theorem \ref{1.27}, Corollary \ref{1.25}, Theorem \ref{1.29}, Corollary \ref{1.33}, and Theorem \ref{teoBahtSehgZaic}.


\begin{theorem}\label{1.30}
Let $\mathsf{G}$ be a group and $\mathfrak{A}=\mathfrak{B}\oplus\mathsf{J}$ a finite dimensional algebra with a $\mathsf{G}$-grading, where $\mathfrak{B}$ is a maximal semisimple $\mathsf{G}$-graded subalgebra of $\mathfrak{A}$, and $\mathsf{J}=\mathsf{J}(\mathfrak{A})$ is the Jacobson radical and a graded ideal of $\mathfrak{A}$. 
Then $\mathsf{J}$ can be decomposed as 
	\begin{equation}\nonumber
\mathsf{J}=\mathsf{J}_{00}\oplus \mathsf{J}_{10}\oplus \mathsf{J}_{01}\oplus \mathsf{J}_{11} \ ,
	\end{equation}
where $\mathsf{J}_{ij}$'s are $\mathsf{G}$-graded $\mathfrak{B}$-bimodules such that:
	\begin{itemize}
		\item[i)] for $r=0,1$, $\mathsf{J}_{0r}$ is a $0$-left $\mathfrak{B}$-module and $\mathsf{J}_{1r}$ is a left $\mathsf{G}$-graded  faithful $\mathfrak{B}$-module;
		\item[ii)] for $s=0,1$, $\mathsf{J}_{s0}$ is a $0$-right $\mathfrak{B}$-module and $\mathsf{J}_{s1}$ is a right $\mathsf{G}$-graded faithful $\mathfrak{B}$-module;
		\item[iii)] $\mathsf{J}_{rq}\mathsf{J}_{qs}\subseteq \mathsf{J}_{rs}$  for all $r,p,q,s\in\{0,1\}$, with $\mathsf{J}_{rp}\mathsf{J}_{qs}=\{0\}$ for $p\neq q$.
	\end{itemize}
Moreover, if $\mathfrak{B}=M_{k_1}(\mathbb{F}^{\sigma_1}[\mathsf{H}_1]) \times \cdots \times M_{k_p}(\mathbb{F}^{\sigma_p}[\mathsf{H}_p])$, $\mathsf{G}$ is finite abelian, and $\mathbb{F}$ is an algebraically closed field with $\mathsf{char}(\mathbb{F})=0$, then following statements are true:
	\begin{itemize} 
		\item[iv)] $\mathsf{J}_{11}=\bigoplus_{s,r=1}^p \mathfrak{i}_r \mathsf{J}_{11} \mathfrak{i}_s$, where each $\mathfrak{i}_r \mathsf{J}_{11} \mathfrak{i}_s$ is a $\mathsf{G}$-graded $(\mathfrak{B}_r,\mathfrak{B}_s)$-bimodule, where $\mathfrak{B}=\bigoplus_{i=1}^p\mathfrak{B}_i$ with $\mathfrak{B}_i=M_{k_i}(\mathbb{F}^{\sigma_i}[\mathsf{H}_i])$. In addition, $\mathfrak{i}_r \mathsf{J}_{11} \mathfrak{i}_s\neq0$ implies that $\mathfrak{i}_r \mathsf{J}_{11} \mathfrak{i}_s$ is a faithful left $\mathfrak{B}_i$-module and a faithful right $\mathfrak{B}_j$-module;
		\item[v)] For each $s=1,\dots, p$, there exist nonzero homogeneous elements $d_{s,1},\dots,d_{s,r_s}\in\mathfrak{i}_s \mathsf{J}_{11} \mathfrak{i}_s$ such that the $\mathsf{G}$-graded vector space $\mathsf{N}_s=\mathsf{span}_\mathbb{F}\{d_{s,1},\dots,d_{s,r_s}\}$ satisfies $\mathfrak{i}_s \mathsf{J}_{11} \mathfrak{i}_s=\mathfrak{B}_s \mathsf{N}_s$. The elements $d_{s,i}$'s satisfy $bd_{s,i}=\gamma_{s,i}(b)d_{s,i}b\neq0$ for any nonzero $b\in\beta_s$, and $i=1,\dots,r_s$, where $\gamma_{s,i}\in\mathbb{F}$, and $\beta_s=\{E_{l_sj_s}\eta_{h_s}\in\mathfrak{B}_s: l_s,j_s=1,\dots, k_s, h_s\in \mathsf{H}_s\}$ is the canonical homogeneous basis of $\mathfrak{B}_s=M_{k_s}(\mathbb{F}^{\sigma_s}[\mathsf{H}_s])$. Moreover, for each $i=1,\dots, r_s$, we have that $\mathfrak{B}_s d_{s,i}$ is a $\mathsf{G}$-simple $\mathfrak{B}_s$-bimodule.
	\end{itemize}
\end{theorem}

Now, let $\mathfrak{A}$ be an algebra satisfying the same assumptions of Lemma \ref{teoIrina03}, and hence, consider $\mathfrak{A}$ with a decomposition equal to (\ref{1.06}). This means that 
	\begin{equation}\nonumber
	\mathfrak{A}=\left(M_{k_1}(\mathbb{F}^{\sigma_1}[\mathsf{H}_1]) \times \cdots \times M_{k_p}(\mathbb{F}^{\sigma_p}[\mathsf{H}_p])\right) \oplus \mathsf{J} \ ,
	\end{equation}
where $k_1, \dots, k_p\in\mathbb{N}$, $\mathsf{H}_1, \dots, \mathsf{H}_p\unlhd\mathsf{G}$, $\sigma_1\in\mathsf{Z}^2(\mathsf{H}_1, \mathbb{F}^*), \dots, \sigma_p\in\mathsf{Z}^2(\mathsf{H}_p, \mathbb{F}^*)$, and $\mathsf{J}=\mathsf{J}(\mathfrak{A})$ is the Jacobson radical of $\mathfrak{A}$. Furthermore, supposing that the conditions of Corollary \ref{1.32} are verified, we have that $\mathsf{J}=\mathsf{J}_{00}\oplus\mathsf{J}_{11}$ where $\mathsf{J}_{11}=\mathfrak{i}_1 \mathsf{J} \mathfrak{i}_1\oplus\cdots\oplus\mathfrak{i}_p \mathsf{J} \mathfrak{i}_p$. Since each $\mathfrak{i}_r$, the unity of the $M_{k_r}(\mathbb{F}^{\sigma_r}[\mathsf{H}_r])$, commutes with $\mathfrak{A}$, it is easy to see that
%
%
%
%
	\begin{equation}\label{1.48}
	\mathfrak{A}\cong_{\mathsf{G}}  \left(M_{k_1}(\mathbb{F}^{\sigma_1}[\mathsf{H}_1])\oplus\mathfrak{i}_1 \mathsf{J} \mathfrak{i}_1 \right) \times \cdots \times \left(M_{k_p}(\mathbb{F}^{\sigma_p}[\mathsf{H}_p])\oplus\mathfrak{i}_p \mathsf{J} \mathfrak{i}_p \right) \times\mathsf{J}_{00} \ .
	\end{equation}
Hence, without loss of generality, we can assume that $\mathfrak{A}= \mathfrak{A}_1\times\cdots\times \mathfrak{A}_p\times\mathsf{J}_{00}$, where $\mathfrak{A}_r=\mathfrak{B}_r\oplus\mathsf{J}_r$, with $\mathfrak{B}_r=M_{k_r}(\mathbb{F}^{\sigma_r}[\mathsf{H}_r])$ and $\mathsf{J}_r=\mathfrak{i}_r\mathsf{J}\mathfrak{i}_r$. Therefore, we can deduce some (or perhaps several) properties of $\mathfrak{A}$ from studying each algebra $\mathfrak{A}_r$ and each  $\mathfrak{A}_r$-bimodule $\mathsf{J}_r$. By Theorem \ref{1.30}, it follows that each algebra $\mathfrak{A}_r$ is unitary, where each $\mathsf{J}_r$ has a decomposition in a finite direct sum of irreducible graded subbimodules of the form $\mathfrak{A}_r w$, for some nonzero homogeneous element $w\in\mathsf{J}_r$ which satisfies $\mathfrak{A}_r w=w\mathfrak{A}_r$, and $\mathsf{J}_{00}$ is an $\mathfrak{A}_1\times\cdots\times \mathfrak{A}_p$-bimodule which is $0$-left and $0$-right (as one-sided module).

%
%


\subsection{An application in PI-Theory} 

Let $\mathsf{G}$ be a group and $\mathfrak{A}$ a finite dimensional $\mathbb{F}$-algebra with a $(\mathsf{G}\times\mathbb{Z}_2)$-grading. The {\bf Grassmann Envelope} of $\mathfrak{A}$, denoted by $\mathsf{E}^\mathsf{G}(\mathfrak{A})$, is defined by 
	\begin{equation}\nonumber
\mathsf{E}^\mathsf{G}(\mathfrak{A})=\left(\mathfrak{A}_0\otimes \mathsf{E}_0 \right) \oplus \left(\mathfrak{A}_1\otimes \mathsf{E}_1 \right)
 ,
	\end{equation}
 where $\mathsf{E}=\mathsf{E}_0\oplus \mathsf{E}_1$ is an infinitely generated non-unitary Grassmann algebra. Record that Grassmann algebra $\mathsf{E}$ is the $\mathbb{F}$-algebra generated by elements $e_1,e_2,e_3,\dots$, such that $e_ie_j=-e_je_i$, for all $i,j\in\mathbb{N}$, where $\mathsf{char}(\mathbb{F})=0$. Naturally $\mathsf{E}$ is $\mathbb{Z}_2$-graded with $\mathsf{E}_0=\mathsf{span}_\mathbb{F}\{e_{i_1}e_{i_2}\cdots e_{i_n}: n \mbox{ is even}\}$, and $\mathsf{E}_1=\mathsf{span}_\mathbb{F}\{e_{j_1}e_{j_2}\cdots e_{j_m}: m \mbox{ is odd}\}$. 
Note that $\mathsf{E}^\mathsf{G}(\mathfrak{A})$ is an $\left(\mathfrak{A}_0\otimes \mathsf{E}_0\right)$-bimodule, and  
\begin{equation}\nonumber
\mathsf{E}^\mathsf{G}(\mathfrak{A})=\bigoplus_{g\in\mathsf{G}}\left((\mathfrak{A}_{(g,0)}\otimes_\mathbb{F}\mathsf{E}_0 )+ (\mathfrak{A}_{(g,1)}\otimes_\mathbb{F} \mathsf{E}_1)\right)
\end{equation}
defines a $\mathsf{G}$-grading on $\mathsf{E}^\mathsf{G}(\mathfrak{A})$. In addition, being $\mathfrak{B}$ a $\left(\mathsf{G}\times\mathbb{Z}_2 \right)$-graded subalgebra of $\mathfrak{A}$, it is easy to see that  $\mathsf{E}^\mathsf{G}(\mathfrak{B})$ is a $(\mathsf{G}\times\mathbb{Z}_2)$-graded subalgebra of $\mathsf{E}^\mathsf{G}(\mathfrak{A})$. 

The next theorem, due to I. Sviridova, gives the positive answer to the well-known {\bf Specht's Problem} for graded varieties of graded associative algebras in characteristic zero. Recall that Specht's Problem asks whether, given any algebra $\mathfrak{A}$, any set of polynomial identities of $\mathfrak{A}$ is a consequence of a finite number of identities of $\mathfrak{A}$. For more details, as well as an overview, about Specht's Problem, see \cite{BeloRoweVish12}.

Record that a $\mathsf{G} T$-ideal of $\mathbb{F}\langle X^\mathsf{G}\rangle$ is a graded (two-side) ideal of $\mathbb{F}\langle X^\mathsf{G}\rangle$ that is invariant under all graded endomorphisms of $\mathbb{F}\langle X^\mathsf{G}\rangle$. It is well known that the set of all $\mathsf{G}$-graded polynomial identities of a $\mathsf{G}$-graded $PI$-algebra over a field $\mathbb{F}$ is a $\mathsf{G} T$-ideal of $\mathbb{F}\langle X^\mathsf{G}\rangle$.

\begin{theorem}[Theorem 2, \cite{Svir11}]\label{teoIrina02}
Let $\mathbb{F}$ be an algebraically closed field of characteristic zero, and $\mathsf{G}$ any finite abelian group. Any $\mathsf{G} T$-ideal of graded identities of a $\mathsf{G}$-graded associative $PI$-algebra over $\mathbb{F}$ coincides with the ideal of $\mathsf{G}$-graded identities of the $\mathsf{G}$-graded Grassmann envelope of some finite dimensional $\mathsf{G} \times \mathbb{Z}_2$-graded associative $\mathbb{F}$-algebra.
\end{theorem}

%
%
%

Now, assume that $\mathfrak{A}$ is a finite dimensional $\mathsf{G}\times\mathbb{Z}_2$-graded algebra as in (\ref{1.48}), i.e. $\mathfrak{A}= \mathfrak{A}_1\times\cdots\times \mathfrak{A}_p\times\mathsf{J}_{00}$, where $\mathfrak{A}_s=\mathfrak{B}_s\oplus\mathsf{J}_s$, with $\mathfrak{B}_s=M_{k_s}(\mathbb{F}^{\sigma_s}[\mathsf{H}_s])$ and $\mathsf{J}_s=\mathfrak{i}_s\mathsf{J}\mathfrak{i}_s$. We have that 
\begin{equation}\nonumber
\mathsf{E}^\mathsf{G}(\mathfrak{A})=\mathsf{E}^\mathsf{G}(\mathfrak{A}_1\times\cdots\times \mathfrak{A}_p\times\mathsf{J}_{00})=\mathsf{E}^\mathsf{G}(\mathfrak{A}_1)\times\cdots\times \mathsf{E}^\mathsf{G}(\mathfrak{A}_p)\times\mathsf{E}^\mathsf{G}(\mathsf{J}_{00}) \ ,
\end{equation}
and $\mathsf{E}^\mathsf{G}(\mathfrak{A}_s)=\mathsf{E}^\mathsf{G}(\mathfrak{B}_s)\oplus\mathsf{E}^\mathsf{G}(\mathsf{J}_s)$. Remember that $\mathfrak{B}=\mathfrak{B}_1\times\cdots\times\mathfrak{B}_p$ is a maximal semisimple graded subalgebra of $\mathfrak{A}$ and $\mathsf{J}(\mathfrak{A})=\mathsf{J}_{00}\times\mathsf{J}_{11}$ is the Jacobson radical of $\mathfrak{A}$, where $\mathsf{J}_{11}=\mathsf{J}_1\times\cdots\times \mathsf{J}_p$ and $\mathsf{J}_{00}$ is a $\mathfrak{B}\oplus\mathsf{J}_{11}$-bimodule which is $0$-left and $0$-right (as one-sided module). From this, and by Theorem \ref{teoIrina02}, to study any graded variety $\mathfrak{W}$ (of graded associative algebras in characteristic zero), it is important to study the $\mathfrak{B}_s$-bimodule $\mathsf{J}_s$, as well as the algebras $\mathsf{E}^\mathsf{G}(\mathsf{J}_s)$ and $\mathsf{E}^\mathsf{G}(\mathsf{J}_{00})$. Recall that a $\mathsf{G}$-graded variety $\mathfrak{W}$ is a family of $\mathsf{G}$-graded (associative) algebras which satisfy all the graded polynomial identities of a given subset $S\subset \mathbb{F}\langle X^\mathsf{G}\rangle$.

On the other hand, in \cite{Mardua01}, De Fran\c{c}a and Sviridova proved that, in finite grading, $\mathfrak{A}_e$ is nilpotent iff $\mathfrak{A}$ is nilpotent (see \cite{Mardua01}, Theorem 3.9, p.237). Consequently, considering $\mathsf{E}=\mathsf{E}_0\oplus\mathsf{E}_1$ and $\mathsf{J}=\mathsf{J}_0\oplus\mathsf{J}_1$ with their $\mathbb{Z}_2$-gradings, since $\mathsf{E}_0$ is commutative and $\mathsf{J}_0$ is nilpotent (because $\mathsf{J}$ is so), we have that $\left(\mathsf{E}^\mathsf{G}(\mathsf{J})\right)_0=\mathsf{J}_0\otimes_\mathbb{F} \mathsf{E}_0$ is nilpotent, and hence, it follows that $\mathsf{E}^\mathsf{G}(\mathsf{J})$ is nilpotent, whose nilpotency index $\mathsf{nd}(\mathsf{E}^\mathsf{G}(\mathsf{J}))$ is at most $2\cdot\mathsf{nd}(\mathsf{J})$. In particular, the algebras $\mathsf{E}^\mathsf{G}(\mathsf{J}_s)$'s and $\mathsf{E}^\mathsf{G}(\mathsf{J}_{00})$ are nilpotent with nilpotency indices at most $2\cdot\mathsf{nd}(\mathsf{J})$.

Finally, assume $\mathfrak{A}_s=\mathfrak{B}_s\oplus\mathsf{J}_s$, with $\mathfrak{B}_s=M_{k_s}(\mathbb{F}^{\sigma_s}[\mathsf{H}_s])$ and $\mathsf{J}_s=\mathfrak{i}_s\mathsf{J}\mathfrak{i}_s$. By Theorem \ref{1.30}, item (v), there exist $d_{s,1},\dots,d_{s,r_s}\in\mathsf{J}_s$ such that $\mathsf{J}_s=\bigoplus_{i=1}^{r_s}\mathfrak{B}_s d_{s,i}$.  So, we have that 
	\begin{equation}\nonumber
		\begin{split}
\mathsf{E}^\mathsf{G}(\mathsf{J}_{11}) =\mathsf{E}^\mathsf{G}\left(\bigtimes_{s=1}^p \mathsf{J}_s \right)=\bigtimes_{s=1}^p \mathsf{E}^\mathsf{G}(\mathsf{J}_s) 
							=\bigtimes_{s=1}^p \mathsf{E}^\mathsf{G} \left(\bigoplus_{i=1}^{r_s}\mathfrak{B}_s d_{s,i} \right)=\bigtimes_{s=1}^p\left(\bigoplus_{i=1}^{r_s} \mathsf{E}^\mathsf{G}\left(\mathfrak{B}_s d_{s,i} \right)\right) \ .
		\end{split}
	\end{equation}
Thus, to study $\mathsf{E}^\mathsf{G}(\mathsf{J}_{11})$, we can to analyze $\mathsf{E}^\mathsf{G}(\mathfrak{B}_s d_{s,i})=
\displaystyle\bigoplus_{j\in\mathbb{Z}_2} d_{s,i}(\mathfrak{B}_s)_j\otimes_{\mathbb{F}}\mathsf{E}_{l_{s,i}+j}
$, for all $s=1,\dots,p$ and $i=1,\dots,r_s$, where $\mathsf{deg}_{\mathbb{Z}_2}(d_{s,i})=l_i$. Therefore, the decomposition presented in Theorem \ref{1.30} can be an important tool to study $\mathsf{E}^\mathsf{G}(\mathsf{J}_{11})$, and, consequently, help to provide a description of $\mathsf{E}^\mathsf{G}(\mathsf{J})$ and of $\mathsf{E}^\mathsf{G}(\mathfrak{A})$, as well as of any graded variety $\mathfrak{W}$ of graded associative algebras in characteristic zero.

Let us finalize this work with the following result, which is a consequence of the observations above.

\begin{proposition}
Let $\mathbb{F}$ be an algebraically closed field with $\mathsf{char}(\mathbb{F})=0$, and $\mathsf{G}$ a finite group. Let $\mathfrak{W}$ be the $\mathsf{G}$-graded variety generated by subset $S=\left\{\left[y^{(e)}, z^{(g)}\right]:g\in\mathsf{G}\right\}$ of $\mathbb{F}\langle X^\mathsf{G}\rangle$. If $\mathsf{G}=\mathbb{Z}_p$, with $\mathsf{gcd}(p,2)=1$, then any $\mathsf{G}$-graded algebra of $\mathfrak{W}$ satisfies the (ordinary) polynomial identity $[x_1,x_2,x_3][x_4,x_5,x_6]\cdots[x_{n-2},x_{n-1},x_n]\in \mathbb{F}\langle X\rangle$ for some $n\in\mathbb{N}$.
\end{proposition}
\begin{proof}
From the observations above, there exists a finite dimensional $\mathsf{G}\times\mathbb{Z}_2$-graded algebra $\mathfrak{A}$ which satisfies $\mathfrak{W}=\mathsf{var}^\mathsf{G}\left(\mathsf{E}^\mathsf{G}(\mathfrak{A}) \right)$, and hence, we can write $\mathsf{E}^\mathsf{G}(\mathfrak{A})=\mathsf{E}^\mathsf{G}(\mathfrak{A}_1)\times\cdots\times \mathsf{E}^\mathsf{G}(\mathfrak{A}_p)\times\mathsf{E}^\mathsf{G}(\mathsf{J}_{00})$, where each $\mathfrak{A}_s=M_{k_s}(\mathbb{F}^{\sigma_s}[\mathsf{H}_s])\oplus\mathsf{J}_s$ and $\mathsf{J}(\mathfrak{A})=\mathsf{J}_1\times\cdots\times\mathsf{J}_p\times\mathsf{J}_{00}$. Since $\mathsf{E}^\mathsf{G}(\mathfrak{A})_e$ is central in $\mathsf{E}^\mathsf{G}(\mathfrak{A})$, it is easy to see that $k_1=\cdots=k_p=1$, and so $\mathfrak{A}_s=\mathbb{F}^{\sigma_s}[\mathsf{H}_s]\oplus\mathsf{J}_s$ for all $s=1,\dots,p$.

On other the hand, as $\mathsf{G}=\mathbb{Z}_p$, where $p$ and $2$ are coprime, we have that $\mathsf{G}\times\mathbb{Z}_2$ is isomorphic to $\mathbb{Z}_{2p}$. Hence, for any subgroup $\mathsf{H}$ of $\mathsf{G}\times\mathbb{Z}_2$, it follows that any $2$-cocycle $\sigma\in\mathsf{Z}^2(\mathsf{H},\mathbb{F}^*)$ is symmetric, i.e. $\sigma(g,h)=\sigma(h,g)$ for any $g,h\in\mathsf{H}$ (for a proof of this fact, see \cite{MarduaThesis}, Corollary 1.2.8 , p.28). Consequently, the algebras $\mathbb{F}^{\sigma_s}[\mathsf{H_s}]$'s are commutative. Observe that $\mathsf{E}^\mathsf{G}(\mathfrak{A}_r)\mathsf{E}^\mathsf{G}(\mathfrak{A}_s)=\mathsf{E}^\mathsf{G}(\mathfrak{A}_r)\mathsf{E}^\mathsf{G}(\mathsf{J}_{00}) =\mathsf{E}^\mathsf{G}(\mathsf{J}_{00})\mathsf{E}^\mathsf{G}(\mathfrak{A}_r)=\{0\}$ for all $r,s=1,\dots,p$, $r\neq s$. Now, for any $i=0,1$ and $g,h,t\in\mathsf{G}$, we have that
\begin{equation}\nonumber
		\begin{split}
\left[(\mathfrak{A}_s)_{(g,0)}\otimes_\mathbb{F}\mathsf{E}_0,(\mathfrak{A}_s)_{(h,i)}\otimes_\mathbb{F}\mathsf{E}_i\right] \subseteq &
 \left[(\mathbb{F}^{\sigma_s}[\mathsf{H_s}])_{(g,0)},(\mathsf{J}_s)_{(h,i)}\right]\otimes_\mathbb{F}\mathsf{E}_i +
 \left[(\mathsf{J}_s)_{(g,0)},(\mathbb{F}^{\sigma_s}[\mathsf{H_s}])_{(h,i)}\right]\otimes_\mathbb{F}\mathsf{E}_i \\
  & + \left[(\mathsf{J}_s)_{(g,0)},(\mathsf{J}_s)_{(h,i)}\right]\otimes_\mathbb{F}\mathsf{E}_i \\
 \subseteq& (\mathsf{J}_s)_{(gh,i)} \otimes_\mathbb{F} \mathsf{E}_i \subseteq \mathsf{E}^\mathsf{G}(\mathsf{J}_s) \ ,
		\end{split}
\end{equation}
and since $\left[(\mathfrak{A}_s)_{(g,1)}\otimes_\mathbb{F}\mathsf{E}_1,(\mathfrak{A}_s)_{(h,1)}\otimes_\mathbb{F}\mathsf{E}_1\right] \subseteq (\mathfrak{A}_s)_{(gh,0)}\otimes_\mathbb{F}\mathsf{E}_0$, it follows that 
\begin{equation}\nonumber
		\begin{split}
\left[\left[(\mathfrak{A}_s)_{(g,1)}\otimes_\mathbb{F}\mathsf{E}_1,(\mathfrak{A}_s)_{(h,1)}\otimes_\mathbb{F}\mathsf{E}_1\right] , (\mathfrak{A}_s)_{(t,i)}\otimes_\mathbb{F}\mathsf{E}_i\right] \subseteq \left[(\mathfrak{A}_s)_{(gh,0)}\otimes_\mathbb{F}\mathsf{E}_0,(\mathfrak{A}_s)_{(t,i)}\otimes_\mathbb{F}\mathsf{E}_i\right] \subseteq  \mathsf{E}^\mathsf{G}(\mathsf{J}_s)
		\end{split} \ .
\end{equation}
From this, as $\mathsf{E}^\mathsf{G}(\mathfrak{A}_s)=\bigoplus_{g\in\mathsf{G}}\left((\mathfrak{A}_s)_{(g,0)}\otimes_\mathbb{F}\mathsf{E}_0 + (\mathfrak{A}_s)_{(g,1)}\otimes_\mathbb{F} \mathsf{E}_1 \right)$, we conclude that $\left[\mathsf{E}^\mathsf{G}(\mathfrak{A}_s),\mathsf{E}^\mathsf{G}(\mathfrak{A}_s)\right]\subseteq \mathsf{E}^\mathsf{G}(\mathsf{J}_s)$, for all $s=1,\dots,p$. Therefore, since  $\mathsf{E}^\mathsf{G}(\mathsf{J}_{00})$ and $\mathsf{E}^\mathsf{G}(\mathsf{J}_s)$ are subalgebras of $\mathsf{E}^\mathsf{G}(\mathsf{J})$ for all $s=1,\dots,p$, and $\mathsf{E}^\mathsf{G}(\mathsf{J})$ are nilpotent, the result follows.
\end{proof}

%
\section*{ACKNOWLEDGMENTS}
The work was carried out when the first author was a Professor at the University of Bras\'ilia, between 2022 and 2023. 
The authors are thankful to the article's referees for the careful reading and very useful stylistic recommendations.


\bibliographystyle{amsplain}


\end{document}